\documentclass[reqno,11pt,a4paper]{amsart}
\usepackage[a4paper,left=30mm,right=30mm,top=30mm,bottom=30mm,marginpar=20mm]{geometry}
\usepackage{amsmath}
\usepackage{amssymb}
\usepackage{amsthm}
\usepackage{mathrsfs}
\usepackage{setspace}
\usepackage{xcolor}
\usepackage{pict2e}
\usepackage{graphicx}
\usepackage[normalem]{ulem}

\usepackage{subcaption}
\usepackage{caption}

\usepackage{tikz}

\usepackage{cite}
\usepackage[colorlinks=true, linkcolor=blue, urlcolor=blue, citecolor=green]{hyperref}

{\theoremstyle{remark}\newtheorem*{remark}{Remark}}

\newcommand{\wkto}{\mathrel{\smash{{\overset{\raisebox{-0.2ex}{$\scriptstyle\ast$}}{\rightharpoonup}}}}}

\newcommand{\mvert}{\mathrel{
		\begingroup
		\begin{picture}(7,7)
			\roundcap\roundjoin
			\linethickness{0.6pt}
			\polyline(0,7)(0,0)(7,0)
		\end{picture}
		\endgroup
}}

\newcommand{\Brm}{\mathrm{B}}
\newcommand{\Crm}{\mathrm{C}}

\newcommand{\Irm}{\mathrm{I}}

\newcommand{\Lrm}{\mathrm{L}}

\newcommand{\Nrm}{\mathrm{N}}

\newcommand{\Prm}{\mathrm{P}}

\newcommand{\Srm}{\mathrm{S}}

\newcommand{\Wrm}{\mathrm{W}}

\newcommand{\Zrm}{\mathrm{Z}}

\newcommand{\Dcal}{\mathcal{D}}

\newcommand{\Mcal}{\mathcal{M}}

\newcommand{\Fbf}{\mathbf{F}}

\newcommand{\Mbf}{\mathbf{M}}
\newcommand{\Nbf}{\mathbf{N}}

\DeclareMathOperator*{\esssup}{ess\,sup}

\DeclareMathOperator{\interior}{int}

\DeclareMathOperator{\dist}{dist}

\DeclareMathOperator{\supp}{supp}

\DeclareMathOperator{\Wedge}{{\textstyle\bigwedge}}

\newcommand{\norm}[1]{\|#1\|}

\newcommand{\abs}[1]{|#1|}

\newcommand{\tv}[1]{\norm{#1}}

\newcommand{\altnorm}[1]{{\left|\kern-0.25ex\left|\kern-0.25ex\left| #1 \right|\kern-0.25ex\right|\kern-0.25ex\right|}}

\newcommand{\dprb}[1]{\bigl\langle #1 \bigr\rangle}

\newcommand{\cl}[1]{\overline{#1}}
\newcommand{\di}{\mathrm{d}}
\newcommand{\dd}{\;\mathrm{d}}

\newcommand{\N}{\mathbb{N}}
\newcommand{\R}{\mathbb{R}}

\newcommand{\Z}{\mathbb{Z}}

\newcommand{\Var}{\mathrm{Var}}
\newcommand{\Nlip}{\Nrm^{\mathrm{Lip}}}

\newcommand{\Lip}{\mathrm{Lip}}
\newcommand{\Trajlip}{\mathrm{Traj}^\Lip}
\newcommand{\STF}[1]{\vert\vert\vert#1\vert\vert\vert}
\newcommand{\sign}{\mathrm{sign}\hspace{0.2em}}

\newcommand{\tbf}{\mathbf{t}}
\newcommand{\pbf}{\mathbf{p}}

\newtheorem{definition}{Definition}[section]
\newtheorem{proposition}[definition]{Proposition}
\newtheorem{lemma}[definition]{Lemma}
\newtheorem{theorem}[definition]{Theorem}
\newtheorem{corollary}[definition]{Corollary}
\newtheorem*{claim}{Claim}
\newtheorem*{theorem*}{Theorem}
\numberwithin{equation}{section}

\begin{document}
	
	\title[The Space-Time Connectivity Theorem for Normal Currents]{The Space-Time Connectivity Theorem for Normal Currents}
	
	\begin{abstract}
    	This work establishes a Space-Time Connectivity Theorem for normal currents. In analogy with classical results of Federer and Fleming, this result allows one to witness the weak* convergence of a uniformly bounded sequence of boundaryless normal currents with a space-time normal current that connects the elements of the sequence to their limit. The space-time setting is distinguished from the classical case in that this connecting current has a time coordinate and thus constitutes a progressive-in-time deformation of an element of the sequence to the limit.
        \\ \vspace{0.2em} 
        \noindent\textsc{Date:} \today{}
	\end{abstract}
	
	\author{Paolo Bonicatto}
	\address{Universit\`a di Trento, Dipartimento di Matematica, Via Sommarive 14, 38123 Trento, Italy}
	\email{paolo.bonicatto@unitn.it}
	
	\author{Filip Rindler}
	\address{Mathematics Institute, University of Warwick, Coventry, CV4 7AL, United Kingdom}
	\email{F.Rindler@warwick.ac.uk}
	
	\author{Harry Turnbull}
	\address{Mathematics Institute, University of Warwick, Coventry, CV4 7AL, United Kingdom}
	\email{harry.turnbull@warwick.ac.uk}
	
	\maketitle

    \section{Introduction}\label{s:Intro}

    The following result forms part of the classical \emph{Connectivity Theorem} of Federer and Fleming (see Section \ref{s:Prelims} for the relevant notation):

    \begin{theorem*} Let $K$ be a compact Lipschitz neighbourhood retract and let $M > 0$. Let $(T_j)_j \subset \Irm_{k}(K)$ be a sequence of boundaryless integral currents with $\Mbf(T_{j}) \leq M$ for $j \in \N$ and suppose that $T_{j} \wkto T$ for some $T \in \Irm_{k}(K)$. Then there is a sequence $(R_j)_j \subset \Irm_{k+1}(K)$ with
        \begin{align*}
            T - T_j = \partial R_j \hspace{0.4em} \text{eventually} \hspace{1em} \text{and} \hspace{1em} \Mbf(R_j) \to 0 \hspace{0.4em} \text{as $j \to \infty$}.
        \end{align*}
    \end{theorem*}

    In particular, this result corresponds to the integral version of \cite[Theorem 7.1]{Federer69book} in the case $(A,B) = (K,\varnothing)$, and we shall simply refer to it as the Integral Connectivity Theorem.
    
    The Integral Connectivity Theorem turns out to be related to a recent modelling approach in the theory of dislocation motion in crystalline materials, as introduced in~\cite{HudsonRindler22,Rindler23,Rindler25}. There, so-called \emph{slip trajectories}, that is, space-time integral $2$-currents in $\R^{1+3} \cong \R \times \R^{3}$, are employed to describe the motion of a dislocation: Given $S \in \Irm_2([0,1] \times \cl{\Omega})$ (where $\Omega \subset \R^3$ is a bounded Lipschitz domain representing a crystal specimen) with
    \begin{align*}
        \partial S = \delta_1 \times T_1 - \delta_0 \times T_0,  \hspace{1em} \text{where} \hspace{1em} T_0,T_1 \in \Irm_1(\cl{\Omega}),
    \end{align*}
    one interprets $S$ as ``transporting'' the dislocations represented by $T_0$ into $T_1$. Intermediate states can be recovered via slicing with respect to the temporal projection $\tbf(t,x):=t$ and then projecting onto the spatial component, namely
    \[
      T_t := \pbf_*S|_t \in \Irm_1(\cl{\Omega})  \qquad\text{for $t \in (0,1)$,}
    \]
    whenever this slice is defined. Here, $\pbf(t,x) := x$ denotes the spatial projection and $\pbf v := \pbf(v)$ for any vector $v \in \R^{1+3}$. Furthermore, the \textit{slip surface} corresponding to $S$, given by $R := \pbf_{*}S$, satisfies $\partial R = T_{1}-T_{0}$ and describes the region that the dislocation traverses over the course of its motion, taking into account multiplicities and possible cancellations.
    
    To see the aforementioned connection, take a sequence $(T_{j})_j \subset \Irm_{1}(\overline{\Omega})$ of boundaryless integral $1$-currents with uniformly bounded mass and suppose, as in the Integral Connectivity Theorem, that $T_{j} \wkto T$ for some $T \in \Irm_{1}(\overline{\Omega})$. The Integral Connectivity Theorem then furnishes us with slip surfaces $R_{j} \in \Irm_{2}(\overline{\Omega})$ which satisfy $\partial R_{j} = T-T_{j}$ eventually (that is, there is $J \in \N$ such that $\partial R_{j} = T - T_{j}$ for $j \geq J$) and $\Mbf(R_{j}) \to 0$. However, it is not immediately clear how to construct slip trajectories $S_{j} \in \Irm_{2}([0,1] \times \overline\Omega)$ satisfying $\pbf_{*}S_{j} = R_{j}$. In fact, the $S_{j}$ contain more information than the $R_{j}$: As space-time currents they have a time coordinate, and this is crucial for the description of dislocation dynamics (essentially, since the plastic and dissipative effects resulting from dislocation motion may depend on the path taken by the dislocations). 
    
    Nevertheless, the existence of the appropriate slip trajectories $S_{j}$ is provided by the \emph{Equivalence Theorem}~\cite[Theorem 5.1]{Rindler23}, so named because it shows the equivalence of two different notions of distance between integral currents; in this work, we shall refer to this result as the \emph{Space-Time Connectivity Theorem} for integral currents. By means of this result, one can replace the $R_j$ in the Integral Connectivity Theorem above by \emph{space-time trajectories}
    \[
        S_j \in \Irm_{1+k}([0,1] \times K) \hspace{1em} \text{with} \hspace{1em}
        \partial S_j = \delta_1 \times T - \delta_0 \times T_j,
    \]
    and the vanishing mass conclusion with
    \[
        \Var(S_j) := \int \abs{\pbf \vec{S}_j} \hspace{0.2em} \mathrm{d}\tv{S_{j}} \to 0  \hspace{1em} \text{as $j \to \infty$}.
    \]
    This \emph{variation} of $S_j$ measures the area of the traversed surface (taking into account absolute multiplicity) and in dislocation theory corresponds to a simplified version of the dissipation. Furthermore, by appealing to a more refined argument, one can also get an $\Lrm^\infty$-estimate for the slices, namely
    \[
      \limsup_{j \to \infty} \Vert S_j \Vert_{\Lrm^\infty} := \limsup_{j \to \infty}\esssup_{t \in [0,1]} \Mbf(S_j|_t) \leq C \cdot \limsup_{j \to \infty}\Mbf(T_{j}),
    \]
    where $C > 0$ depends only on $k$, $d$, and $K$.
    The Space-Time Connectivity Theorem is a key ingredient in the existence theorem~\cite[Theorem 5.1]{Rindler25} for the fully nonlinear system of discrete dislocation motion and its plastic effects. In particular, the uniform $\Lrm^\infty$-estimate is crucial as it ensures that intermediate dislocations have finite length, without which these trajectories would be physically inadmissible.

    However, in real crystalline materials, the amount of dislocations is usually very large; for example, the total length of all dislocation lines in well-annealed metals is on the order of $10^{7}$ to $10^{8}$ cm per cm$^{3}$ (see~\cite{HullBacon11book}). To model the motion of many dislocations effectively, a possible approach is to describe them as a \emph{field}. In fact, one can model a \emph{dislocation field} using boundaryless \emph{normal} 1-currents, but existence results for elastoplastic evolutions of these fields require one to establish a suitable Space-Time Connectivity Theorem for \emph{normal} currents, which is the main goal of the present paper. This goal is realised in Theorem~\ref{tm:Equivalence}.
    
    We shall now discuss a few reasons why this result is more subtle than the corresponding result for integral currents described above. First notice that the statement of the Integral Connectivity Theorem above would fail if we were to directly replace integral currents with normal currents without further adjustments. To see this, let $K \subset \R^{2}$ be an annulus and set $T_{j} := j^{-1}R$ for $j \in \N$, where $R \in \Irm_{1}(K)$ is a boundaryless integral $1$-current that loops around the hole of $K$. Then $T_{j} \wkto 0$ but there is no $R_{j} \in \Nrm_{2}(K)$ with $\partial R_{j} = T_{j}$. However, if one ensures that this behaviour does not occur by assuming that $T_{j}$ eventually belongs to the same homology class as $T$, then one can observe a positive result:

    \begin{theorem*}[Normal Connectivity Theorem] Let $K$ be a compact Lipschitz neighbourhood retract and let $M > 0$. Let $(T_{j})_{j} \subset \Nrm_{k}(K)$ be a sequence of boundaryless normal currents with $\Mbf(T_{j}) \leq M$ for $j \in \N$, $T_{j} \wkto T$ for some $T \in \Nrm_{k}(K)$, and, for $j$ sufficiently large,
    \begin{align}\label{NullHom}
        T - T_{j} \in \Brm_{k}(K), \hspace{1em} \text{i.e. $T - T_{j} = \partial Z_{j}$ for some $Z_{j} \in \Nrm_{k+1}(K)$}.
    \end{align}
    Then there is a sequence $(R_j)_j \subset \Nrm_{k+1}(K)$ with
        \begin{align*}
            T - T_j = \partial R_j \hspace{0.4em} \text{eventually} \hspace{1em} \text{and} \hspace{1em} \Mbf(R_j) \to 0 \hspace{0.4em} \text{as $j \to \infty$}.
        \end{align*}
    \end{theorem*}

    Notice that, under the assumptions of the Normal Connectivity Theorem, the Federer-Fleming Connectivity Theorem only demonstrates that $\Fbf_{K}(T-T_{j}) \to 0$, that is, there are $Y_{j} \in \Nrm_{k}(K)$ and $X_{j} \in \Nrm_{k+1}(K)$ with
    \begin{align*}
        T - T_{j} &= \partial X_{j} + Y_{j}, & \Mbf(X_{j}) + \Mbf(Y_{j}) \to 0.
    \end{align*}
    The additional assumption \eqref{NullHom} allows us to write $Y_{j}$ as the boundary of a normal $(k+1)$-current for large enough $j \in \N$ and so the result follows provided we can control the mass of this filling. To this end, one can use an estimate first obtained in \cite[3.2]{Federer74} which can be interpreted as a ``linear isoperimetric inequality'' for normal boundaries (see Section \ref{s:FlatEquiv}). In the integral case, the isoperimetric inequality~\cite[Theorem 6.1]{FedererFleming60} ensures that $T - T_{j} \in \Brm_{k}(K)$ whenever $T_{j}$ and $T$ are close enough in the (integral) flat norm, and so the assumption \eqref{NullHom} is superfluous there. 
    
    Our main result, in an abbreviated form of the full statement in Theorem \ref{tm:Equivalence}, is:
    
    \begin{theorem}[Space-Time Connectivity Theorem]\label{tm:AbbrvSTConnect} Let $K$ be a compact Lipschitz neighbourhood retract and let $M > 0$. Let $(T_{j})_{j} \subset \Nrm_{k}(K)$ be a sequence of boundaryless normal currents with $\Mbf(T_{j}) \leq M$ for $j \in \N$, $T_{j} \wkto T$ for some $T \in \Nrm_{k}(K)$, and $T-T_{j} \in \Brm_{k}(K)$ for $j$ sufficiently large. Then, for all $j \in \N$ along a (not relabelled) subsequence, there is $S_{j} \in \Nrm_{1+k}([0,1] \times K)$ such that
	\begin{align}\label{AbbrvSTConnect1}
		\partial S_{j} &= \delta_{1} \times T - \delta_{0} \times T_{j}, & \Var(S_{j}) &\to 0, 
	\end{align}
	and
	\begin{align}\label{AbbrvSTConnect2}
		\limsup_{j \to \infty} \Vert S_j \Vert_{\Lrm^\infty} \leq C \cdot \limsup_{j \to \infty} \Mbf(T_{j}),
	\end{align}
	where $C > 0$ depends only on $k$, $d$, and $K$.
    \end{theorem}

    To obtain the first conclusion \eqref{AbbrvSTConnect1}, we use the two results referred to in the previous paragraph, but we also need a way to ``stretch out'' a filling that witnesses $T - T_{j} \in \Brm_{k}(K)$ in such a way that the variation is well-behaved. To do this, we can use an argument similar to that of~\cite[Theorem 5.4]{Rindler23}. This result is called the \emph{Equality Theorem}, but the arguments of~\cite{Rindler23} in fact only apply for $K = \overline{\Omega}$ with $\Omega$ a $\Crm^1$-domain, and we give an example of a Lipschitz domain where the argument fails. Nevertheless, the statement of the theorem is true: In Theorem \ref{tm:Equality} we will also show the Equality Theorem for $K = \overline{\Omega}$ where $\Omega \subset \mathbb{R}^{d}$ is a $\Crm^0$ domain (see Section \ref{s:GoodDirections}). Central to the proof of this result is the notion of ``good directions'' introduced in~\cite{BallZarnescu17}. 
    
    A major disadvantage of the aforementioned Equality Theorem is that it does not control the masses of the slices of the resulting space-time trajectory. To obtain the $\Lrm^{\infty}$ estimate \eqref{AbbrvSTConnect2}, we must therefore adopt a more careful approach, which is given in Section \ref{s:STConnect}.
    
    As in the integral case, we shall frequently make use of a deformation result adapted to the space-time setting; in the present paper, we call this result the \emph{Dynamical Deformation Theorem} (see Theorem \ref{tm:STDeform}). This result is one of many that require variation estimates which are relatively straightforward for a space-time integral current $S \in \Irm_{k+1}([0,1] \times \R^{d})$ (where one can appeal to the area formula since $\tv{S} \ll \mathscr{H}^{k+1}$), but need some care in the normal case. In particular, we need to appeal to the recent results established in~\cite{AlbertiMarchese16} to obtain precise formulae (see Section \ref{s:VarEsts}). 
	
	\section{Preliminaries}\label{s:Prelims}
	
	This section collects some notation and known results, together with some simplified proofs in special cases.
	
	\subsection{Currents} Fix $d,m \in \N$ and let $k,l \in \N_{0}$. We refer to~\cite[Chapter 4]{Federer69book} for a comprehensive treatment of currents and here recall some facts that we shall use frequently in the sequel. The notation used follows that of~\cite{Rindler23} but, since we shall make use of both the Euclidean and mass/comass norms on $\Wedge_k \R^{d}$ and $\Wedge^k \R^{d}$, we shall use $\vert \cdot \vert$ and $\Vert \cdot \Vert$ to denote the former and latter pair respectively, following~\cite{Federer69book}. We stress that any terms involving the boundary or slices of a $0$-current in what follows should be ignored.
	
	A $k$-current $T \in \Dcal_{k}(\R^d)$ is said to be \textbf{normal} if $\Nbf(T) := \Mbf(T) + \Mbf(\partial T) < \infty$ and $\supp T$ is compact. We shall use $\Nrm_{k}(\R^d)$ to denote the space of normal $k$-currents on $\R^d$.
	
	Let $K \subset \R^d$ be compact. We shall use $\Nrm_{k}(K)$ to denote the space of normal $k$-currents with support in $K$. Furthermore, we define the spaces of (normal) \textbf{cycles} and \textbf{boundaries} as follows:
	\begin{align*}
		\Zrm_{k}(K) &:= \{T \in \Nrm_{k}(K): \partial T = 0\} &   \Brm_{k}(K) &:= \{\partial S: S \in \Nrm_{k+1}(K)\}.
	\end{align*}
	Notice that $\Brm_{k}(K) \subset \Zrm_{k}(K)$.
	
	Fix $T \in \Nrm_{k}(K)$. We define the \textbf{flat norm} of $T$ as follows:
	\begin{align}\label{FlatDefn}
		\Fbf_{K}(T) := \inf\{\Mbf(T - \partial S) + \Mbf(S): S \in \Nrm_{k+1}(K)\}.
	\end{align}
	Notice that $\Fbf_{K}(T) \leq \Mbf(T)$ and that $\Fbf_{K}$ is indeed a norm on $\Nrm_{k}(K)$. It turns out that the infimum in~\eqref{FlatDefn} is attained. Note that, in~\cite[4.1.12]{Federer69book}, the infimum is taken over $S \in \Dcal_{k+1}(\R^d)$ with $\supp S \subset K$ but, since $\Mbf(T) < \infty$, the minimiser is normal. We also define the \textbf{homogeneous flat norm} of $T$ as follows:
	\begin{align*}
		\Fbf^\circ_{K}(T) := \inf\{\Mbf(S): S \in \Nrm_{k+1}(K), \hspace{0.2em} \partial S = T\}.
	\end{align*}
	In particular, $\Fbf^\circ_{K}(T) = \infty$ if $T \not\in \Brm_{k}(K)$, but $\Fbf^\circ_{K}$ is a norm on $\Brm_{k}(K)$. Note that this infimum is also attained. 

    Let $f: \R^d \to \R^{m}$ be smooth and suppose that $f\vert_{\supp T}$ is proper (recall that a map is proper if preimages of compact sets are compact). The \textbf{pushforward} of $T$ by $f$ is given by $f_{*}T(\omega) := T(\theta \cdot f^{*}\omega)$ for $\omega \in \Dcal^{k}(\R^{m})$, where $f^{*}\omega$ is the pullback of $\omega$ by $f$ and $\theta: \R^d \to \R$ is a smooth function with $\theta = 1$ on a neighbourhood of $(\supp T) \cap f^{-1}(\supp \omega)$. One can check that this definition is independent of $\theta$ and that $f_{*}T \in \Nrm_{k}(\R^{m})$. If $f$ is only Lipschitz, then one can use the estimate in~\cite[4.1.13]{Federer69book} to make sense of $f_{*}T \in \Nrm_{k}(\R^{m})$ in such a way that the homotopy formula (see~\cite[4.1.9]{Federer69book}) still holds. Furthermore, we have
    \begin{align}\label{PushFwdAC}
		\Vert f_{*}T \Vert \ll f_{\#}\Vert T \Vert,
	\end{align}
	where $f_{\#}\Vert T \Vert$ denotes the measure-theoretic pushforward of $\Vert T \Vert$ by $f$ (see~\cite[Definition 1.70]{AmbrosioFuscoPallara00book}). We remark that we follow the convention of~\cite{AmbrosioFuscoPallara00book} here and so we use $f_{*}T$ (as opposed to $f_{\#}T$) to denote the previously defined pushforward of $T$ by $f$ to avoid the overloading of notation. In the following, $f$ may only be defined on $\supp T$. In this case, we define $f_{*}T := \tilde{f}_{*}T$ where $\tilde{f}: \R^d \to \R^{m}$ is a Lipschitz extension of $f$. This procedure is well-defined by~\cite[4.1.15]{Federer69book}.
	
	If $f$ is smooth, we can use the Radon--Nikodym decomposition $T = \vec{T} \, \Vert T \Vert$ and the definition of the pullback of a $k$-form to deduce that, for $\omega \in \Dcal^{k}(\mathbb{R}^{m})$,
	\begin{align}\label{SmoothPushFwdFormula}
		f_{*}T(\omega) = \int \langle Df(x)\vec{T}(x),\omega(f(x)) \rangle \dd \Vert T \Vert(x).
	\end{align}
	If $f$ is only Lipschitz, we cannot expect~\eqref{SmoothPushFwdFormula} to hold as $Df(x)$ may not exist for $\Vert T \Vert$-a.e.\ $x \in \R^d$. However, it was shown in~\cite[Theorem 1.1]{AlbertiMarchese16} that there is a vector bundle associated to the measure $\Vert T \Vert$, that is, an assignment of a subspace $V(\Vert T \Vert,x)$ of $\R^d$ for each $x \in \R^d$, such that for any Lipschitz function $f: \R^d \to \R^{m}$ there is a linear map $D^{T}f(x): V(\Vert T \Vert,x) \to \R^{m}$ for $\Vert T \Vert$-a.e.\ $x \in \R^d$ with
	\begin{align}\label{TangDiffDefn}
		D^{T}f(x)v = \lim_{h \to 0} \frac{f(x+hv)-f(x)}{h}.
	\end{align}
    for $v \in V(\Vert T \Vert,x)$. This bundle is called the \textbf{decomposability bundle} of $\Vert T \Vert$. Furthermore, it was shown in~\cite[Proposition 5.6]{AlbertiMarchese16} that
	\begin{align}\label{PolarInBundle}
		\mathrm{span}\hspace{0.2em}\vec{T}(x) \subset V(\Vert T \Vert,x) \hspace{2em} \text{for $\Vert T \Vert$-a.e.\ $x \in \R^d$}.
	\end{align}
	Again, let $f: \R^d \to \R^{m}$ be Lipschitz. It turns out that the only obstruction to a generalisation of~\eqref{SmoothPushFwdFormula} is the aforementioned lack of $\Vert T \Vert$-a.e.\ existence of the differential of $f$; it was shown in~\cite[Proposition 5.19]{AlbertiMarchese16} that, for $\omega \in \Dcal^{k}(\R^{m})$, we have
	\begin{align}\label{LipPushFwdFormula}
		f_{*}T(\omega) = \int \langle D^{T}f(x)\vec{T}(x),\omega(f(x)) \rangle \dd \Vert T \Vert(x).
	\end{align}
    In the following, we shall occasionally abuse notation and write $Df(x) := D^{T}f(x)$ when the choice of $T \in \Nrm_{k}(\mathbb{R}^{d})$ is clear. We stress that, in this case, it may happen that the full gradient $Df(x)$ does not exist at $x$, yet the derivative in the direction of the current $T$ does exist. We adopt this lighter notation in order to reduce clutter.
    
    Let $(T_{j})_{j} \subset \Dcal_{k}(\R^{d})$ and take $T \in \Dcal_{k}(\R^{d})$. We say that $T_{j} \to T$ \textbf{strictly} if $T_{j} \wkto T$ and $\Mbf(T_{j}) \to \Mbf(T)$.
    
	We shall use $\Irm_{k}(\R^d)$ to denote the space of integral $k$-currents in $\R^d$. Let $f: \R^d \to \R^{m}$ be Lipschitz and suppose that $f\vert_{\supp T}$ is proper. It turns out that $f_{*}T \in \Irm_{k}(\R^{m})$ whenever $T \in \Irm_{k}(\R^d)$. Given a compact set $K \subset \R^d$, we define $\Irm_{k}(K) := \Nrm_{k}(K) \cap \Irm_{k}(\R^d)$.

    Now fix $v_{0},v_{1},\dotsc,v_{k} \in \R^d$. If $v_{0},\dotsc,v_{k}$ are linearly independent, we define the $k$-dimensional \textbf{oriented simplex} $[\![v_{0},v_{1},\dotsc,v_{k}]\!]$ as
    \begin{align}\label{SimplicesAreRect}
		[\![v_{0},v_{1},\dotsc,v_{k}]\!] := \frac{\eta}{\vert \eta \vert} \mathscr{H}^{k} \mvert C,
	\end{align}
    where $C \subset \R^d$ is the convex hull of $\{v_{0},v_{1},\dotsc,v_{k}\}$ and $\eta := (v_{1}-v_{0}) \wedge \cdots \wedge (v_{k}-v_{k-1})$. If $v_{0},\dotsc,v_{k}$ are linearly dependent, then we set $[\![v_{0},v_{1},\dotsc,v_{k}]\!] := 0$. Notice that this agrees with the notation of \cite[4.1.11]{Federer69book} by \cite[4.1.27]{Federer69book}. The space of \textbf{polyhedral chains} $\Prm_{k}(\R^{d})$ is the vector space generated by all $k$-dimensional oriented simplices in $\R^{d}$. It follows from \eqref{SimplicesAreRect} that 
	\begin{align}\label{PolyVsIntegral}
		[\![v_0,v_1,\dotsc,v_{k}]\!] \in \Irm_{k}(\R^d),
	\end{align}
	and thus integral combinations of simplices are both integral currents and polyhedral chains. Given a compact set $K \subset \R^{d}$, we define $\Prm_{k}(K) := \Nrm_{k}(K) \cap \Prm_{k}(\R^{d})$. In \cite[4.1.11]{Federer69book}, it is shown that
	\begin{align}\label{PolyBdry}
		\partial [\![v_{0},v_{1},\dotsc,v_{k}]\!] = \sum_{i=0}^{k}(-1)^{i}[\![v_{0},\dotsc,\hat{v}_{i},\dotsc,v_{k}]\!],
	\end{align}
	where $\hat{v}_{i}$ indicates that $v_{i}$ has been omitted from the tuple. It follows that, if $P \in \Prm_{k}(\R^{d})$, then $\partial P \in \Prm_{k-1}(\mathbb{R}^{d})$. Similarly, if $A: \R^d \to \R^{m}$ is affine, then $A_{*}P \in \Prm_{k}(\R^{m})$; in \cite[4.1.11]{Federer69book}, it is shown that
	\begin{align}\label{PolyPushFwd}
		A_{*}[\![v_{0},v_{1},\dotsc,v_{k}]\!] = [\![A(v_{0}),A(v_{1}),\dotsc,A(v_{k})]\!].
	\end{align}
    It is also shown in \cite[4.1.11]{Federer69book} that
    \begin{align}\label{SimpPermute}
        [\![v_{\sigma(0)},\dotsc,v_{\sigma(k)}]\!] = \mathrm{index}(\sigma)[\![v_{0},\dotsc,v_{k}]\!]
    \end{align}
    for every permutation $\sigma$ of $\{0,\dotsc,k\}$. We shall also need the formula for the simplicial decomposition of an oriented simplicial prism (again, see \cite[4.1.11]{Federer69book}): Given $a,b \in \R$, we have
	\begin{align}\label{SimpDecomp}
		[\![a,b]\!] \times [\![v_{0},v_{1},\dotsc,v_{k}]\!] = \sum_{i=0}^{k}(-1)^{i}[\![(a,v_{0}),\dotsc,(a,v_{i}),(b,v_{i}),\dotsc,(b,v_{k})]\!].
	\end{align}
	If $P_{1} \in \Prm_{k}(\R^d)$ and $P_{2} \in \Prm_{l}(\R^{m})$ (here we assume that $l \leq m$), then it follows from~\eqref{SimpDecomp} that $P_{1} \times P_{2} \in \Prm_{k+l}(\R^{d+m})$.
    Now take $v \in \R^{d}$. It turns out (see \cite[4.1.11]{Federer69book}) that
    \begin{align}\label{FedererSimplicesJoin}
        [\![v,v_{0},\dotsc,v_{k}]\!] = G_{*}([\![0,1]\!] \times [\![v_{0},\dotsc,v_{k}]\!]),
    \end{align}
    where $G: [0,1] \times \R^{d} \to \R^{d}$ is given by $G(t,x) = (1-t)v+tx$ for $(t,x) \in [0,1] \times \R^{d}$. We claim that
    \begin{align}\label{SimplicesJoin}
        [\![v_{0},\dotsc,v_{k},v]\!] = (-1)^{k} H_{*}([\![0,1]\!] \times [\![v_{0},\dotsc,v_{k}]\!]),
    \end{align}
    where $H: [0,1] \times \R^{d} \to \R^{d}$ is given by $H(t,x) = (1-t)x+tv$ for $(t,x) \in [0,1] \times \R^{d}$. To see this, notice that $H = G \circ a$ where $a: [0,1] \times \R^{d} \to [0,1] \times \R^{d}$ is given by $a(t,x)=(1-t,x)$ for $(t,x) \in \R^{d}$. We can therefore use \eqref{SimpDecomp}, \eqref{PolyPushFwd} with $A = a$, \eqref{FedererSimplicesJoin}, and \eqref{SimpPermute} to deduce that
    \begin{align*}
        H_{*}([\![0,1]\!] \times [\![v_{0},\dotsc,v_{k}]\!]) 
        &= G_{*}([\![1,0]\!] \times [\![v_{0},\dotsc,v_{k}]\!]) \\
        &= - G_{*}([\![0,1]\!] \times [\![v_{0},\dotsc,v_{k}]\!]) \\
        &= - [\![v,v_{0},\dotsc,v_{k}]\!] \\
        &= (-1)^{k}[\![v_{0},\dots,v_{k},v]\!].
    \end{align*}
	
	Let $T \in \Nrm_{k}(\R^d)$ and let $f: \R^d \to \R$ be Lipschitz. We shall need to consider the \textbf{slice} $\langle T,f,t \rangle$ of $T$ by $f$ at $t \in \R$. In the vast majority of cases, we shall be slicing space-time normal currents by the projection map $\tbf: \R^{1+d} \to \R$, given by $\tbf(t,x) := t$ for $(t,x) \in \R^{1+d}$, so we thoroughly treat this case in Section~\ref{s:Slicing}.

	\subsection{Polyhedral Approximation of Normal Currents}\label{s:Approx} In this subsection, we collect some of the classical approximation results for normal currents. In the following, $\Lambda(k,d)$ is the set of $(i_{1},\cdots,i_{k}) \in \{1,\dotsc,d\}^{k}$ with $i_{1} < \cdots < i_{k}$; given $I \in \Lambda(k,d)$, we set $e_{I} := e_{i_{1}} \wedge \cdots \wedge e_{i_{k}}$.
		
	Let $T \in \Nrm_{k}(\R^d)$ and let $f: \R^d \to \R^d$ be Lipschitz. If $H$ is a Lipschitz homotopy from $\mathrm{id}_{\R^d}$ to $f$, we can view the pushforward $f_{*}T$ as a ``deformation'' of $T$ using the homotopy formula:
	\begin{align}\label{NaiveDeformation}
        f_{*}T = T + H_{*}([\![0,1]\!]\times\partial T) + \partial H_{*}([\![0,1]\!] \times T).
	\end{align}
	If we repeatedly apply (a refined version of) this procedure with well-chosen maps, we can deform $T$ into a polyhedral chain $P \in \Prm_{k}(\R^d)$ with support on a cubical `skeleton'; this is the content of the deformation theorem. Here, we shall state and give an abbreviated proof of the theorem in the case that $\partial T = 0$ so that we can compare it with the proof of the Dynamical Deformation Theorem in Section~\ref{s:dynamical}.  Before stating the theorem, we give some definitions, following~\cite[4.2.5]{Federer69book}. To start, we define
	\begin{align*}
			\Z^{d}_{j} := \{z \in \Z^{d}: \#\{i \in \{1,\dotsc,d\}: z_{i} \hspace{0.2em} \text{is even}\} = j\}.
		\end{align*}
	For each $z \in \Z^{d}_{k}$, we can define the following cubes:
	\begin{align*}
		\Wrm'(z) &:= \left\{x \in \R^d: 
        \begin{array}{cl}
            \vert x_{i} - z_{i} \vert < 1 & \text{if $z_{i}$ is even} \\
			x_{i} = z_{i} & \text{if $z_{i}$ is odd}
        \end{array}\right\}, \\
		\Wrm''(z) &:= \left\{x \in \R^d: 
        \begin{array}{cl}
			x_{i} = z_{i} & \text{if $z_{i}$ is even} \\
			\vert x_{i} - z_{i} \vert < 1 & \text{if $z_{i}$ is odd}
        \end{array}\right\}.
	\end{align*}
	The \textbf{dual cubical $k$-skeletons} of $\R^d$ are 
	\begin{align*}
			\Wrm'_{k} &:= \bigcup\{\Wrm'(z): z \in \Z^{d}_{j}, \hspace{0.2em} j \leq k\}, & 
			\Wrm''_{k} &:= \bigcup\{\Wrm''(z): z \in \Z^{d}_{j}, \hspace{0.2em} j \geq d-k\}.
	\end{align*}
	We illustrate the case $d=2$ in Figure~\ref{fig:skeletons}.
	\begin{figure}[htbp]
        \centering
		\begin{minipage}{0.45\textwidth}
			\centering
			\begin{tikzpicture}[scale=0.75]	
				\colorlet{w0}{red}
				\colorlet{w1}{green}
				\colorlet{w2}{blue}			
								
				\draw[help lines, step=1] (-3.5,-3.5) grid (3.5,3.5);
				\draw[->, thick] (-4,0)--(4,0) node[right]{$x_{1}$};
			    \draw[->, thick] (0,-4)--(0,4) node[above]{$x_{2}$};
								
				\foreach \x in {-3,-2,-1,0,1,2,3} {
			    \foreach \y in {-3,-2,-1,0,1,2,3} {
				    \begin{scope}[shift={(\x,\y)}]
				        \ifodd \x
                            \ifodd \y
                                \fill[w0] (0,0) circle (2pt);
						      \else
                                \draw[w1, thick] (0,-0.9) -- (0,0.9); 
						      \fi
					   \else
							\ifodd \y
								\draw[w1, thick] (-0.9,0) -- (0.9,0);
							\else
							     \fill[w2,opacity=0.1] (-0.9,-0.9) rectangle (0.9,0.9);
							\fi
						\fi
					\end{scope}
				}
                }
								
				\foreach \x in {-3,-2,-1,0,1,2,3} {
				    \ifodd \x
						\draw[w1, thick] (\x,-3.1) -- (\x,-3.9);
						\draw[w1, thick] (\x,3.1) -- (\x,3.9);
					\else 
				        \fill[w2,opacity=0.1] (\x-0.9,-3.9) rectangle (\x+0.9,-3.1);
                        \fill[w2,opacity=0.1] (\x-0.9,3.9) rectangle (\x+0.9,3.1);
					\fi
                }
								
				\foreach \y in {-3,-2,-1,0,1,2,3} {
					\ifodd \y
				        \draw[w1, thick] (-3.1,\y) -- (-3.9,\y);
						\draw[w1, thick] (3.1,\y) -- (3.9,\y);
					\else 
                        \fill[w2,opacity=0.1] (-3.9,\y-0.9) rectangle (-3.1,\y+0.9);
						\fill[w2,opacity=0.1] (3.9,\y-0.9) rectangle (3.1,\y+0.9);
					\fi
				}
								
                \fill[w2,opacity=0.1] (-3.1,-3.1) rectangle (-3.9,-3.9);
				\fill[w2,opacity=0.1] (-3.1,3.1) rectangle (-3.9,3.9);
				\fill[w2,opacity=0.1] (3.1,-3.1) rectangle (3.9,-3.9);
				\fill[w2,opacity=0.1] (3.1,3.1) rectangle (3.9,3.9);
								
                \fill[w2,opacity=0.1] (-2.25,-5.25) rectangle (-1.75,-4.75);
				\node[right] at (-1,-5) {$\Wrm'_{2} \backslash \Wrm'_{1}$};
								
				\draw[w1,thick] (-2.25,-6) -- (-1.75,-6);
				\node[right] at (-1,-6) {$\Wrm'_{1} \backslash \Wrm'_{0}$};
								
				\fill[w0] (-2,-7) circle (2pt); 
				\node[right] at (-1,-7) {$\Wrm'_{0}$};
			\end{tikzpicture}
        \end{minipage}
		\hfill
        \begin{minipage}{0.45\textwidth}
            \centering
			\begin{tikzpicture}[scale=0.75]
				\colorlet{w0}{red}
				\colorlet{w1}{green}
				\colorlet{w2}{blue}
								
				\draw[->,w1, thick] (-4,0)--(4,0) node[right,black]{$x_{1}$};
				\draw[->,w1, thick] (0,-4)--(0,4) node[above,black]{$x_{2}$};
				\draw[help lines, step=1] (-3.5,-3.5) grid (3.5,3.5);
								
				\foreach \x in {-3,-2,-1,0,1,2,3} {
				\foreach \y in {-3,-2,-1,0,1,2,3} {
					\begin{scope}[shift={(\x,\y)}]
                        \ifodd \x
                            \ifodd \y
                                \fill[w2,opacity=0.1] (-0.9,-0.9) rectangle (0.9,0.9);
							\else 
                                \draw[w1, thick] (-0.9,0) -- (0.9,0);
							\fi
						\else  	
							\ifodd \y
								\draw[w1, thick] (0,-0.9) -- (0,0.9);
							\else 
								\fill[w0] (0,0) circle (2pt);
                            \fi 
						\fi
					\end{scope}
				}
				}
								
				\fill[w2,opacity=0.1] (-2.25,-5.25) rectangle (-1.75,-4.75);
                \node[right] at (-1,-5) {$\Wrm''_{2} \backslash \Wrm''_{1}$};
								
				\draw[w1,thick] (-2.25,-6) -- (-1.75,-6);
				\node[right] at (-1,-6) {$\Wrm''_{1} \backslash \Wrm''_{0}$};
                
                \fill[w0] (-2,-7) circle (2pt); 
				\node[right] at (-1,-7) {$\Wrm''_{0}$};
			\end{tikzpicture}
        \end{minipage}
				
        \caption{The lattice $\Z^2$ and the standard dual cubical subdivisions of $\R^2$.}
		\label{fig:skeletons}
	\end{figure}	
		
	We now define the space of (normal) \textbf{skeletal} $k$-currents as follows:
	\begin{align*}
		\Srm_{k}(\R^d) := \{T \in \Nrm_{k}(\R^d): \supp T \subset \Wrm_{k}', \hspace{0.2em} \supp \partial T \subset \Wrm_{k-1}'\}.
	\end{align*}
	It turns out that $\Srm_{k}(\R^d)$ has a geometrically well-behaved basis; the following result can be found in~\cite[4.2.5]{Federer69book}:
		
	\begin{proposition}\label{pp:Skeletal} The vector space $\Srm_{k}(\R^d)$ is generated by the oriented $k$-dimensional cubes
	\begin{align*}
	   \{e_{I(z)} \, \mathscr{L}^{k}\mvert \Wrm'(z)  : z \in \Z^{d}_{k}\}
	\end{align*}
	where, for $z \in \Z^{d}_{k}$, $I(z)$ is the element of $\Lambda(k,d)$ consisting of those $i \in \{1,\dotsc,d\}$ such that $z_{i}$ is even. In particular, $\Srm_{k}(\R^d) \subset \Prm_{k}(\R^d)$.
	\end{proposition}
		
	We shall also need to work with scaled versions of skeletal currents. In particular, for $\varepsilon > 0$, we define
	\begin{align*}
        \Srm_{k,\varepsilon}(\R^d) := \{P \in \Nrm_{k}(\R^d): (\lambda_{\varepsilon}^{-1})_{*}P \in \Srm_{k}(\R^d)\}
    \end{align*}
	where $\lambda_{\varepsilon}: \R^d \to \R^d$ is the homothety given by $\lambda_{\varepsilon}(x) := \varepsilon x$ for $x \in \R^d$. We can now state the deformation theorem (in the case $\partial T = 0$):
		
	\begin{proposition}\label{pp:BoundarylessDeformation} Let $T \in \Nrm_{k}(\R^d)$  with $\partial T = 0$ and fix $\varepsilon > 0$. Then there is $P \in \Srm_{k,\varepsilon}(\R^d)$ and $R \in \Nrm_{k+1}(\R^d)$ such that $T = P + \partial R$ and the following mass estimates hold with a constant $\gamma > 0$ depending only on $k$ and $d$:
	\begin{align}\label{DeformMassEsts}
		\Mbf(P) &\leq \gamma \Mbf(T) & \Mbf(R) &\leq \gamma\varepsilon \Mbf(T).
	\end{align}
	Furthermore, we have $\supp P$, $\supp R \subset (\supp T)_{2d\varepsilon}$.
	\end{proposition}

    Here, for a set $A \subset \R^{d}$ and $\delta > 0$, we have defined $A_\delta := \{ x \in \R^{d} : \mathrm{dist}(x,A) < \delta\}$. Again, let $T \in \Nrm_{k}(\R^{d})$. The proof of the deformation theorem is the content of~\cite[4.2.9]{Federer69book}, but this simplifies in the case that $\partial T = 0$.
        
    The key to the proof is to construct the ``well-chosen maps'' mentioned at the start of this subsection. The most important of these is a retraction $\sigma = \sigma_{k}: \R^d\backslash\Wrm''_{d-k-1} \to \Wrm'_{k}$ which is defined in~\cite[4.2.6]{Federer69book}. We give an illustration of this map in the case $d=2$, $k=1$ in Figure~\ref{fig:retract}.
		
	\begin{figure}[h]
		\centering
		\begin{tikzpicture}[scale=0.75]
			\colorlet{w0}{red}
			\colorlet{w1}{green}
						
			\draw[->, thick] (-4,0) -- (4,0) node[right]{$x_{1}$};
			\draw[->, thick] (0,-4) -- (0,4) node[above]{$x_{2}$};
						
			\draw[w1, thick] (-3,-3.9) -- (-3,3.9);
			\draw[w1, thick] (3,-3.9) -- (3,3.9);
			\draw[w1, thick] (-3.9,-3) -- (3.9,-3);
			\draw[w1, thick] (-3.9,3) -- (3.9,3);
						
			\draw[gray, dashed] (0,0) -- (2,3);
			\draw[gray, dashed] (0,0) -- (-3,-2);
			\draw[gray, dashed] (0,0) -- (-1.5,3);
			\draw[gray, dashed] (0,0) -- (3,-3);
						
			\fill[blue] (1,1.5) circle (2pt) node[left]{$x$};
			\fill[blue] (2,3) circle (2pt) node[above]{$\sigma(x)$};
			\fill[blue] (-1.5,-1) circle (2pt);
			\fill[blue] (-3,-2) circle (2pt);
			\fill[blue] (-1,2) circle (2pt);
			\fill[blue] (-1.5,3) circle (2pt);
			\fill[blue] (1,-1) circle (2pt);
			\fill[blue] (3,-3) circle (2pt);
						
			\fill[w0] (0,0) circle (2pt);
		\end{tikzpicture}
				
		\caption{The retraction $\sigma: \R^{2} \backslash \Wrm''_{0} \to \Wrm'_{1}$.}
		\label{fig:retract}
	\end{figure}
		
	In~\cite[4.2.6]{Federer69book}, it is shown that
	\begin{align}\label{DeformEst1}
		\vert \sigma(x) - x \vert \leq d^{1/2} \hspace{2em} \text{for $x \in \R^d\backslash\Wrm''_{d-k-1}$}.
	\end{align}
	Notice that, although $\sigma$ is locally Lipschitz, it cannot be extended to a locally Lipschitz function on $\R^d$. Fix $\delta > 0$. If we restricted our attention to $T \in \Nrm_{k}(\R^{d})$ with $(\supp T)_{\delta} \subset \R^d \backslash \Wrm''_{d-k-1}$, then Proposition~\ref{pp:BoundarylessDeformation} would follow (with $\gamma > 0$ depending also on $\delta$) by setting $P = \sigma_{*}T$ and using the homotopy formula as in~\eqref{NaiveDeformation}. In general, however, we must adopt a more careful approach.
		
	\begin{proof}[Proof of Proposition~\ref{pp:BoundarylessDeformation}] First note that we can rescale using the homothety $\lambda_{\varepsilon}$ to reduce to the case that $\varepsilon = 1$. If $k = d$, then we can use the constancy theorem (see~\cite[4.1.7]{Federer69book}) to deduce that $T = 0$ and there is nothing to show. Otherwise, let $u = u_{k}: \R^d \to \R$ be the auxiliary map constructed in~\cite[4.2.6]{Federer69book}. It is shown in~\cite[4.2.6]{Federer69book} that $\Wrm''_{d-k-1} = u^{-1}(\{0\})$ and that
	\begin{align}\label{DeformEst2}
		\Vert D\sigma(x) \Vert \leq \frac{k+1}{u(x)} \leq \frac{d}{u(x)} \hspace{2em} \text{for $\mathscr{L}^{d}$-a.e.\ $x \in \R^d\backslash\Wrm''_{d-k-1}$}.
	\end{align}
	Now we use~\cite[4.2.7]{Federer69book} to find $a \in [-1,1]^{d}$ with 
	\begin{align}\label{DeformEst3}
		\int \frac{1}{u(x+a)^{k}} \dd \Vert T \Vert(x) \leq 2\binom{d}{k}\Mbf(T).
	\end{align}
	Let $\tau_{a}: \R^d \to \R^d$ where $\tau_{a}(x) = x+a$ for $x \in \R^d$. Set 
    \begin{align*}
        v := \frac{u \circ \tau_{a}}{d}.
    \end{align*}
    It follows from~\eqref{DeformEst3} that
	\begin{align}\label{DeformEst4}
		\Vert T \Vert(v^{-1}(\{0\})) = 0.
	\end{align}
	For $r > 0$, set 
    \begin{align*}
        V_{r} := v^{-1}((r,\infty))
    \end{align*}
    and define
	\begin{align*}
		P_{r} &:= (\sigma \circ \tau_{a})_{*}(T \mvert V_{r}), & R_{r} &:= H_{*}([\![0,1]\!] \times (T \mvert V_{r})),
	\end{align*}
	where $H: [0,1] \times \R^{d} \to \R^d$ is the affine homotopy from $\mathrm{id}_{\R^d}$ to $\sigma \circ \tau_{a}$. Notice that, for $r > 0$, $\sigma \circ \tau_{a}$ and $H$ are Lipschitz on $V_{r}$ and $[0,1] \times V_r$ respectively. Since we have~\eqref{DeformEst1},~\eqref{DeformEst2},~\eqref{DeformEst3}, and~\eqref{DeformEst4}, we can appeal to~\cite[4.2.2]{Federer69book} to deduce that the following limits exist (with respect to the \textit{strong} convergence, that is, vanishing mass of the difference):
	\begin{align*}
		P &:= \lim_{r \downarrow 0} P_{r}, & R &:= \lim_{r \downarrow 0} R_{r}.
	\end{align*}
	Furthermore, we have $T = P + \partial R$ and
	\begin{align}\label{DeformEst6}
		\Mbf(P) &\leq \int \frac{1}{v(x)^{k}} \dd \Vert T \Vert(x), & \Mbf(R) &\leq \int \vert \sigma(x+a) - x \vert \cdot \frac{1}{v(x)^{k}} \dd \Vert T \Vert(x).
	\end{align}
	In particular, $P \in \Nrm_{k}(\R^{d})$ and $R \in \Nrm_{k+1}(\R^{d})$. It follows from~\eqref{DeformEst1} that
	\begin{align}\label{DeformEst5}
		\vert \sigma(x+a) - x \vert < 2d^{1/2} \hspace{2em} \text{whenever $x+a \not\in \Wrm''_{d-k-1}$}.
	\end{align}
	Notice that $x+a \not\in \Wrm''_{d-k-1}$ if and only if $v(x) > 0$, so it follows from~\eqref{DeformEst3} that~\eqref{DeformEst5} holds for $\Vert T \Vert$-a.e.\ $x \in \R^{d}$; the mass estimates~\eqref{DeformMassEsts} then follow from~\eqref{DeformEst6},~\eqref{DeformEst3}, and~\eqref{DeformEst5}. Now fix $r > 0$ and notice that
	\begin{align}\label{DeformSptEsts}
		\supp P_{r} &\subset (\sigma \circ \tau_{a})(\supp T \cap V_{r}), & \supp R_{r} &\subset H([0,1] \times (\supp T \cap V_{r})).
	\end{align}
	Since $\supp T \cap V_{r} \subset \R^d\backslash\Wrm''_{d-k-1}$, we have $\supp P_{r} \subset (\supp T)_{2d}$ from~\eqref{DeformEst5}. Similarly, for $t \in [0,1]$ and $x \in \supp T \cap V_{r}$, we have
	\begin{align*}
		\vert (1-t)x + t\sigma(x+a) - x \vert = t\vert \sigma(x+a) - x\vert 
	\end{align*}
	and so we have $\supp R_{r} \subset (\supp T)_{2d}$ from~\eqref{DeformEst5} and~\eqref{DeformSptEsts}. Since $r > 0$ was arbitrary, we deduce that $\supp P$, $\supp R \subset (\supp T)_{2d}$. Finally, we have $P \in \Srm_{k}(\R^d)$ since $\supp P_{r} \subset \Wrm'_{k}$ for $r > 0$.
	\end{proof}

	The following result, which can be found in~\cite[Section 1.3.4, Theorem 6]{GiaquintaModicaSoucek98book2}, shows that we can strictly approximate a normal current with polyhedral boundary by polyhedral chains with the \textit{same} boundary. The proof actually demonstrates that, in the case that the original current is boundaryless, the approximant is \textit{homologous} to the original current and so, for this reason, we rephrase the result here:
	
	\begin{proposition}\label{pp:StrictPolyEqlBdry} Let $K \subset \R^d$ be compact and fix $\varepsilon > 0$. Let $T \in \Nrm_{k}(\R^d)$. Suppose that $\supp T \subset \interior K$ (the interior of $K$) and that $\partial T \in \Prm_{k-1}(\R^d)$. Then there is $P \in \Prm_{k}(\R^d)$ and $R \in \Nrm_{k+1}(\R^d)$ with $\supp P$, $\supp R \subset K$ and 
	\begin{align*}
		T &= P + \partial R, & \Mbf(P) &\leq \Mbf(T) + \varepsilon, & \Mbf(R) &\leq \varepsilon.
	\end{align*}
	\end{proposition}
	
	\subsection{Bi-Lipschitz Equivalence of Normal Flat Norms}\label{s:FlatEquiv} Let $K \subset \R^d$ be compact. We have $\Fbf_{K}(T) \leq \Fbf_{K}^{\circ}(T)$ for $T \in \Brm_{k}(K)$ by definition, but the reverse inequality fails in general. To see this, consider
	\begin{align*}
		T = \delta_{2} - \delta_{-2} \in \Brm_{0}([-2,2]).
	\end{align*}
	Notice that $\Fbf_{[-2,2]}(T) \leq \Mbf(T) = 2$, but one can show (using the constancy theorem, see~\cite[4.1.7]{Federer69book}) that $S = [\![-2,2]\!] \in \Nrm_{1}([-2,2])$ is the unique normal $1$-current with $\partial S = T$, so $\Fbf^\circ_{[-2,2]}(T) = 4$. Nevertheless, provided that $K$ is also a Lipschitz neighbourhood retract, it turns out that $\Fbf_{K}^{\circ}$ and $\Fbf_{K}$ are bi-Lipschitz equivalent on $\Brm_{k}(K)$. This was first shown in~\cite[3.2]{Federer74}; there, the result corresponds to setting $B = \varnothing$, and the proof simplifies. For this reason, we present this proof below, starting with the following result:
	
	\begin{proposition}\label{pp:HomFlatByMass} Let $K \subset \R^d$ be a compact Lipschitz neighbourhood retract. Then there is $C > 0$ (depending only on $k$, $d$, and $K$) such that, for $T \in \Brm_{k}(K)$,
		\begin{align*}
			\Fbf_{K}^{\circ}(T) \leq C \Mbf(T).
		\end{align*}
	\end{proposition}
	
	\begin{proof} Since $K$ is a compact Lipschitz neighbourhood retract, there is a bounded neighbourhood $U \subset \R^d$ of $K$ and a Lipschitz retraction $r: U \to K$. Fix $\varepsilon > 0$ with $K_{2d\varepsilon} \subset U$. We apply the deformation theorem to $T$ to obtain $P \in \Srm_{k,\varepsilon}(\R^d)$ and $R \in \Nrm_{k+1}(\R^d)$ with $T = P + \partial R$ and $\supp P$, $\supp R \subset U$. It follows that $T = r_{*}P + \partial r_{*}R$; since $T \in \Brm_{k}(K)$, we have $r_{*}P \in \Brm_{k}(K)$. We can use the estimates in the deformation theorem to obtain
	\begin{align}\label{RelFlatEst1}
		\Fbf_{K}^{\circ}(T - r_{*}P) \leq \Mbf(r_{*}R) \leq \varepsilon\gamma\cdot\Lip(r)^{k+1}\Mbf(T).
	\end{align}
	Now notice that the following space is finite-dimensional by Proposition~\ref{pp:Skeletal}:
	\begin{align*}
		\{r_{*}P: P \in \Srm_{k,\varepsilon}(\R^d), \hspace{0.2em} \supp P \subset U\} \cap \Brm_{k}(K).
	\end{align*}
	Since $\Fbf^\circ_{K}$ and $\Fbf_{K}$ are norms on this space, there is $\tilde{C} > 0$ (depending on $k$, $d$, and $K$) such that, for $P \in \Srm_{k,\varepsilon}(\R^{d})$ with $\supp P \subset U$ and $r_{*}P \in \Brm_{k}(K)$,
	\begin{align}\label{RelFlatEst2}
		\Fbf_{K}^{\circ}(r_{*}P) \leq \tilde{C}\Fbf_{K}(r_{*}P).
	\end{align}
	Combining~\eqref{RelFlatEst1} and~\eqref{RelFlatEst2} and using that $\Fbf_{K} \leq \Fbf^\circ_{K}$ yields
	\begin{align*}
		\Fbf_{K}^{\circ}(T) &\leq \Fbf_{K}^{\circ}(r_{*}P) + \Fbf_{K}^{\circ}(T-r_{*}P) \\ 
		&\leq \tilde{C}\Fbf_{K}(r_{*}P) + \Fbf_{K}^{\circ}(T-r_{*}P) \\ 
		&\leq \tilde{C}(\Fbf_{K}(T) + \Fbf_{K}(r_{*}P - T)) + \Fbf_{K}^{\circ}(T-r_{*}P) \\
		&\leq \tilde{C}(\Fbf_{K}(T) + \Fbf_{K}^{\circ}(r_{*}P - T)) + \Fbf_{K}^{\circ}(T-r_{*}P) \\
		&\leq \tilde{C}\Fbf_{K}(T) + \gamma\varepsilon(\tilde{C}+1)\Lip(r)^{k+1}\Mbf(T),
	\end{align*} 
	and since $\Fbf_{K}(T) \leq \Mbf(T)$, we obtain the result with $C = \tilde{C} +\gamma\varepsilon(\tilde{C}+1)\Lip(r)^{k+1}$. 
	\end{proof}
	We remark that Proposition~\ref{pp:HomFlatByMass} also corresponds to the ``linear isoperimetric inequality'' of~\cite{DePauwHardt22} in the setting of Lipschitz neighbourhood retracts.
	
	\begin{corollary}\label{cr:HomFlatVsFlat} Let $K \subset \R^d$ be a compact Lipschitz neighbourhood retract. Then there is $C > 0$ (depending only on $k$, $d$, and $K$) such that, for $T \in \Brm_{k}(K)$, 
	\begin{align*}
	   \Fbf_{K}^{\circ}(T) \leq C \Fbf_{K}(T).
	\end{align*}
	\end{corollary}
	
	\begin{proof} Let $T \in \Brm_{k}(K)$, and take $S \in \Nrm_{k+1}(K)$ with $\Fbf_{K}(T) = \Mbf(T-\partial S) + \Mbf(S)$. Notice that $T - \partial S \in \Brm_{k}(K)$, so we can use Proposition~\ref{pp:HomFlatByMass} to find $\tilde{C} > 0$ (depending only on $k$, $d$, and $K$) with
	\begin{align*}
		\Fbf_{K}^{\circ}(T - \partial S) \leq \tilde{C}\Mbf(T-\partial S).
	\end{align*}
	Since $\Fbf_{K}^{\circ}(\partial S) \leq \Mbf(S)$ by definition, we have
    \begin{align*}
		\Fbf_{K}^{\circ}(T) \leq \Fbf_{K}^{\circ}(T-\partial S) + \Fbf_{K}^{\circ}(\partial S) \leq \tilde{C}\Mbf(T - \partial S) + \Mbf(S) \leq C \Fbf_{K}(T)
	\end{align*}
	where $C = \max\{\tilde{C},1\}$.
	\end{proof}
	
	\subsection{Good Directions}\label{s:GoodDirections} Let $\Omega \subset \R^d$ be a Lipschitz domain. In this subsection, we shall construct a Lipschitz map $f: \overline{\Omega} \to \Omega$ which will allow us to strictly approximate normal currents $T \in \Nrm_{k}(\overline{\Omega})$ by normal currents with support in $\Omega$.
	
	Given $x_{0} \in \R^d$ and $v \in \mathbb{S}^{d-1}$, define
	\begin{align*}
		\Pi(x_{0},v) &:= \{x \in \R^d: (x-x_{0})\cdot v = 0\}, & \Pi^{\perp}(x_{0},v) &:= x_{0} + \R\cdot v.
	\end{align*}
	We shall use $\pi_{x_{0},v}$ and $\pi^{\perp}_{x_{0},v}$ to denote the corresponding orthogonal projections. Notice that, for $x \in \R^d$,
	\begin{align}\label{OrthProjDefns}
		\pi_{x_{0},v}(x) &= \pi_{v}(x-x_{0}) + x_{0}, & \pi^{\perp}_{x_{0},v}(x) &= \pi^{\perp}_{v}(x-x_{0})+x_{0}
	\end{align}
	where $\pi_{v} = \pi_{0,v}$ and $\pi^{\perp}_{v} = \pi_{0,v}^{\perp}$ are the (linear) orthogonal projections onto $(\R\cdot v)^{\perp}$ and $\R\cdot v$ respectively. We shall parameterise $\Pi^{\perp}(x_{0},v)$ with the map $\phi_{x_{0},v}: \R \to \R^d$ where $\phi_{x_{0},v}(t) = x_{0}+tv$ for $t \in \R$.
	
	The following definition was first introduced in~\cite[Definition 2.1]{BallZarnescu17}:
	
	\begin{definition}\label{dfn:GoodDirections} Let $\Omega \subset \R^d$ be open. Fix $\delta > 0$ and let $x_{0} \in \R^d$ with $\partial\Omega \cap B(x_{0},\delta) \neq \varnothing$. We say that $v \in \mathbb{S}^{d-1}$ is a \textbf{good direction} for $\Omega$ at $x_{0}$ (at scale $\delta$) if there is a continuous $g: \Pi(x_{0},v) \to \R$  that ``parameterises $\partial \Omega \cap B(x_{0},\delta)$'', that is,
	\begin{align*}
		\Omega \cap B(x_{0},\delta) &= \{x \in B(x_{0},\delta): g(\pi_{x_{0},v}(x)) < \phi^{-1}_{x_{0},v}(\pi^{\perp}_{x_{0},v}(x))\} \\
		\partial \Omega \cap B(x_{0},\delta) &= \{x \in B(x_{0},\delta): g(\pi_{x_{0},v}(x)) = \phi^{-1}_{x_{0},v}(\pi^{\perp}_{x_{0},v}(x))\}.
	\end{align*}
	\end{definition}
	\noindent Informally, the good directions for $\Omega$ at a given point are the directions in which $\partial\Omega$ looks like the graph of a continuous function. Of course, there may be no good directions for $\Omega$ at points on $\partial \Omega$, so we make the following definition:
	\begin{definition}\label{dfn:C0Bdry} An open set $\Omega \subset \R^d$ is said to have a \textbf{$\Crm^{0}$ boundary} if, for all $x_{0} \in \partial\Omega$, there is a $v \in \mathbb{S}^{d-1}$ and a $\delta > 0$ for which $v$ is a good direction for $\Omega$ at $x_{0}$ at scale $\delta$.
	\end{definition}
    In what follows, we say that $\Omega \subset \R^{d}$ is a \textbf{$\Crm^{0}$ domain} if $\Omega$ is open, bounded, connected, and has a $\Crm^{0}$ boundary in the sense of Definition \ref{dfn:C0Bdry}.
    
    Notice that, unlike domains with $\Crm^{k}$ boundaries for $k \geq 1$, an open set with a boundary covered by ``$\Crm^{0}$ charts'' may not have a $\Crm^{0}$ boundary in the sense of Definition \ref{dfn:C0Bdry} (to see this, consider the Koch snowflake). A similar phenomenon also occurs for domains with Lipschitz regularity (see \cite[Section 1.2.1]{Grisvard85}).
	
	The following lemma relates the notion of good directions to a crucial property that we shall need in the sequel.
	
	\begin{lemma}\label{lm:GoodDirectionsAreInward} Let $\Omega \subset \R^d$ be open and let $\delta > 0$. Let $x_{0} \in \R^d$ with $\partial\Omega \cap B(x_{0},\delta) \neq \varnothing$ and let $v \in \mathbb{S}^{d-1}$ be a good direction for $\Omega$ at $x_{0}$ at scale $\delta$. Fix $x \in B(x_{0},\delta) \cap \overline{\Omega}$. Then, for $t > 0$,
	\begin{align}\label{GoodDirectionsAreInward}
		x+tv \in B(x_{0},\delta) \Rightarrow x + tv \in \Omega.
	\end{align}
	In particular, if $x_{0} \in \overline{\Omega}$, then $x_{0}+tv \in \Omega$ for $t \in (0,\delta)$.
	\end{lemma}
	\begin{proof} Fix $t > 0$. We can use~\eqref{OrthProjDefns} to show that
	\begin{align*}
		\pi_{x_{0},v}(x+tv) &= \pi_{x_{0},v}(x), & \pi^{\perp}_{x_{0},v}(x+tv) &= \pi^{\perp}_{x_{0},v}(x)+tv.
	\end{align*}
	It follows that $\phi^{-1}(\pi^{\perp}_{x_{0},v}(x+tv)) = \phi^{-1}_{x_{0},v}(\pi^{\perp}_{x_{0},v}(x))+t$. Let $g: \Pi(x_{0},v) \to \R$ be a continuous function which parameterises $\partial \Omega \cap B(x_{0},\delta)$. Since $x \in \overline{\Omega} \cap B(x_{0},\delta)$, we have $g(\pi_{x_{0},v}(x)) \leq \phi^{-1}_{x_{0},v}(\pi^{\perp}_{x_{0},v}(x))$. Combining these relations yields 
	\begin{align*}
		g(\pi_{x_{0},v}(x+tv)) < \phi^{-1}_{x_{0},v}(\pi^{\perp}_{x_{0},v}(x+tv)),
	\end{align*}
	which proves~\eqref{GoodDirectionsAreInward}.
	\end{proof}
	
	The following result is the content of~\cite[Proposition 2.1]{BallZarnescu17}, and we also give a proof of this using our notation in Appendix~\ref{ax:BZ}:
	
	\begin{proposition}\label{pp:InwardPointingVecFld} Let $\Omega \subset \R^d$ be open and bounded with $\Crm^{0}$ boundary. Then there is $\delta > 0$ and a smooth Lipschitz function $\tilde{n}: (\partial\Omega)_{\delta} \to \mathbb{S}^{d-1}$ such that, for $x \in (\partial\Omega)_{\delta}$, $\tilde{n}(x)$ is a good direction for $\Omega$ at $x$ at scale $\delta$.
	\end{proposition}
	We can now construct the map described at the start of this subsection.
	
	\begin{proposition}\label{pp:ContractInto} Let $\Omega \subset \R^d$ be a $\Crm^{0}$ domain and take $\varepsilon > 0$. Then there is a Lipschitz $f_{\Omega,\varepsilon}: \overline{\Omega} \to \Omega$ with $\Lip(f_{\Omega,\varepsilon}-\mathrm{id}_{\overline{\Omega}}) \leq \varepsilon$ and $\Vert f_{\Omega,\varepsilon} - \mathrm{id}_{\overline{\Omega}} \Vert_{\infty} < \varepsilon$. Furthermore, the affine homotopy $H_{\Omega,\varepsilon}:[0,1]\times\overline{\Omega} \to \R^d$ from $\mathrm{id}_{\overline{\Omega}}$ to $f_{\Omega,\varepsilon}$ satisfies $H_{\Omega,\varepsilon}([0,1]\times\overline{\Omega}) \subset \overline{\Omega}$.
	\end{proposition}
	
	\begin{proof} Since $\Omega \subset \R^d$ is open and bounded with $\Crm^{0}$ boundary, we can apply Proposition~\ref{pp:InwardPointingVecFld} to find $\delta > 0$ and $\tilde{n}: (\partial \Omega)_{\delta} \to \mathbb{S}^{d-1}$ such that, for $x \in (\partial\Omega)_{\delta}$, $\tilde{n}(x)$ is a good direction for $\Omega$ at $x$ at scale $\delta$. Let $\varphi: \R^{+} \to [0,1]$ be a smooth cut-off function with $\varphi(t)=1$ for $t \in [0,\delta/2]$ and $\varphi(t)=0$ for $t \geq 3\delta/4$. Fix $\eta \in (0,\min\{\varepsilon,\delta\})$ (to be chosen later). We define $g_{\Omega,\varepsilon}: \overline{\Omega} \to \R^{d}$ as follows:
	\begin{align*}
	   g_{\Omega,\varepsilon}(x) := \left\{\begin{array}{cl}
       \eta\cdot\varphi(\mathrm{dist}(x,\partial\Omega))\cdot\tilde{n}(x) & \text{if}\hspace{0.2em} x \in \overline{\Omega}\backslash\Omega_{-\delta} \\
       0 & \text{if}\hspace{0.2em} x \in \Omega_{-\delta}.
	   \end{array}\right.
	\end{align*}
    Here, for a set $A \subset \R^{d}$ and $\delta > 0$, we have defined $A_{-\delta} := \{x \in A: B(x,\delta) \subset A\}$.
	Set $f_{\Omega,\varepsilon} := g_{\Omega,\varepsilon} + \mathrm{id}_{\overline{\Omega}}$. First notice that, since $\eta < \varepsilon$, $\vert \vert g_{\Omega,\varepsilon} \vert \vert_{\infty} < \varepsilon$. We also have $f_{\Omega,\varepsilon}(x) \in \Omega$ for $x \in \overline{\Omega}$ by Lemma~\ref{lm:GoodDirectionsAreInward}. Similarly, we have $H_{\Omega,\varepsilon}([0,1]\times\overline{\Omega}) \subset \overline{\Omega}$. It follows that $\overline{\Omega}$ is path-connected and one can therefore show that
	\begin{align*}
		\Lip(g_{\Omega,\varepsilon}) \leq \eta ( \Lip(\varphi)\cdot\Lip(\mathrm{dist}(\cdot,\partial\Omega)) + \Lip(\tilde{n})).
	\end{align*}
	The result follows by choosing $\eta > 0$ small enough that $\Lip(g_{\Omega,\varepsilon}) < \varepsilon$.
	\end{proof}
	
	\section{Space-Time Normal Currents}\label{s:STNormal}
	In this section, we discuss some noteworthy features of space-time normal currents, that is, normal currents in $\R^{1+d} \cong \R \times \R^{d}$, and generalise the results of~\cite{Rindler23} to this setting. 
    
    Recall from the introduction that the orthogonal temporal projection is written as $\tbf \colon \mathbb R^{1+d} \to \mathbb R \times \{0\}^d \cong \mathbb R$ where $\tbf(t,x) := t$ for $(t,x) \in \R^{1+d}$, whilst the orthogonal spatial projection is written as $\pbf \colon \mathbb R^{1+d} \to \{0\} \times \mathbb R^d \cong \mathbb R^d$ where $\pbf(t,x) := x$ for $(t,x) \in \R^{1+d}$. 
	
	\subsection{Slices of Space-Time Normal Currents}\label{s:Slicing} 
	In the sequel we shall need to consider ``slices'' of space-time normal currents by the projection map $\tbf$. The general theory of slicing presented in~\cite[Section 4.3]{Federer69book} considers the slice of a $(k+m)$-flat chain in $\R^{m+d}$ by a Lipschitz map $f: \R^{m+d} \to \R^{m}$. Since we need slightly adjusted statements and since our particular case allows one to simplify the arguments, we will give complete proofs.
	
	For a positive Borel measure $\mu$ on $\R$ and $t \in \R$, we shall use $\Theta^{1*}(\mu,t)$ to denote the \textbf{upper $1$-density} of $\mu$ at $t$, that is,
	\begin{align*}
		\Theta^{1*}(\mu,t) := \limsup_{r \downarrow 0} \frac{\mu([t-r,t+r])}{2r}.
	\end{align*}
	
	\begin{proposition}\label{pp:SlicesExist} Let $S \in \Nrm_{1+k}(\R^{1+d})$. For $\mathscr{L}^{1}$-a.e.\ $t \in \R$, the \textbf{slice} 
		\begin{align*}
			S\vert_{t} \in \Nrm_{k}(\R^{1+d}) 
		\end{align*}
		exists where, for $\omega \in \Dcal^{k}(\R^{1+d})$,
		\begin{align*}
			S\vert_{t}(\omega) := \lim_{r \downarrow 0} \frac{1}{2r}S \mvert \tbf^{*}(\boldsymbol{1}_{[t-r,t+r]}\dd t)(\omega)
            = \lim_{r \downarrow 0} \frac{1}{2r} \int_{[t-r,t+r] \times \R^d} \dprb{\vec{S}, \di t \wedge \omega} \dd \Vert S \Vert
		\end{align*}
		Furthermore, for all $t \in \R$ such that $S\vert_{t}$ exists, $(\partial S)\vert_{t}$ exists and we have the following \textbf{boundary formula}:
		\begin{align}\label{BdryFormula}
			(\partial S)\vert_{t} = -\partial(S\vert_{t}).
		\end{align}
	\end{proposition}
	\begin{proof} Fix $\omega \in \Dcal^{k}(\R^{1+d})$. For $t \in \R$ and $r > 0$, we compute as follows:
		\begin{equation}\label{SliceWeakRep}
        \begin{aligned}
			S \mvert \tbf^{*}(\boldsymbol{1}_{[t-r,t+r]} \dd t)(\omega) &= S(\tbf^{*}(\boldsymbol{1}_{[t-r,t+r]} \dd t) \wedge \omega) \\
            &= (-1)^{k} S(\omega \wedge \tbf^{*}(\boldsymbol{1}_{[t-r,t+r]} \dd t)) \\
            &= (-1)^{k}\tbf_{*}(S \mvert \omega)(\boldsymbol{1}_{[t-r,t+r]} \dd t).
		\end{aligned}
        \end{equation}
		Since $\tbf_{*}(S \mvert \omega) \in \Nrm_{1}(\R)$, there is (see~\cite[4.1.18]{Federer69book}) an $\mathscr{L}^{1}$-integrable $\xi_{\omega}: \R \to \R$ such that $\tbf_{*}(S \mvert \omega) = (\xi_{\omega} e_{0})\mathscr{L}^{1}$. It follows from Lebesgue's differentiation theorem that, for $\mathscr{L}^{1}$-a.e.\ $t \in \R$, $S\vert_{t}(\omega)$ exists and
		\begin{align}\label{SliceFormula}
			S\vert_{t}(\omega) = (-1)^{k}\xi_{\omega}(t).
		\end{align}
		Now notice that, for $\omega \in \Dcal^{k}(\R^{1+d})$, $t \in \R$, and $r > 0$, we have
		\begin{equation}\label{SliceEst1}
        \begin{aligned}
			\frac{1}{2r}\vert S \mvert \tbf^{*}(\boldsymbol{1}_{[t-r,t+r]} \dd t)(\omega)\vert &\leq \frac{1}{2r} \int_{[t-r,t+r] \times \R^d} \vert \omega \vert \dd \Vert S \Vert \\ 
            &\leq \Vert \omega \Vert_{\infty} \frac{\tbf_{\#}\|S\|([t-r,t+r])}{2r}.
		\end{aligned}
        \end{equation}
		Take a countable dense subset $\{\omega_{j}\} \subset \Dcal^{k}(\R^{1+d})$ (with respect to $\Vert \cdot \Vert_{\infty}$). Then there is $N \subset \R$ with $\mathscr{L}^{1}(N) = 0$ such that $S\vert_{t}(\omega_{j}) = (-1)^{k}\xi_{\omega_{j}}(t)$ for $t \in \R\backslash N$ and $j \in \N$. Without loss of generality, we have $\Theta^{1*}(\tbf_{\#}\Vert S \Vert,t) < \infty$ for $t \in \R\backslash N$. Fix $t \in \R\backslash N$, take $\omega \in \Dcal^{k}(\R^{1+d})$ and let $\varepsilon > 0$. Then there is $j \in \N$ with $\Vert \omega - \omega_{j} \Vert_{\infty} < \varepsilon$. It follows from~\eqref{SliceEst1} that
		\begin{align*}
			\bigg\vert \limsup_{r \downarrow 0} \frac{1}{2r} S &\mvert \tbf^{*}(\boldsymbol{1}_{[t-r,t+r]} \dd t)(\omega) \\
			&- \liminf_{r \downarrow 0} \frac{1}{2r} S \mvert \tbf^{*}(\boldsymbol{1}_{[t-r,t+r]} \dd t)(\omega) \bigg\vert \leq 2\varepsilon \cdot \Theta^{1*}(\tbf_{\#}\Vert S \Vert,t).
		\end{align*}
		Since $\varepsilon > 0$ was arbitrary, we deduce that $S\vert_{t}(\omega)$ exists. Notice that 
		\begin{align*}
			\frac{1}{2r} S \mvert \tbf^{*}(\boldsymbol{1}_{[t-r,t+r]} \dd t) \wkto S\vert_{t}, 
		\end{align*}
		so from the weak* lower semicontinuity of the mass and~\eqref{SliceEst1} we obtain 
		\begin{align}\label{SliceMassEst1}
			\Mbf(S\vert_{t}) \leq \Theta^{1*}(\tbf_{\#}\Vert S \Vert,t).
		\end{align}
		Now, for $\omega \in \Dcal^{k-1}(\R^{1+d})$, notice that $S \mvert \omega \in \Dcal_{2}(\R^{1+d})$ and so $\tbf_*(S \mvert \omega) = 0$; it follows that $\tbf_{*}(\partial S \mvert \omega) = \tbf_{*}(S \mvert \mathrm{d}\omega)$. For $t \in \R$ and $r > 0$, we can work as in~\eqref{SliceWeakRep} to obtain
		\begin{align}\label{SlicePreBdry}
			\frac{1}{2r}\partial S \mvert \tbf^{*}(\boldsymbol{1}_{[t-r,t+r]}\dd t) = -\frac{1}{2r}\partial(S \mvert \tbf^{*}(\boldsymbol{1}_{[t-r,t+r]} \dd t)).
		\end{align}
		Take $t \in \R$ and suppose that $S\vert_{t}$ exists. We can then take $r \downarrow 0$ in~\eqref{SlicePreBdry} to deduce that $(\partial S)\vert_{t}$ exists and that the boundary formula holds. We can now work as above with $\partial S$ in place of $S$ and deduce that $\Mbf(\partial(S\vert_{t})) \leq \Theta^{1*}(\tbf_{\#}\Vert \partial S\Vert,t)$. It follows that $S\vert_{t} \in \Nrm_{k}(\{t\}\times\R^d)$ for $\mathscr{L}^{1}$-a.e.\ $t \in \R$.
	\end{proof}
	Fix $S \in \Nrm_{1+k}(\R^{1+d})$. Notice that we can estimate more carefully than in~\eqref{SliceEst1} and obtain
	\begin{align}\label{SliceEst2}
		\frac{1}{2r}\vert S \mvert \tbf^{*}(\boldsymbol{1}_{[t-r,t+r]} \dd t)(\omega) \vert \leq \frac{1}{2r}\int_{[t-r,t+r] \times \R^d} \vert \omega \vert \dd \Vert S \mvert \mathrm{d}t \Vert.
	\end{align}
	Let $t \in \R$ be such that $S\vert_{t}$ exists. Taking $r \downarrow 0$ in~\eqref{SliceEst2} yields
	\begin{align}\label{SliceEst3}
		\Mbf(S\vert_{t}) \leq \Theta^{1*}(\tbf_{\#}\Vert S \mvert \mathrm{d}t \Vert,t).
	\end{align}
	The following \textbf{coarea formula} demonstrates that this estimate is in some sense optimal.
	
	\begin{proposition}\label{pp:Coarea} Let $S \in \Nrm_{1+k}(\R^{1+d})$ and let $\phi: \R \to \R$ be a bounded Borel function. Then
	\begin{align*}
		\int \vert \phi(t) \vert \, \Mbf(S\vert_{t}) \dd t = \int \vert \phi \circ \tbf \vert \hspace{0.2em} \Vert \vec{S} \mvert \mathrm{d}t \Vert \dd \Vert S \Vert.
	\end{align*}
	\end{proposition}

    To prove Proposition \ref{pp:Coarea}, we shall need the following elementary result:

    \begin{lemma}\label{lm:CoareaApprox} Let $\mu$ be a finite Borel measure on $\R^{d}$ and take $M,\varepsilon > 0$. Let $\xi: \mathbb{R}^{d} \to \Wedge_k\mathbb{R}^{d}$ be a $\mu$-measurable $k$-vector field with $\Vert \xi(x) \Vert \leq M$ for $\mu$-a.e.\ $x \in \R^{d}$. Then there is a $\mu$-measurable $G \subset \R^{d}$ with $\mu(\R^{d}\backslash G) < \varepsilon$ and a smooth $k$-covector field $\omega: \mathbb{R}^{d} \to \Wedge^k\mathbb{R}^{d}$ with $\Vert \omega \Vert_{\infty} \leq 1$ and 
    \begin{align*}
        \langle \xi(x),\omega(x)\rangle &\geq \Vert \xi(x) \Vert - \varepsilon \hspace{2em} \text{for $x \in G$}.
    \end{align*}
    \end{lemma}

    \begin{proof} Take a countable dense subset $\{\alpha_{i}\} \subset \{\alpha \in \Wedge^k\mathbb{R}^{d}: \Vert\alpha\Vert \leq 1\}$ and define $j: \mathbb{R}^{d} \to \mathbb{N}$ where, for $x \in \mathbb{R}^{d}$,
    \begin{align*}
        j(x) := \mathrm{min}\{n \in \mathbb{N}: \langle \xi(x),\alpha_{n}\rangle > \Vert\xi(x)\Vert-\varepsilon\}.
    \end{align*}
    Now define $\tilde{\omega}: \mathbb{R}^{d} \to \Wedge^k\mathbb{R}^{d}$ as $\tilde{\omega}(x) = \alpha_{j(x)}$ for $x \in \mathbb{R}^{d}$, and notice that $\Vert \tilde{\omega}(x) \Vert \leq 1$ and $\langle \xi(x),\tilde{\omega}(x)\rangle > \Vert \xi(x)\Vert - \varepsilon$ for all $x \in \mathbb{R}^{d}$ by construction. One can also show that $\tilde{\omega}$ is $\mu$-measurable. Notice that, since $\tilde{\omega}$ is uniformly bounded and $\mu$ is finite, $\tilde{\omega}$ is $\mu$-integrable. We can therefore approximate $\tilde{\omega}$ by a sequence of smooth $k$-covector fields and appeal to Egorov's theorem to obtain the result.
    \end{proof}

	\begin{proof}[Proof of Proposition \ref{pp:Coarea}] First notice that, for $\omega \in \Dcal^{k}(\R^{1+d})$,  
    \begin{align*}
        S \mvert \mathrm{d}t(\omega) = \int \langle \vec{S} \mvert \mathrm{d}t,\omega \rangle \dd \Vert S \Vert.
    \end{align*}
    It follows that $\Vert S \mvert \mathrm{d} t\Vert = \Vert \vec{S} \mvert \mathrm{d} t \Vert \, \Vert S \Vert$ and so it suffices to show that 
    \begin{align}\label{Coarea1}
		\int \vert \phi(t) \vert \, \Mbf(S\vert_{t}) \dd t = \int \vert \phi \circ \tbf \vert \dd \Vert S \mvert \mathrm{d} t\Vert.
	\end{align}
    To this end, we take the upper integral (see~\cite[2.4.2]{Federer69book}) of~\eqref{SliceEst3} and use Besicovitch's differentiation theorem to obtain
	\begin{align}\label{Coarea2}
	   \int^{*} \vert \phi(t) \vert \, \Mbf(S\vert_{t}) \dd t \leq \int \vert \phi \vert \dd \tbf_{\#}\Vert S \mvert \mathrm{d}t \Vert = \int \vert \phi \circ \tbf \vert \dd \Vert S \mvert \mathrm{d}t \Vert.
	\end{align}
    Now take $\varepsilon > 0$. We apply Lemma \ref{lm:CoareaApprox} to obtain a $\Vert S \mvert \mathrm{d}t \Vert$-measurable $G \subset \R^{1+d}$ with $\Vert S \mvert \mathrm{d}t \Vert(\R^{1+d}\backslash G) < \varepsilon$ and a smooth $\omega: \mathbb{R}^{1+d} \to \Wedge^k\mathbb{R}^{1+d}$ with $\Vert\omega\Vert_{\infty} \leq 1$ and
    \begin{align*}
        1 - \varepsilon \leq \langle \overrightarrow{S \mvert \mathrm{d}t}(t,x), \omega(t,x) \rangle \hspace{2em} \text{for $(t,x) \in G$}.
    \end{align*}
    Since $\supp S$ is compact, we can multiply $\omega$ by a cut-off function if necessary to ensure that $\omega$ has compact support. It follows that
    \begin{align}\label{Coarea3}
        \int \vert \phi \circ \tbf \vert \dd\Vert S \mvert \mathrm{d}t \Vert \leq
        (S \mvert \vert \phi \circ \tbf \vert \hspace{0.2em}\mathrm{d}t)(\omega) + \varepsilon \Vert \phi \Vert_{\infty} (1+\Mbf(S \mvert \mathrm{d}t)).
    \end{align}
    Now, similarly to~\eqref{SliceWeakRep}, we have
    \begin{align*}
		(S \mvert \vert \phi \circ \tbf\vert \hspace{0.2em}\mathrm{d}t)(\omega) = (-1)^{k}\tbf_{*}(S \mvert \omega)(\vert \phi \vert \hspace{0.2em} \mathrm{d} t).
	\end{align*}
    We then write $\tbf_{*}(S \mvert \omega) = (\xi_{\omega}e_{0})\mathscr{L}^{1}$ as in Proposition~\ref{pp:SlicesExist} and we can use~\eqref{SliceFormula} to deduce that
	\begin{align*}
		(S \mvert \vert \phi \circ \tbf\vert\hspace{0.2em}\mathrm{d}t)(\omega) = \int \vert \phi(t)\vert\, S\vert_{t}(\omega)\hspace{0.2em}\mathrm{d}t.
	\end{align*}
	Now notice that $S\vert_{t}(\omega) \leq \Mbf(S\vert_{t})$ for $t \in \mathbb{R}$ and we can take a lower integral (again, see~\cite[2.4.2]{Federer69book}) to deduce that
    \begin{align}\label{Coarea4}
        \int \vert \phi(t)\vert \, S\vert_{t}(\omega)\hspace{0.2em}\mathrm{d}t \leq \int_{*} \vert \phi(t) \vert \, \Mbf(S\vert_{t})\hspace{0.2em}\mathrm{d}t.
    \end{align}
    Combining \eqref{Coarea3} and \eqref{Coarea4} yields
    \begin{align}\label{Coarea5}
        \int \vert \phi \circ \tbf \vert \dd \Vert S \mvert \mathrm{d} t\Vert &\leq \int_{*} \vert \phi(t) \vert \, \Mbf(S\vert_{t})\hspace{0.2em}\mathrm{d}t + \varepsilon \Vert \phi \Vert_{\infty} (1+\Mbf(S \mvert \mathrm{d}t)).
    \end{align}
    Since $\varepsilon > 0$ was arbitrary, \eqref{Coarea1} follows from \eqref{Coarea2}, \eqref{Coarea5}, and~\cite[2.4.3 (7)]{Federer69book}.
	\end{proof}
	It turns out that we can make sense of $S\vert_{t} \in \Dcal_{k}(\R^{1+d})$ for \textit{all} $t \in \R$. The following result is the content of~\cite[4.3.4]{Federer69book}:
	
	\begin{proposition}\label{pp:CylinderFormula} Let $S \in \Nrm_{1+k}(\R^{1+d})$. For all $t \in \R$, $S\vert_{t} \in \Dcal_{k}(\R^{1+d})$ exists and
		\begin{align}\label{SliceAverageDefn}
			S\vert_{t} = \frac{S\vert_{t}^{+} + S\vert_{t}^{-}}{2},
		\end{align}
		where
		\begin{align*}
			S\vert_{t}^{+} &:= \partial S \mvert (t,\infty) \times \R^d - \partial(S \mvert (t,\infty) \times \R^d), \\
			S\vert_{t}^{-} &:= \partial(S \mvert (-\infty,t) \times \R^d) - \partial S \mvert (-\infty,t) \times \R^d.
		\end{align*}
	\end{proposition}
	It follows from Proposition~\ref{pp:CylinderFormula} that if $\tbf_{\#}(\Vert S \Vert + \Vert \partial S \Vert)(\{t\}) = 0$ (which is the case outside of a countable set), we have the following \textbf{cylinder formula}:
	\begin{align}\label{Cylinder}
		S\vert_{t} = \partial(S \mvert (-\infty,t) \times \R^d) - \partial S \mvert (-\infty,t) \times \R^d.
	\end{align}
	Note that the boundary formula~\eqref{BdryFormula} holds for all $t \in \R$ since it follows immediately from~\eqref{SliceAverageDefn}.
	
	For $t \in \R$, we define 
    \begin{align*}
        S(t) := \pbf_{*}S\vert_{t}.
    \end{align*}
    One can show (see~\cite[4.2.1]{Federer69book}) that $\supp S\vert_{t} \subset \{t\} \times \R^d$, so we have $S\vert_{t} = \delta_{t} \times S(t)$ and $\Mbf(S\vert_{t}) = \Mbf(S(t))$.
	
	\subsection{Space-Time Variation and Rate-Independence} Let $S \in \Nrm_{1+k}(\R^{1+d})$ and let $I \subset \R$ be $\mathscr{L}^{1}$-measurable. The \textbf{variation} and \textbf{boundary variation} of $S$ in $I$ are defined as:
	\begin{align*}
		\Var(S;I) &:= \int_{I \times \R^d} \Vert \pbf\vec{S} \Vert \dd \Vert S \Vert \\
		\Var(\partial S;I) &:= \int_{I \times \R^d} \Vert \pbf\vec{\partial S} \Vert \dd \Vert \partial S \Vert.
	\end{align*}
	Notice the presence of the mass norm $\Vert \cdot \Vert$, as opposed to the Euclidean norm $\vert \cdot \vert$, on $\Wedge_{k+1}\R^{d}$ and $\Wedge_k\R^{d}$. We also define the \textbf{$\Lrm^{\infty}$ norm} of $S$ and $\partial S$ on $I$ as follows:
	\begin{align*}
		\Vert S \Vert_{\Lrm^{\infty}(I)} &:= \esssup_{t \in I} \Mbf(S(t)) \\
		\Vert \partial S \Vert_{\Lrm^{\infty}(I)} &:= \esssup_{t \in I} \Mbf(\partial S(t)).
	\end{align*}
	If $\tbf(\supp S) \subset I$, we set
	\begin{align*}
		\Var(S) &:= \Var(S;I), & \Vert S \Vert_{\Lrm^{\infty}} &:= \Vert S \Vert_{\Lrm^{\infty}(I)}, \\
		\Var(\partial S) &:= \Var(\partial S;I), & \Vert \partial S \Vert_{\Lrm^{\infty}} &:= \Vert \partial S \Vert_{\Lrm^{\infty}(I)}.
	\end{align*}
    We remark that, in keeping with the convention set out in Section \ref{s:Prelims}, any terms involving $\Vert \partial S \Vert_{\Lrm^{\infty}(I)}$ in what follows should be ignored in the special case $k=0$ (since the slices of a $0$-current are not defined).
    
	Let $\sigma,\tau \in \R$ with $\sigma < \tau$. We define the space of \textbf{Lip-normal $(1+k)$-currents} in $[\sigma,\tau] \times \R^d$ as follows:
	\renewcommand{\arraystretch}{1.2}
	\begin{align*}
		\Nlip_{1+k}([\sigma,\tau] \times \R^d) := \left\{S \in \Nrm_{k+1}([\sigma,\tau] \times \R^d) :
		\begin{array}{l}
			\Vert S \Vert_{\Lrm^{\infty}([\sigma,\tau])} + \Vert \partial S \Vert_{\Lrm^{\infty}([\sigma,\tau])} < \infty, \\
			\tbf_{\#}\Vert S \Vert(\{\sigma,\tau\}) = 0, \\
			t \mapsto \Var(S;(\sigma,t)) \in \Lip((\sigma,\tau)), \\
			t \mapsto \Var(\partial S;(\sigma,t)) \in \Lip((\sigma,\tau))
		\end{array}\right\}.
	\end{align*}
	Furthermore, if $K \subset \R^d$ is compact, we define
	\begin{align*}
		\Nlip_{1+k}([\sigma,\tau] \times K) := \Nlip_{1+k}([\sigma,\tau] \times \R^d) \cap \Nrm_{1+k}([\sigma,\tau] \times K).
	\end{align*}
	Fix $S \in \Nlip_{1+k}([\sigma,\tau] \times \R^d)$. We shall use $\Lip(S)$ to denote the Lipschitz constant of 
	\begin{align*}
		t \mapsto \Var(S;(\sigma,t))+\Var(\partial S;(\sigma,t))
	\end{align*}
	in $\Lip((\sigma,\tau))$.
	
	Given $a \in \Lip([\sigma,\tau])$ and $S \in \Nrm_{1+k}([\sigma,\tau] \times \R^d)$, we follow the convention of~\cite{Rindler23} and define $a_{*}S := \overline{a}_{*}S$ where $\overline{a}: [\sigma,\tau] \times \R^d \to \R^{1+d}$ is given by $\overline{a}(t,x) = (a(t),x)$ for $(t,x) \in [\sigma,\tau] \times \R^d$.
	
	One of the most important properties of the variation is illustrated in the following ``rate-independence'' result:
	\begin{proposition}\label{pp:VarRescale} Let $S \in \Nrm_{1+k}([\sigma,\tau] \times \R^d)$ and let $a \in \Lip([\sigma,\tau])$ be injective. Let $I \subset \R$ be $\mathscr{L}^{1}$-measurable. Then
	\begin{align*}
		\Var(a_{*}S;a(I)) &= \Var(S;I), \\
			\Var(\partial(a_{*}S);a(I)) &= \Var(\partial S;I).
	\end{align*}
	\end{proposition}
	
	\begin{proof} We shall only show the first equality; the second is similar. Notice that we can use~\eqref{LipPushFwdFormula} to deduce that	
	\begin{align*}
		a_{*}S(\omega) = \int_{[\sigma,\tau] \times \R^d} \langle D^{S}\overline{a}(t,x)\vec{S}(t,x),\omega(a(t),x) \rangle \dd \Vert S \Vert(t,x) \hspace{2em} \text{for $\omega \in \Dcal^{1+k}(\R^{1+d})$}.
	\end{align*}
	Since $a$ is injective, the map $\overline{a}: [\sigma,\tau] \times \R^d \to a([\sigma,\tau]) \times \R^d$ has an inverse, so we can push forward the measure $\Vert S \Vert$ by $\overline{a}$ to yield, for $\omega \in \Dcal^{1+k}(\R^{1+d})$,
	\begin{align*}
		a_{*}S(\omega) = \int_{a([\sigma,\tau]) \times \R^d} \langle D^{S}\overline{a}(a^{-1}(t),x)\vec{S}(a^{-1}(t),x),\omega(t,x) \rangle \dd \overline{a}_{\#}\Vert S \Vert(t,x).
	\end{align*}
	It follows that
	\begin{align*}
		\Var(a_{*}S;a(I)) = \int_{a(I) \times \R^d} \Vert \pbf (D^{S}\overline{a}(a^{-1}(t),x)\vec{S}(a^{-1}(t),x)) \Vert \dd \overline{a}_{\#}\Vert S \Vert(t,x).
	\end{align*}
	We then push forward $\overline{a}_{\#}\Vert S \Vert$ by $\overline{a}^{-1}$ to obtain
	\begin{align*}
		\Var(a_{*}S;a(I)) = \int_{I \times \R^d} \Vert \pbf(D^{S}\overline{a}(t,x)\vec{S}(t,x)) \Vert \dd \Vert S \Vert(t,x).
	\end{align*}
	One can then use~\eqref{TangDiffDefn} to deduce that $\pbf \circ (D^{S}\overline{a}(t,x)) = \pbf$ for $\Vert S \Vert$-a.e.\ $(t,x) \in \R^d$. The result follows.
	\end{proof}
	
	We can similarly treat the $\Lrm^{\infty}$ norm:
	
	\begin{proposition}\label{pp:LInftyRescale} Let $S \in \Nrm_{1+k}([\sigma,\tau] \times \R^d)$ and let $a \in \Lip([\sigma,\tau])$ be injective. Then
		\begin{align}\label{SlicePushFwd}
			a_{*}S(a(t)) &= \sign a'(t) \cdot S(t) & & \text{for $\mathscr{L}^{1}$-a.e.\ $t \in [\sigma,\tau]$}.
		\end{align}
		It follows that, if $I \subset [\sigma,\tau]$ is $\mathscr{L}^{1}$-measurable, we have
		\begin{align*}
			\Vert a_{*}S \Vert_{\Lrm^{\infty}(a(I))} &= \Vert S \Vert_{\Lrm^{\infty}(I)}, \\
			\Vert \partial(a_{*}S) \Vert_{\Lrm^{\infty}(a(I))} &= \Vert \partial S \Vert_{\Lrm^{\infty}(I)}.
		\end{align*}
	\end{proposition}
	
	\begin{proof} For $\mathscr{L}^{1}$-a.e.\ $t \in [\sigma,\tau]$, we can use~\cite[4.3.2 (6,7)]{Federer69book} to deduce that 
		\begin{align}\label{SlicePushFwdPf}
			(a_{*}S)\vert_{a(t)} = \overline{a}_{*}\langle S, \tbf \circ \overline{a}, a(t) \rangle = \overline{a}_{*}\langle S, a \circ \tbf, a(t) \rangle = \sign a'(t) \cdot \overline{a}_{*}(S\vert_{t}).
		\end{align}
		The result follows after taking the pushforward of~\eqref{SlicePushFwdPf} by $\pbf$.
	\end{proof}
	
	One can use Propositions~\ref{pp:VarRescale} and~\ref{pp:LInftyRescale} to show that, if $a \in \Lip([\sigma,\tau])$ is injective and $S \in \Nlip_{1+k}([\sigma,\tau] \times \R^d)$, then $a_{*}S \in \Nlip_{1+k}(a([\sigma,\tau]) \times \R^d)$.
	
	\subsection{Estimates for the Variation}\label{s:VarEsts} In the sequel, we shall often need to estimate the variation of pushforwards of space-time normal currents. We start with the following observation, recalling that we write $Df(t,x) = D^{S}f(t,x)$ when the choice of $S \in \Nrm_{1+k}(\mathbb{R}^{1+d})$ is clear:
	
	\begin{lemma}\label{lm:PushFwdVar} Let $S \in \Nrm_{1+k}(\R^{1+d})$ and let $f: \R^{1+d} \to \R^{1+d}$ be Lipschitz. For all $\mathscr{L}^{1}$-measurable $I \subset \R$, we have
		\begin{align*}
			\Var(f_{*}S;I) \leq \int_{f^{-1}(I \times \R^d)} \Vert \pbf( Df(t,x)\vec{S}(t,x)) \Vert \dd \Vert S \Vert(t,x).
		\end{align*}
	\end{lemma}
	\begin{proof} In this setting,~\eqref{LipPushFwdFormula} tells us that, for $\omega \in \Dcal^{1+k}(\R^{1+d})$, 
		\begin{align*}
			f_{*}S(\omega) = \int \langle Df(t,x)\vec{S}(t,x), \omega(f(t,x)) \rangle \dd \Vert S \Vert(t,x).
		\end{align*}
		We then disintegrate $\Vert S \Vert$ with respect to the map $f$ (see~\cite[Theorem 5.3.1]{AmbrosioGigliSavare05book}), which yields a Borel family $\{\Vert S \Vert_{(s,y)}\}_{(s,y) \in \R^{1+d}}$ of probability measures such that
		\begin{align}\label{Disintegration}
			\int \phi \dd \Vert S \Vert = \int_{\R^{1+d}}\int_{f^{-1}(s,y)} \phi(t,x) \dd \Vert S \Vert_{(s,y)}(t,x) \dd f_{\#}\Vert S \Vert(s,y)
		\end{align}
		for all bounded Borel $\phi: \R^{1+d} \to \R$. In particular, for $\omega \in \Dcal^{1+k}(\R^{1+d})$, we have
		\begin{align*}
			f_{*}S(\omega) = \int_{\R^{1+d}} \left\langle \int_{f^{-1}(s,y)} Df(t,x)\vec{S}(t,x) \dd \Vert S \Vert_{(s,y)}(t,x), \omega(s,y) \right\rangle \dd f_{\#}\Vert S \Vert(s,y).
		\end{align*}
		It follows that
		\begin{align*}
			\Var(f_{*}S;I) = \int_{I \times \R^d} \left\Vert \pbf\left(\int_{f^{-1}(s,y)} Df(t,x)\vec{S}(t,x) \dd \Vert S \Vert_{(s,y)}(t,x) \right) \right\Vert \dd f_{\#}\Vert S \Vert(s,y).
		\end{align*}
		The result follows after passing $\pbf$ and the mass norm $\Vert \cdot \Vert$ into the inner integral and appealing to~\eqref{Disintegration}.
	\end{proof}
	
	We shall now generalise~\cite[Lemma 4.3]{Rindler23} to normal currents.
	
	\begin{proposition}\label{pp:AffHomVarEst} Let $f,g: \R^{d} \to \R^{d}$ be Lipschitz and let $H: \R^{1+d} \to \R^d$ be the affine homotopy between them. Let $\overline{H}: \R^{1+d} \to \R^{1+d}$ where $\overline{H}(t,x) = (t,H(t,x))$ for $(t,x) \in \R^{1+d}$. Let $T \in \Nrm_{k}(\R^d)$ and set $S = \overline{H}_{*}([\![0,1]\!] \times T)$. Then, for all intervals $I \subset [0,1]$, we have
	\begin{align*}
		\Var(S;I) &\leq \mathscr{L}^{1}(I) \int_{\R^d} \vert g(x) - f(x) \vert \cdot (\Vert Df(x) \Vert^{k} + \Vert Dg(x) \Vert^{k}) \dd \Vert T \Vert(x) \hspace{0.2em}, \\
		\Var(\partial S;\interior I) &\leq \mathscr{L}^{1}(I) \int_{\R^d} \vert g(x) - f(x) \vert \cdot (\Vert Df(x) \Vert^{k-1} + \Vert Dg(x) \Vert^{k-1}) \dd \Vert \partial T \Vert(x) \hspace{0.2em}.
	\end{align*}
	Furthermore, for $\mathscr{L}^{1}$-a.e.\ $t \in [0,1]$, the following mass estimates hold:
	\begin{align*}
		\Mbf(S(t)) &\leq \int_{\R^d} \Vert Df(x) \Vert^{k} + \Vert Dg(x) \Vert^{k} \dd \Vert T \Vert(x) \hspace{0.2em}, \\
		\Mbf(\partial S(t)) &\leq \int_{\R^d} \Vert Df(x) \Vert^{k-1} + \Vert Dg(x) \Vert^{k-1} \dd \Vert \partial T \Vert(x) \hspace{0.2em}.
	\end{align*}
	In fact, we have $S \in \Nlip_{1+k}([0,1] \times \R^d)$.
	\end{proposition}
	
	\begin{proof} Fix $a,b \in [0,1]$ with $[a,b] \subset [0,1]$. To estimate $\Var(S;[a,b])$, we shall appeal to Lemma~\ref{lm:PushFwdVar}. We can use~\cite[4.1.8]{Federer69book} and~\eqref{PolarInBundle} to deduce that $\Vert [\![0,1]\!] \times T \Vert = (\mathscr{L}^{1}\mvert[0,1]) \times \Vert T \Vert$ and that
		\begin{align}\label{ProdPolarInBundle}
			\mathrm{span}\hspace{0.2em}(e_{0} \wedge \vec{T}(x)) \subset V([\![0,1]\!] \times T,(t,x)) \hspace{2em} \text{for $\Vert T \Vert$-a.e.\ $x \in \R^d$, $\mathscr{L}^{1}$-a.e.\ $t \in [0,1]$}.
		\end{align}
		Here and in what follows, we identify $\vec{T}(x)$ with $i(\vec{T}(x))$ where $i: \R^d \to \R^{1+d}$ is given by $i(x) = (0,x)$ for $x \in \R^d$. It follows from~\eqref{TangDiffDefn} that
		\begin{align*}
			D^{[\![0,1]\!] \times T}\overline{H}(t,x)e_{0} = (1,g(x)-f(x)) \hspace{2em} \text{for $\Vert T \Vert$-a.e.\ $x \in \R^d$, $\mathscr{L}^{1}$-a.e.\ $t \in [0,1]$}.
		\end{align*}
		Notice that, for $x \in \R^d$, we have $(0,v) \in \mathrm{span}\hspace{0.2em}(e_{0} \wedge \vec{T}(x))$ whenever $v \in \mathrm{span}\hspace{0.2em}\vec{T}(x)$. We can use this fact together with~\eqref{PolarInBundle},~\eqref{ProdPolarInBundle}, and~\eqref{TangDiffDefn} to deduce that, for $\Vert T \Vert$-a.e.\ $x \in \R^d$, $\mathscr{L}^{1}$-a.e.\ $t \in [0,1]$, and $v \in \mathrm{span}\hspace{0.2em}\vec{T}(x)$,
		\begin{align*}
			D^{[\![0,1]\!] \times T}\overline{H}(t,x)(0,v) = (0,(1-t)D^{T}f(x)v+tD^{T}g(x)v).
		\end{align*}
		It follows that, for $\Vert T \Vert$-a.e.\ $x \in \R^d$ and $\mathscr{L}^{1}$-a.e.\ $t \in [0,1]$,
		\begin{align*}
			\pbf (D^{[\![0,1]\!] \times T}\overline{H}(t,x)(e_{0} \wedge 
			\vec{T}(x))) = (g(x) - f(x)) \wedge ((1-t)D^{T}f(x)+tD^{T}g(x))\vec{T}(x)
		\end{align*}
		and we can estimate
		\begin{align*}
			\Vert \pbf (D^{[\![0,1]\!] \times T}\overline{H}(t,x)(e_{0} \wedge 
			\vec{T}(x))) \Vert \leq \vert g(x) - f(x)\vert \cdot (\Vert D^{T}f(x) \Vert^{k} + \Vert D^{T}g(x) \Vert^{k}).
		\end{align*}
		The desired estimate of $\Var(\overline{H}_{*}S;[a,b])$ now follows from Lemma~\ref{lm:PushFwdVar}. Now observe that
		\begin{align*}
			\partial S = \delta_{1} \times g_{*}T - \delta_{0} \times f_{*}T - \overline{H}_{*}([\![0,1]\!] \times \partial T).
		\end{align*}
		Since $0,1 \not\in (a,b)$, we have
		\begin{align*}
			\Var(\partial S;(a,b)) = \Var(\overline{H}_{*}([\![0,1]\!] \times \partial T);(a,b)).
		\end{align*}
		We can then work as above to deduce the estimate on $\Var(\partial S;(a,b))$. For the estimate on $\Mbf(S(t))$, notice that we can use~\cite[4.3.2 (7)]{Federer69book} to deduce that $S(t) = (H(t,\cdot))_{*}T$ for $\mathscr{L}^{1}$-a.e.\ $t \in [0,1]$. We can then use~\eqref{LipPushFwdFormula} to deduce that
		\begin{align*}
			\Mbf(S(t)) \leq \int \Vert D^{T}H(t,\cdot)(x) \Vert^{k} \dd \Vert T \Vert(x).
		\end{align*}
		One can use~\eqref{TangDiffDefn} to deduce that, for $\Vert T \Vert$-a.e.\ $x \in \R^d$ and $v \in V(\Vert T \Vert,x)$, 
		\begin{align*}
			D^{T}H(t,\cdot)(x) = (1-t)D^{T}f(x) + tD^{T}g(x).
		\end{align*}
		The estimate on $\Mbf(S(t))$ follows. For the estimate on $\Mbf(\partial S(t))$, we use the boundary formula~\eqref{BdryFormula} and work as above. Finally, notice that $\Vert S \Vert \ll \overline{H}_{\#}\Vert [\![0,1]\!] \times T \Vert$ by~\eqref{PushFwdAC} and that 
		\begin{align*}
			\tbf_{\#}\overline{H}_{\#}\Vert [\![0,1]\!] \times T \Vert(\{0,1\}) = \mathscr{L}^{1}(\{0,1\}) \cdot \Mbf(T) = 0.
		\end{align*}
		It follows that $\tbf_{\#}\Vert S \Vert(\{0,1\}) = 0$ and thus $S \in \Nlip_{1+k}([0,1] \times \R^d)$. 
	\end{proof}
	
	We shall also need to estimate the variation when the pushforward is ``space-like'':
	
	\begin{proposition}\label{pp:SpaceRescale} Let $f: \R^d \to \R^d$ be Lipschitz, and let $\overline{f} = \mathrm{id}_{\R} \times f$. Let $S \in \Nrm_{1+k}(\R^{1+d})$. Then, for all $\mathscr{L}^{1}$-measurable $I \subset \R$, we have
	\begin{align*}
		\Var(\overline{f}_{*}S;I) &\leq \Lip(f)^{k+1}\Var(S;I) \\
		\Var(\partial\overline{f}_{*}S;I) &\leq \Lip(f)^{k}\Var(\partial S;I).
	\end{align*}
	Furthermore, for $\mathscr{L}^{1}$-a.e.\ $t \in \R$, the following mass estimates hold:
	\begin{align*}
		\Mbf(\overline{f}_{*}S(t)) &\leq \Lip(f)^{k}\Mbf(S(t)) \\
		\Mbf(\partial\overline{f}_{*}S(t)) &\leq \Lip(f)^{k-1}\Mbf(\partial S(t)).
	\end{align*}
	Let $\sigma,\tau \in \R$ with $\sigma < \tau$. If we additionally have $S \in \Nlip_{1+k}([\sigma,\tau] \times \R^d)$, then also $\overline{f}_{*}S \in \Nlip_{1+k}([\sigma,\tau] \times \R^d)$.
	\end{proposition}
	
	\begin{proof} Fix $I \subset \R$. For $\Vert S \Vert$-a.e.\ $(t,x) \in \R^d$, we can use~\eqref{TangDiffDefn} and the fact that $\pbf \circ \overline{f} = f \circ \pbf$ to deduce that
	\begin{align}\label{SpaceRescale1}
		\pbf (D^{S}\overline{f}(t,x)v) = \lim_{h \to 0} \frac{f(x+h\pbf v)-f(x)}{h} \hspace{2em} \text{for $v \in V(\Vert S \Vert,(t,x))$}, 
	\end{align}
    We now claim that, for $\Vert S \Vert$-a.e.\ $(t,x) \in \mathbb{R}^{d}$, the following estimate holds:
    \begin{align}\label{SpaceRescale2}
		\Vert \pbf (D^{S}\overline{f}(t,x)\vec{S}(t,x)) \Vert \leq \Lip(f)^{k+1}\Vert \pbf\vec{S}(t,x) \Vert.
	\end{align}
    To see this, take $(t,x) \in \mathbb{R}^{1+d}$ and suppose that $D^{S}\overline{f}(t,x)$ exists. We can then define a linear map
    $g(t,x): \pbf(V(\Vert S \Vert,(t,x))) \to \mathbb{R}^{d}$ where
    \begin{align*}
        g(t,x)(\pbf v) = \pbf(D^{S}\overline{f}(t,x)v) \hspace{2em} \text{for $v \in V(\Vert S \Vert,(t,x))$}.
    \end{align*}
	This is well-defined by \eqref{SpaceRescale1} and we can estimate as follows:
	\begin{align*}
		\vert g(\pbf v) \vert \leq \Lip(f)\cdot\vert \pbf v \vert \hspace{2em} \text{for $v \in V(\Vert S \Vert,(t,x))$}.
	\end{align*}
    Let $\pi(t,x): \mathbb{R}^{d} \to \pbf(V(\Vert S \Vert,(t,x)))$ be the orthogonal projection map, and notice that
    \begin{align*}
		\vert (g(t,x) \circ \pi(t,x))(w) \vert \leq \Lip(f)\cdot\vert w \vert \hspace{2em} \text{for $w \in \R^{d}$}.
	\end{align*}
    It follows that
    \begin{align}\label{SpaceRescale3}
        \Vert (g(t,x) \circ \pi(t,x))(\xi)\Vert \leq \Lip(f)^{k+1}\Vert \xi \Vert \hspace{2em} \text{for $\xi \in \Wedge_{k+1}\mathbb{R}^{d}$}.
    \end{align}
    We obtain \eqref{SpaceRescale2} by using \eqref{SpaceRescale3} with $\xi = \pbf\vec{S}(t,x)$ and appealing to \eqref{PolarInBundle}. The desired estimate of $\Var(\overline{f}_{*}S;I)$ now follows from \eqref{SpaceRescale2} and Lemma~\ref{lm:PushFwdVar}. Since $\partial \overline{f}_{*}S = \overline{f}_{*}\partial S$, we can work as above to yield the estimate of $\Var(\partial\overline{f}_{*}S;I)$. One can use~\cite[4.3.2 (7)]{Federer69book} to show that $\overline{f}_{*}S(t) = f_{*}(S(t))$ for $\mathscr{L}^{1}$-a.e.\ $t \in \R$, from which the desired mass estimate follows. The boundary mass estimate follows in the same way after using the boundary formula~\eqref{BdryFormula}. Now suppose that $S \in \Nlip_{1+k}([\sigma,\tau] \times \R^d)$. Notice that $\Vert \overline{f}_{*}S \Vert \ll \overline{f}_{\#}\Vert S \Vert$ by~\eqref{PushFwdAC} and we have
		\begin{align*}
			\tbf_{\#}\overline{f}_{\#}\Vert S \Vert(\{\sigma,\tau\}) = \tbf_{\#}\Vert S \Vert(\{\sigma,\tau\}) = 0.
		\end{align*}
		The estimates shown above then ensure that $\overline{f}_{*}S \in \Nlip_{1+k}([\sigma,\tau] \times \R^d)$.
	\end{proof}
	
	\subsection{A ``Pythagoras'' Lemma} Let $S \in \Nrm_{1+k}(\R^{1+d})$ and let $I \subset \R$ be $\mathscr{L}^{1}$-measurable. We now investigate the relationship between $\Mbf(S \mvert I \times \R^d)$, $\Var(S;I)$, and $\Vert S \Vert_{\Lrm^{\infty}(I)}$. First notice that
	\begin{align}\label{VarByMass}
		\Var(S;I) \leq \Mbf(S \mvert I \times \R^d) \leq \Mbf(S).
	\end{align}
    One cannot expect a similar estimate for $\Vert S \Vert_{\Lrm^{\infty}(I)}$; to see this, take $T \in \Nrm_{k}(\R^{d})$ and notice that, for $\eta > 0$,
    \begin{align*}
        \Vert [\![0,\eta]\!] \times T \Vert_{\Lrm^{\infty}} &= \Mbf(T), & \Mbf([\![0,\eta]\!]\times T) &= \eta\cdot\Mbf(T). 
    \end{align*}
    On the other hand, for $\omega \in \Dcal^{1+k}(\R^d)$, we have
	\begin{align*}
		\pbf_{*}(S \mvert I \times \R^d)(\omega) = \int_{I \times \R^d} \langle\pbf\vec{S}(t,x),\omega(x)\rangle \dd \Vert S \Vert(t,x) \leq \int_{I \times \R^d} \Vert \pbf\vec{S} \Vert \dd \Vert S \Vert,
	\end{align*}
	and so
	\begin{align}\label{pMassByVar}
		\Mbf(\pbf_{*}(S \mvert I \times \R^d)) \leq \Var(S;I).
	\end{align}
    Clearly, one cannot control $\Mbf(S \mvert I \times \R^{d})$ by $\Var(S;I)$ or $\Vert S \Vert_{\Lrm^{\infty}}$ alone; for the variation, it suffices to consider $[\![0,1]\!] \times T$ where $T \in \Nrm_{k}(\R^{d})$ and $T \neq 0$ and, for the $\Lrm^{\infty}$ norm, it suffices to consider $\delta_{0} \times R$, where $R \in \Nrm_{k+1}(\R^{d})$ and $R \neq 0$. However, if one uses both the variation and the $\Lrm^{\infty}$ norm then one can obtain an estimate. The following is a generalisation of the ``Pythagoras'' lemma of~\cite[Lemma 3.5]{Rindler23}:
	
	\begin{lemma}\label{lm:Pythagoras} Let $S \in \Nrm_{1+k}(\R^{1+d})$. Then 
	\begin{align}\label{Pythagoras}
		\vert \vec{S}(t,x) \mvert \mathrm{d}t \vert^{2} + \vert \pbf\vec{S}(t,x) \vert^{2} &= \vert \vec{S}(t,x) \vert^{2} & & \text{for $\Vert S \Vert$-a.e.\ $(t,x) \in \R^{1+d}$},
	\end{align}
	and, for all $\mathscr{L}^{1}$-measurable $I \subset \R$,
	\begin{align}\label{MassByVarLInfty}
		\Vert S \Vert(I \times \R^d) \leq \int_{I} \Mbf(S(t)) \dd t + \Var(S;I) \leq \mathscr{L}^{1}(I) \cdot \Vert S \Vert_{\Lrm^{\infty}(I)} + \Var(S;I).
	\end{align}
	\end{lemma}
	
	\begin{proof} Take $\xi \in \Wedge_{1+k}\R^{1+d}$. We can write $\xi$ in coordinates as follows:
		\begin{align*}
			\xi = \sum_{I \in \Lambda(k,d)} \xi_{0,I} e_{0} \wedge e_{I} + \sum_{I \in \Lambda(1+k,d)} \xi_{I}e_{I}.
		\end{align*}
		Notice that
		\begin{align*}
			\xi \mvert \mathrm{d}t &= \sum_{I \in \Lambda(k,d)} \xi_{0,I}e_{I}, &
			\pbf\xi &= \sum_{I \in \Lambda(1+k,d)} \xi_{I}e_{I}.
		\end{align*}
		We can then use orthogonality to deduce that
		\begin{align}
			\vert \xi \mvert \mathrm{d}t \vert^{2} + \vert \pbf\xi \vert^{2} = \vert \xi \vert^{2}.
		\end{align}
		We can apply this pointwise to $\vec{S}$ to obtain~\eqref{Pythagoras}.

        On the other hand, we can estimate
        \begin{align}\label{PythagorasEst0}
            \Vert \xi \Vert \leq \Vert e_{0} \wedge( \xi \mvert \mathrm{d}t) \Vert + \bigg\Vert \sum_{I \in \Lambda(1+k,d)} \xi_{I}e_{I} \bigg\Vert = \Vert \xi \mvert \mathrm{d}t \Vert + \Vert \pbf\xi \Vert.
        \end{align}
        Let $I \subset \R$ be $\mathscr{L}^{1}$-measurable. We obtain \eqref{MassByVarLInfty} by applying \eqref{PythagorasEst0} pointwise to $\vec{S}$, integrating this over $I \times \R^{d}$ with respect to $\Vert S \Vert$, and applying the coarea formula.
	\end{proof}

	\subsection{Compactness and Continuity} The following compactness result is the analogue of~\cite[Theorem 3.7]{Rindler23} in our setting:

	\begin{proposition}\label{pp:SpaceTimeCpct} Fix $\sigma,\tau \in \R$ with $\sigma < \tau$ and let $K \subset \R^{d}$ be compact. Let $(S_{j})_{j} \subset \Nrm_{1+k}([\sigma,\tau] \times K)$ and suppose that there is $M > 0$ with 
		\begin{align}\label{SpaceTimeCpctBounds}
			\Vert S_{j} \Vert_{\Lrm^{\infty}} + \Vert \partial S_{j} \Vert_{\Lrm^{\infty}} + \Var(S_{j}) + \Var(\partial S_{j}) \leq M \hspace{1em} \text{for $j \in \N$}.
		\end{align}
		Then there is $S \in \Nrm_{1+k}([\sigma,\tau] \times K)$, a (not relabelled) subsequence of $(S_{j})_{j}$, and a countable set $N \subset [\sigma,\tau]$ with 
		\begin{align*}
			S_{j} &\wkto S & &\text{and} & S_{j}(t) &\wkto S(t) \hspace{1em} \text{for all $t \in [\sigma,\tau]\backslash N$}.
		\end{align*}
		Furthermore, 
		\begin{align*}
			\Var(S) &\leq \liminf_{j \to \infty} \Var(S_{j}), & \Vert S \Vert_{\Lrm^{\infty}} &\leq \liminf_{j \to \infty} \Vert S_{j} \Vert_{\Lrm^{\infty}}, \\
			\Var(\partial S) &\leq \liminf_{j \to \infty} \Var(\partial S_{j}), & \Vert \partial S \Vert_{\Lrm^{\infty}} &\leq \liminf_{j \to \infty} \Vert \partial S_{j} \Vert_{\Lrm^{\infty}}.
		\end{align*}
		If we additionally have $\{S_{j}\} \subset \Nlip_{1+k}([\sigma,\tau] \times K)$ and $(\Lip(S_{j}))_{j}$ is uniformly bounded, then
		\begin{align*}
			S &\in \Nlip_{1+k}([\sigma,\tau] \times K) & &\text{and} & S_{j}(t) &\wkto S(t) \hspace{1em} \text{for all $t \in [\sigma,\tau]$}.
		\end{align*}
		Furthermore, we have
		\begin{align}\label{LipLSC}
			\Lip(S) \leq \liminf_{j \to \infty} \Lip(S_{j}).
		\end{align}
	\end{proposition}

    The proof of Proposition \ref{pp:SpaceTimeCpct} is very similar to that of \cite[Theorem 3.7]{Rindler23}, but relies upon generalisations of other results from \cite{Rindler23} which are established above. For this reason, we give the proof below:
    
    \begin{proof}[Proof of Proposition \ref{pp:SpaceTimeCpct}] First notice that, for $j \in \N$, we can use Lemma \ref{lm:Pythagoras} and \eqref{SpaceTimeCpctBounds} to estimate
    \begin{align*}
        \Mbf(S_{j})+\Mbf(\partial S_{j}) \leq \max\{1,\vert \tau - \sigma \vert\} \cdot M.
    \end{align*}
    It follows that there is $S \in \Nrm_{1+k}([\sigma,\tau] \times K)$ and a (not relabelled) subsequence of $(S_{j})_{j}$ with $S_{j} \wkto S$. The lower semicontinuity of the variation then follows from Reshetnyak's lower semicontinuity theorem (see \cite[Theorem 2.38]{AmbrosioFuscoPallara00book}).
    
    After possibly passing to another subsequence, we can assume without loss of generality that $\Vert S_{j} \Vert + \Vert \partial S_{j} \Vert \wkto \mu$ for some $\mu \in \Mcal^{+}(\R^{1+d})$. Fix $t \in [\sigma,\tau]$ with $\mu(\{t\} \times \R^{d})=0$. We can use \cite[Theorem 1.62(b)]{AmbrosioFuscoPallara00book} to deduce that, for $\omega \in \Dcal^{k}(\R^{1+d})$,
    \begin{equation}\label{SpaceTimeCpct1}
    \begin{aligned}
        S_{j}(\mathbf{1}_{(-\infty,t) \times \R^{d}}\cdot\mathrm{d}\omega) &\to S(\mathbf{1}_{(-\infty,t) \times \R^{d}}\cdot\mathrm{d}\omega), \\
        \partial S_{j}(\mathbf{1}_{(-\infty,t) \times \R^{d}}\cdot\omega) &\to \partial S(\mathbf{1}_{(-\infty,t) \times \R^{d}}\cdot\omega), \\
        S_{j}(\mathbf{1}_{(t,\infty) \times \R^{d}}\cdot\mathrm{d}\omega) &\to S(\mathbf{1}_{(t,\infty) \times \R^{d}}\cdot\mathrm{d}\omega), \\ 
        \partial S_{j}(\mathbf{1}_{(t,\infty) \times \R^{d}}\cdot\omega) &\to \partial S(\mathbf{1}_{(t,\infty) \times \R^{d}}\cdot\omega).
    \end{aligned}
    \end{equation}
    Let $N \subset [\sigma,\tau]$ be the set of points for which $\mu(\{t\} \times \R^{d}) > 0$ , It follows from \eqref{SpaceTimeCpct1} and Proposition \ref{pp:CylinderFormula} that $S_{j}\vert_{t} \wkto S\vert_{t}$, and therefore $S_{j}(t) \wkto S(t)$, for $t \in [\sigma,\tau]\backslash N$. One can then deduce the lower semicontinuity of the $\Lrm^{\infty}$ norms.

    Now suppose additionally that $(S_{j})_j \subset \Nlip_{1+k}([\sigma,\tau] \times K)$ and that there is $L > 0$ such that $\Lip(S_{j}) \leq L$ for $j \in \N$. Take $s,t \in [\sigma,\tau]$ with $s < t$. For $j \in \N$, we can use Lemma \ref{lm:Pythagoras} to estimate
    \begin{align*}
        \Vert S_{j} \Vert([s,t] \times \R^{d}) + \Vert \partial S_{j} \Vert([s,t] \times \R^{d}) &\leq (M+L)\cdot\vert t-s \vert.
    \end{align*}
    It follows that $\mu(\{t\} \times \R^{d}) = 0$ for \textit{all} $t \in [\sigma,\tau]$ and so $N = \varnothing$. In particular, we have $S_{j}(t) \wkto S(t)$ for all $t \in [\sigma,\tau]$. Finally, for $s,t \in [\sigma,\tau]$ with $s < t$, we can again appeal to Reshetnyak's lower semicontinuity theorem to deduce that
    \begin{align*}
        \Var(S;[s,t]) + \Var(\partial S;[s,t]) &\leq \liminf_{j \to \infty} \Var(S_{j};[s,t]) + \Var(\partial S_{j};[s,t]) \\
        &\leq \liminf_{j \to \infty} \Lip(S_{j}) \cdot \vert t-s \vert.
    \end{align*}
    In particular, $S \in \Nlip_{1+k}([\sigma,\tau] \times \R^{d})$ and \eqref{LipLSC} holds.
    \end{proof}
 
	\begin{corollary}\label{cr:SliceWkCty} Fix $\sigma,\tau \in \R$ with $\sigma < \tau$ and let $K \subset \R^{d}$ be compact. Let $(S_{j})_{j} \subset \Nrm_{1+k}([\sigma,\tau] \times K)$ and let $S \in \Nrm_{1+k}([\sigma,\tau] \times K)$. Suppose that $S_{j} \wkto S$ and that there is $M > 0$ with 
	\begin{align*}
        \Mbf(S_{j})+\Mbf(\partial S_{j}) \leq M
    \end{align*}
    for all $j \in \N$. Then there is a (not relabelled) subsequence of $(S_{j})_{j}$ such that $S_{j}(t) \wkto S(t)$ for $\mathscr{L}^{1}$-a.e.\ $t \in [\sigma,\tau]$.
	\end{corollary}
	
    \begin{remark} Under the assumptions of Corollary \ref{cr:SliceWkCty}, one cannot expect that $S_{j}(t) \wkto S(t)$ for $\mathscr{L}^{1}$-a.e. $t \in [\sigma,\tau]$ along the \textit{full} sequence. To see this, consider all dyadic intervals in $[0,1]$ listed by generation and, for $j \in \N$, let $S_j \in \Irm_1([0,1])$ be the 1-current associated with the $j$-th such interval.  
    Notice that $\Mbf(S_{j}) \to 0$, but one can show that $(S_{j}(t))_{j}$ does not converge for any $t \in [0,1]$. This behaviour is analogous to the fact that convergent sequences in $\Lrm^{1}([0,1])$ may not converge $\mathscr{L}^{1}$-a.e.
    \end{remark}
	
	\section{Space-Time Trajectories and the Dynamical Deformation Theorem}\label{s:dynamical} The deformation theorem (see Section~\ref{s:Approx}) is distinguished among approximation theorems in that the approximant $P \in \Prm_{k}(\R^d)$ of $T \in \Nrm_{k}(\R^d)$ is constructed from pushforwards of $T$. In fact, we have a decomposition $T = P + \partial R$ for some $R \in \Nrm_{k+1}(\R^d)$ where $R$ is (up to technical details) provided by the homotopy formula (see \eqref{NaiveDeformation}). Although the fact that $P$ is a deformation of $T$ is witnessed by $R$, $R$ does not show the ``history'' of a deformation as there may be cancellations. To overcome this, we shall instead construct a ``space-time trajectory'' $S \in \Nrm_{1+k}([0,1] \times \R^d)$ with $\partial S = \delta_{1} \times P - \delta_{0} \times T$ and estimate the variation of $S$. We start by developing the theory of space-time trajectories.
	
	\subsection{Space-Time Trajectories} A \textbf{space-time trajectory} is a space-time normal current $S \in \Nrm_{1+k}([0,1] \times \R^d)$ with $\partial S \mvert (0,1) \times \R^d = 0$. To motivate this definition, let $S \in \Nrm_{1+k}([0,1] \times \R^d)$ be a space-time trajectory and notice that
	\begin{align}
		\partial S = \partial S \mvert \{0\} \times \R^d + \partial S \mvert \{1\} \times \R^d.
	\end{align}
	We define
	\begin{align*}
		\partial^{+}S &:= \pbf_{*}(\partial S \mvert \{1\} \times \R^d), \\
		\partial^{-}S &:= -\pbf_{*}(\partial S \mvert \{0\} \times \R^d), 
	\end{align*}
	so that
	\begin{align}\label{TrajDeform}
		\partial S = \delta_{1} \times \partial^{+}S - \delta_{0} \times \partial^{-}S.
	\end{align}
	In particular, $S$ describes a ``trajectory'' from $\partial^{-}S$ to $\partial^{+}S$.  Taking the boundary of~\eqref{TrajDeform} yields $0 = \delta_{1} \times \partial(\partial^{+}S) - \delta_{0} \times \partial(\partial^{-}S)$ and so $\partial(\partial^{-}S) = \partial(\partial^{+}S) = 0$.
	
	We say that $S$ is a \textbf{Lipschitz trajectory} if we also have $S \in \Nlip_{1+k}([0,1] \times \R^d)$; notice that this is the case if and only if $\Vert S \Vert_{\Lrm^{\infty}} < \infty$, $\tbf_{\#}\Vert S \Vert(\{0,1\}) = 0$, and the map $t \mapsto \Var(S;(0,t))$ is Lipschitz on $(0,1)$. The following result relates $\Mbf(\partial^{-}S)$ and $\Mbf(\partial^{+}S)$ to $\Vert S \Vert_{\Lrm^{\infty}}$ whenever $S$ is a Lipschitz trajectory:

    \begin{lemma}\label{lm:SlicesByBdry} Let $S \in \Nlip_{1+k}([0,1] \times \R^{d})$ and suppose that $\partial S \mvert (0,1) \times \R^{d} = 0$. Then
    \begin{align*}
        \max\{\Mbf(\partial^{-}S),\Mbf(\partial^{+}S)\} \leq \Vert S \Vert_{\Lrm^{\infty}}.
    \end{align*}
    \end{lemma}
    
    \begin{proof} We shall show that $\Mbf(\partial^{-}S) \leq \Vert S \Vert_{\Lrm^{\infty}}$; the inequality $\Mbf(\partial^{+}S) \leq \Vert S \Vert_{\Lrm^{\infty}}$ is similar. Suppose, for a contradiction, that $\Mbf(\partial^{-}S) > \Vert S \Vert_{\Lrm^{\infty}}$. Then there is $N \subset [0,1]$ with $\mathscr{L}^{1}(N) = 0$ and a $\delta > 0$ such that 
    \begin{align}\label{SlicesByBdry0}
        \Mbf(S\vert_{t}) < \Mbf(\partial^{-}S) - \delta \hspace{1em} \text{for $t \in [0,1] \setminus N$}.
    \end{align}
    We can use the remark following Proposition \ref{pp:CylinderFormula} to find a sequence $(t_{j})_{j} \subset [0,1] \backslash N$ with $t_{j} \to 0$ such that, for $j \in \N$,
    \begin{align}\label{SlicesByBdry1}
        S\vert_{t_{j}} = \partial(S \mvert (-\infty,t_{j}) \times \R^{d}) - \partial S \mvert (-\infty,t_{j}) \times \R^{d}.
    \end{align}
    Take $j \in \N$. It follows from \eqref{SlicesByBdry1} that
    \begin{align}\label{SlicesByBdry2}
        S(t_{j}) - \partial^{-}S = \pbf_{*}\partial(S \mvert (-\infty,t_{j}) \times \R^{d}).
    \end{align}
    We can use Lemma \ref{lm:Pythagoras} to estimate
    \begin{align*}
        \Mbf(S \mvert (-\infty,t_{j}) \times \R^{d}) &\leq t_{j} \cdot \Vert S \Vert_{\Lrm^{\infty}} + \Var(S;[0,t_{j}]) \\
        &\leq t_{j} \cdot(\Vert S \Vert_{\Lrm^{\infty}} + \Lip(S)).
    \end{align*}
    It follows that $\Mbf(S \mvert (-\infty,t_{j}) \times \R^{d}) \to 0$ and so we can use \eqref{SlicesByBdry2} to deduce that $S(t_{j}) \wkto \partial^{-}S$. In particular, we have
    \begin{align*}
        \Mbf(\partial^{-}S) \leq \liminf_{j \to \infty} \Mbf(S(t_{j})),
    \end{align*}
    which contradicts~\eqref{SlicesByBdry0}.
    \end{proof}
	
	We now define various operations on space-time trajectories (in analogy with~\cite{Rindler23}) that will be useful in the sequel. 
	
	Let $S_{1},S_{2} \in \Nrm_{1+k}([0,1] \times \R^d)$ and suppose that $\partial^{-}S_{2} = \partial^{+}S_{1}$. We define the \textbf{concatenation} of $S_{1}$ and $S_{2}$ as
	\begin{align*}
		S_{2} \circ S_{1} := (a_{1})_{*}S_{1} + (a_{2})_{*}S_{2}
	\end{align*}
	where, for $i \in \{1,2\}$, $a_{i}: [0,1] \to [0,1]$ is given by 
	\begin{align*}
		a_{i}(t) &:= \frac{i-1}{2}+\frac{t}{2}, & t \in [0,1].
	\end{align*}
	Notice that $\partial^{-}(S_{2} \circ S_{1}) = \partial^{-}S_{1}$ and $\partial^{+}(S_{2} \circ S_{1}) = \partial^{+}S_{2}$. We can use Propositions~\ref{pp:VarRescale} and~\ref{pp:LInftyRescale} to deduce that
	\begin{align*}
		\Var(S_{2} \circ S_{1}) &\leq \Var(S_{1}) + \Var(S_{2}) & \Vert S_{2} \circ S_{1} \Vert_{\Lrm^{\infty}} &= \max\{\Vert S_{1} \Vert_{\Lrm^{\infty}},\Vert S_{2} \Vert_{\Lrm^{\infty}}\}.
	\end{align*}
    Furthermore, if $\tbf_{\#}\Vert S_{1} \Vert(\{1\}) = \tbf_{\#}\Vert S_{2} \Vert(\{0\}) = 0$, then $\tbf_{\#}\Vert S_{2} \circ S_{1} \Vert(\{1/2\}) = 0$ and so
    \begin{align*}
        \Var(S_{2} \circ S_{1}) &= \Var(S_{2} \circ S_{1};[0,1/2]) + \Var(S_{2} \circ S_{1};[1/2,1]) \\
        &= \Var((a_{2})_{*}S_{2}) + \Var((a_{1})_{*}S_{1}) \\
        &= \Var(S_{2})+\Var(S_{1}).
    \end{align*}
	One can also show that, if $S_{1}$ and $S_{2}$ are Lipschitz trajectories, then $S_{2} \circ S_{1}$ is also a Lipschitz trajectory. Now, if $j \in \N$ with $j > 1$ and $\{S_{i}\}_{i=1}^{j}$ is a set of space-time trajectories with $\partial^{-}S_{i+1}=\partial^{+}S_{i}$ for $i \in \{1,\dotsc,j-1\}$, we define the concatenation of $\{S_{i}\}_{i=1}^{j}$ as
	\begin{align*}
		S_{j} \circ \cdots \circ S_{1} := \sum_{i=1}^{j} (a^{(j)}_{i})_{*}S_{i}
	\end{align*}
	where, for $i \in \{1,\cdots,j-1\}$, $a_{i}^{(j)}: [0,1] \to [0,1]$ is given by
	\begin{align*}
		a_{i}^{(j)}(t) &:= \frac{i-1}{j}+\frac{t}{j}, & t &\in [0,1].
	\end{align*}
	Notice that $\partial^{-}(S_{j} \circ \cdots \circ S_{1}) = \partial^{-}S_{1}$ and $\partial^{+}(S_{j} \circ \cdots \circ S_{1}) = \partial^{+}S_{j}$. As above, we can use Propositions~\ref{pp:VarRescale} and~\ref{pp:LInftyRescale} to deduce that
	\begin{align*}
		\Var(S_{j} \circ \cdots \circ S_{1}) &\leq \sum_{i=1}^{j} \Var(S_{i}), & \Vert S_{j} \circ \cdots \circ S_{1} \Vert_{\Lrm^{\infty}} = \max\{\Vert S_{i} \Vert_{\Lrm^{\infty}}\}_{i=1}^{j}.
	\end{align*}
	Furthermore, if $\tbf_{\#}\Vert S_{i} \Vert(\{1\}) = \tbf_{\#}\Vert S_{i+1} \Vert(\{0\}) = 0$ for $i \in \{1,\dotsc,j-1\}$, then we can work as above to deduce that
    \begin{align*}
        \Var(S_{j} \circ \cdots \circ S_{1}) &= \sum_{i=1}^{j} \Var(S_{i}).
    \end{align*}
    One can also show that, if $S_{i}$ is a Lipschitz trajectory for $i \in \{1,\dotsc,j-1\}$, then $S_{j} \circ \cdots \circ S_{1}$ is a Lipschitz trajectory.
	
	Again, let $S \in \Nrm_{1+k}([0,1] \times \R^d)$ be a space-time trajectory. We define the \textbf{reversal} of $S$ as $S^{-1}:= a_{*}S$ where $a: [0,1] \to [0,1]$ is given by $a(t) = 1-t$ for $t \in [0,1]$. Notice that $\partial^{-}S^{-1} = \partial^{+}S$ and $\partial^{+}S^{-1} = \partial^{-}S$. It follows from Propositions~\ref{pp:VarRescale} and~\ref{pp:LInftyRescale} that $\Var(S^{-1}) = \Var(S)$ and $\Vert S^{-1} \Vert_{\Lrm^{\infty}} = \Vert S \Vert_{\Lrm^{\infty}}$; one can check that $S^{-1}$ is a Lipschitz trajectory whenever $S$ is a Lipschitz trajectory.
	
	We end this subsection with the following useful fact:
	
	\begin{lemma}\label{lm:Rescale} Let $S \in \Nlip_{1+k}([0,1] \times \R^d)$ and suppose that $\partial S \mvert (0,1) \times \R^d = 0$. Then there is a Lipschitz $a: [0,1] \to [0,1]$ such that $\partial a_{*}S = \partial S$ and $\Lip(a_{*}S) \leq C \Var(S)$ for some universal constant $C > 0$.
	\end{lemma}
	
	\begin{proof} If $\Var(S;[0,1])=0$, then we can take $a = \mathrm{id}_{\R}$ and there is nothing to show. Otherwise, define $a:[0,1] \to [0,1]$ where, for $t \in [0,1]$,
	\begin{align}\label{Rescale1}
		a(t) := \frac{1}{2}\bigg(\frac{\Var(S;[0,t])}{\Var(S;[0,1])}+t\bigg).
	\end{align}
	Since $a(0)=0$ and $a(1)=1$, we have $a_{*}\partial S = \partial S$. Notice that $a \in \Lip([0,1])$ so, for $t \in [0,1]$, we can use Proposition~\ref{pp:VarRescale} to deduce that $\Var(a_{*}S;[0,t]) = \Var(S;[0,a^{-1}(t)])$; we can then use \eqref{Rescale1} to deduce that
    \begin{align*}
        \Var(a_{*}S;[0,t]) = (2t-a^{-1}(t)) \Var(S;[0,1]).
    \end{align*}
    The result follows since $\Lip(a^{-1}) \leq 2$.
	\end{proof}
	
	\subsection{Cones over Normal Currents} Let $T \in \Nrm_{k}(\R^d)$ and fix $v \in \R^d$. We define the \textbf{(right) cone} of $v$ over $T$ as
	\begin{align*}
		T \triangleright v := \overline{H}_{*}([\![0,1]\!] \times T),
	\end{align*}
	where $\overline{H}: \R^{1+d} \to \R^{1+d}$ is given by $\overline{H}(t,x) = (t,(1-t)x+tv)$ for $(t,x) \in \R^{1+d}$. It follows from Proposition~\ref{pp:AffHomVarEst} that
	\begin{align}\label{ConeLip}
		T \triangleright v \in \Nlip_{1+k}([0,1] \times \R^d)
	\end{align}
    and that
	\begin{align}\label{ConeEsts}
		\Var(T \triangleright v) &\leq \int \vert x - v \vert \dd\Vert T \Vert(x), & \Vert T \triangleright v \Vert_{\Lrm^{\infty}} \leq \Mbf(T).
	\end{align}
    Furthermore, if $\partial T = 0$, then $T \triangleright v$ is a Lipschitz trajectory.
    
    We shall also need to work with cones over polyhedral chains. Take $v_{0},\dotsc,v_{k} \in \R^{d}$. We claim that
    \begin{align}\label{STPolyCone}
        [\![v_{0},\dotsc,v_{k}]\!] \triangleright v = (-1)^{k}[\![(0,v_{0}),\dotsc,(0,v_{k}),(1,v)]\!].
    \end{align}
    To see this, notice that $\overline{H} = H \circ i$ where $H: [0,1] \times \R^{1+d} \to \R^{1+d}$ is given by $H(s,(t,x)) = (1-s)(t,x) + s(1,v)$ for $(s,(t,x)) \in [0,1] \times \R^{1+d}$ and $i: \R^{1+d} \to [0,1] \times \R^{1+d}$ is given by $i(t,x) = (t,(0,x))$ for $(t,x) \in \R^{1+d}$. It follows that
    \begin{align}\label{STPolyCone1}
         [\![v_{0},\dotsc,v_{k}]\!] \triangleright v = H_{*}(i_{*}([\![0,1]\!] \times [\![v_{0},\dotsc,v_{k}]\!])).
    \end{align}
    One can use \eqref{SimpDecomp} and \eqref{PolyPushFwd} to deduce that
    \begin{align*}
        i_{*}([\![0,1]\!] \times [\![v_{0},\dotsc,v_{k}]\!]) = [\![0,1]\!] \times [\![(0,v_{0}),\dotsc,(0,v_{k})]\!],
    \end{align*}
    and so \eqref{STPolyCone} follows from \eqref{STPolyCone1} and \eqref{SimplicesJoin}.
    
    The following lemma demonstrates how to compute the variation of a cone over an oriented simplex:
    
    \begin{lemma}\label{lm:SimplexVar} Take $v_{0},\dotsc,v_{1+k} \in \R^{1+d}$ and set $P := [\![v_{0},\dotsc,v_{1+k}]\!] \in \Prm_{1+k}(\R^{1+d})$. Then
    \begin{align*}
        \Var(P) = \Mbf(\pbf_{*}P).
    \end{align*}
    \end{lemma}

    \begin{proof} We can use \eqref{SimplicesAreRect} to deduce that
    \begin{align*}
        \Var(P) &= \frac{\vert \pbf(\xi) \vert}{\vert \xi \vert}\cdot\mathscr{H}^{1+k}(C)
    \end{align*}
    where $C \subset \R^{1+d}$ is the convex hull of $\{v_{0},\dotsc,v_{1+k}\}$ and $\xi = (v_{1} - v_{0}) \wedge \cdots \wedge (v_{1+k}-v_{k})$. On the other hand, we have 
    \begin{align*}
        \pbf_{*}P(\omega) = \int_{C} \bigg\langle \frac{\pbf(\xi)}{\vert \xi \vert},\omega(x) \bigg\rangle \dd\mathscr{H}^{k+1}(t,x) \hspace{1em} \text{for $\omega \in \Dcal^{1+k}(\R^{d})$},
    \end{align*}
    and so
    \begin{equation*}
        \Mbf(\pbf_{*}P) = \frac{\vert \pbf(\xi) \vert}{\vert \xi \vert}\cdot\mathscr{H}^{1+k}(C) = \Var(P). \qedhere
    \end{equation*}
    \end{proof}

    Take $v_{0},\dotsc,v_{k} \in \R^{d}$. It follows from \eqref{STPolyCone}, Lemma \ref{lm:SimplexVar}, and \eqref{PolyPushFwd} that
    \begin{align}\label{PolyCone}
        \Var([\![v_{0},\dotsc,v_{k}]\!] \triangleright v) = \Mbf([\![v_{0},\dotsc,v_{k},v]\!]).
    \end{align}
		
	The following result, adapted from the proof of~\cite[Theorem 5.4]{Rindler23}, demonstrates that we can `stretch out' a polyhedral chain $P \in \Prm_{1+k}(\R^{d})$ in time:
		
	\begin{lemma}\label{lm:Stretch} Let $K \subset \R^d$ be compact. Let $P \in \Prm_{1+k}(\R^d)$ and suppose that $\supp P \subset K$. There is $S \in \Prm_{1+k}(\R^{1+d})$ with $S \in \Nlip_{1+k}([0,1] \times K)$ such that
	\begin{align*}
		\partial S &= - \delta_{0} \times \partial P,  & \Var(S) &= \Mbf(P).
	\end{align*}
	\end{lemma}
		
	\begin{proof} First suppose that $P = [\![v_{0},\dotsc,v_{1+k}]\!]$ for some $v_{0},\dotsc,v_{1+k} \in \R^d$. In this case, define $S := \partial P \triangleright v_{0}$, so that $\partial S = - \delta_{0} \times \partial P$ and $S \in \Nlip_{1+k}([0,1] \times \R^d)$. Since $\supp P$ is convex and $\supp P \subset K$, we have $\supp S \subset [0,1] \times K$. We can use \eqref{PolyBdry} and \eqref{STPolyCone} to compute
    \begin{align*}
        \pbf_{*}S &= \sum_{j=0}^{1+k} (-1)^{j} \pbf_{*}([\![v_{0},\dotsc,\hat{v}_{j},\dotsc,v_{1+k}]\!] \triangleright v_{0}) \\
        &= \sum_{j=0}^{1+k} (-1)^{j}\cdot(-1)^{k} [\![v_{0},\dotsc,\hat{v}_{j},\dotsc,v_{1+k},v_{0}]\!].
    \end{align*}
    It follows that $\pbf_{*}S = (-1)^{k} [\![v_{1},\dotsc,v_{1+k},v_{0}]\!]$ because all but the first term in the sum vanishes. It then follows from \eqref{pMassByVar} that $\Mbf(P) \leq \Var(S)$. For the reverse estimate, we use \eqref{PolyBdry} and \eqref{PolyCone} to deduce that
    \begin{align*}
        \Var(S) &\leq \sum_{j=0}^{1+k} \Var([\![v_{0},\dotsc,\hat{v}_{j},\dotsc,v_{1+k}]\!] \triangleright v_{0}) \\
        &= \sum_{j=0}^{1+k} \Mbf([\![v_{0},\dotsc,\hat{v}_{j},\dotsc,v_{1+k},v_{0}]\!]).
    \end{align*}
    It then follows that $\Var(S) \leq \Mbf(P)$ because all but the first term in the sum vanishes. In the general case, we write $P$ as a sum of simplices which are disjoint up to $\mathscr{H}^{1+k}$-negligible subsets and apply the above construction to each simplex.
	\end{proof}

	\subsection{The Dynamical Deformation Theorem} The following is the sought generalisation of~\cite[Theorem 4.7]{Rindler23} to normal currents, where we recall that $\lambda_{\varepsilon}(x) = \varepsilon x$ for $x \in \R^d$:
    
    \begin{theorem}[Dynamical Deformation Theorem]\label{tm:STDeform} Let $T \in \Nrm_{k}(\R^d)$ with $\partial T = 0$ and fix $\varepsilon > 0$. Then there is $P \in \Srm_{k,\varepsilon}(\R^d)$ and a Lipschitz trajectory $S \in \Nlip_{1+k}([0,1] \times \R^d)$ with
	\begin{align}\label{STDeformBdryS}
		\partial S = \delta_{1} \times P - \delta_{0} \times T
	\end{align}
    such that the following estimates hold with a constant $\gamma > 0$ depending only on $k$ and $d$:
	\begin{align*}
		\Mbf(P) &\leq \gamma \Mbf(T), & \Var(S) &\leq \gamma\varepsilon \Mbf(T), & \Vert S \Vert_{\Lrm^{\infty}} &\leq \gamma\Mbf(T).
	\end{align*}
	Furthermore, we have $\supp P \subset (\supp T)_{2d\varepsilon}$ and $\supp S \subset [0,1] \times (\supp T)_{2d\varepsilon}$.
	\end{theorem}

    To prove Theorem \ref{tm:STDeform}, we shall need the following result:
	
	\begin{lemma}\label{lm:STDeformVarEsts} Let $K \subset \R^d$ be compact. Let $f: K \to \R^d$ be Lipschitz and define $\overline{H}: [0,1] \times K \to \R^{1+d}$ where 
	\begin{align*}
		\overline{H}(t,x) = (t,(1-t)x+tf(x)) \hspace{2em} \text{for $(t,x) \in [0,1]\times K$}.
	\end{align*}
	Take $T \in \Nrm_{k}(\R^d)$ with $\supp T \subset \interior K$. Suppose that there is a continuous $\alpha: K \to \R^{+}$ such that $Df(x)$ exists and $\Vert Df(x) \Vert \leq \alpha(x)$ for $\mathscr{L}^{d}$-a.e.\ $x \in K$. Then, for $a,b \in [0,1]$ with $a < b$, we have
	\begin{align*}
		\Var(\overline{H}_{*}([\![0,1]\!] \times T);[a,b]) \leq \vert b-a\vert \int_{K} \vert f(x) - x \vert \cdot (1 + \alpha(x)^{k}) \dd \Vert T \Vert(x).
	\end{align*}
	Furthermore, for $\mathscr{L}^{1}$-a.e.\ $t \in [0,1]$, the following mass estimate holds:
	\begin{align*}
		\Mbf(\overline{H}_{*}([\![0,1]\!] \times T)(t)) \leq \int_{K} 1 + \alpha(x)^{k} \dd \Vert T \Vert(x).
	\end{align*}
	\end{lemma}
	\begin{proof} Fix $\delta > 0$ with $(\supp T)_{\delta} \subset \interior K$, and consider the mollification $T_{\delta}$ of $T$ (at scale $\delta$). Recall that there is a smooth $\xi: \R^d \to \Wedge_k\R^d$ with $T_{\delta} = \xi\hspace{0.2em}\mathscr{L}^{d}$. It follows that $\Vert T_{\delta} \Vert = \Vert \xi \Vert \mathscr{L}^{d}$ and so the decomposability bundle satisfies $V(\Vert T_{\delta} \Vert,x) = \R^d$ for $\Vert T_{\delta} \Vert$-a.e.\ $x \in \R^d$ by~\cite[Proposition 2.9 (iii)]{AlbertiMarchese16} (which here just amounts to Rademacher's theorem). Fix $a,b \in [0,1]$ with $a < b$. We can now appeal to Proposition~\ref{pp:AffHomVarEst} to deduce that
	\begin{align*}
		\Var(\overline{H}_{*}([\![0,1]\!] \times T_{\delta});[a,b]) \leq \vert b - a \vert \int_{\R^d} \vert f(x) - x \vert \cdot (1 + \Vert Df(x) \Vert^{k}) \hspace{0.2em}\mathrm{d}\Vert T_{\delta} \Vert(x)
	\end{align*}
	and, for $\mathscr{L}^{1}$-a.e.\ $t \in [0,1]$,
	\begin{align*}
		\Mbf(\overline{H}_{*}([\![0,1]\!] \times T_{\delta})(t)) \leq \int_{\R^d} 1 + \Vert Df(x) \Vert^{k} \dd \Vert T_{\delta} \Vert(x).
	\end{align*}
	Since $\supp T_{\delta} \subset K$ and $\Vert Df(x) \Vert \leq \alpha(x)$ for $\mathscr{L}^{1}$-a.e.\ $x \in K$, we have 
	\begin{align*}
		\Var(\overline{H}_{*}([\![0,1]\!] \times T_{\delta});[a,b]) \leq \vert b-a\vert \int_{K} \vert f(x) - x \vert \cdot (1 + \alpha(x)^{k}) \dd \Vert T_{\delta} \Vert(x)
	\end{align*}
	and, for $\mathscr{L}^{1}$-a.e.\ $t \in [0,1]$,
	\begin{align}\label{STDeformVarEsts1}
		\Mbf(\overline{H}_{*}([\![0,1]\!] \times T_{\delta})(t)) \leq \int_{K} 1 + \alpha(x)^{k} \dd \Vert T_{\delta} \Vert(x).
	\end{align}
    Since $T_{\delta} \to T$ strictly as $\delta \to 0$ and the integrands in the previous two estimates are continuous, we have (see, for example,~\cite[Lemma 13.1]{Rindler26book})
	\begin{align*}
		\int_{K} \vert f(x) - x \vert \cdot (1 + \alpha(x)^{k}) \dd \Vert T_{\delta} \Vert(x) \to \int_{K} \vert f(x) - x \vert \cdot (1 + \alpha(x)^{k}) \dd \Vert T \Vert(x) 
	\end{align*}
	and 
	\begin{align}\label{STDeformVarEsts2}
		\int_{K} 1 + \alpha(x)^{k} \dd \Vert T_{\delta} \Vert(x) \to \int_{K} 1 + \alpha(x)^{k} \dd \Vert T \Vert(x)
	\end{align}
	as $\delta \to 0$. Furthermore, one can show that
	\begin{align*}
		\overline{H}_{*}([\![0,1]\!] \times T_{\delta}) \wkto \overline{H}_{*}([\![0,1]\!] \times T) \hspace{2em} \text{as $\delta \to 0$}.
	\end{align*}
    Since $\Mbf(T_{\delta}) \leq \Mbf(T)$ for $\delta > 0$ (see \cite[4.1.18]{Federer69book}), $(\Mbf(\overline{H}_{*}([\![0,1]\!] \times T_{\delta})))_{\delta}$ is uniformly bounded. The variation estimate then follows from Reshetnyak's lower semicontinuity theorem (see, for example,~\cite[Theorem 2.38]{AmbrosioFuscoPallara00book}). Finally, notice that 
    \begin{align*}
        \partial \overline{H}_{*}([\![0,1]\!] \times T_{\delta}) = \delta_{1} \times f_{*}T_{\delta} - \delta_{0} \times T_{\delta} - \overline{H}_{*}([\![0,1]\!] \times \partial T_{\delta}) \hspace{1em} \text{for $\delta > 0$}.
    \end{align*}
    Since we also have $\Mbf(\partial T_{\delta}) \leq \Mbf(\partial T)$ for $\delta > 0$, $(\Mbf(\partial \overline{H}_{*}([\![0,1]\!] \times T_{\delta})))_{\delta}$ is uniformly bounded and so we can appeal to Corollary \ref{cr:SliceWkCty} to deduce the mass estimate.
	\end{proof}
	
	\begin{proof}[Proof of Theorem \ref{tm:STDeform}] Using the homothety $\lambda_{\varepsilon}: \R^d \to \R^d$ given by $\lambda_{\varepsilon}(x) = \varepsilon x$ for $x \in \R^d$ and Proposition~\ref{pp:SpaceRescale} with the map $\mathrm{id}_{\R} \times \lambda_{\varepsilon}$ and its inverse, we can reduce to the case that $\varepsilon = 1$. If $k=d$, then we can use the constancy theorem (see~\cite[4.1.7]{Federer69book}) to deduce that $T = 0$ and there is nothing to show. Otherwise, define $\sigma: \R^d\backslash\Wrm''_{d-k-1} \to \Wrm'_{k}$, $u: \R^d \to \R$, $a \in \R^d$, $\tau_{a}: \R^d \to \R^d$, $v: \R^d \to \R$, and $V_{r}$ for $r > 0$ as in Subsection~\ref{s:Approx}. Define $P$ as in the proof of Proposition~\ref{pp:BoundarylessDeformation}, and notice that we can deduce that $P \in \Srm_{k}(\R^d)$, $\supp P \subset (\supp T)_{2d}$, and $\Mbf(P) \leq \gamma\Mbf(T)$ in exactly the same way as in Proposition~\ref{pp:BoundarylessDeformation}. To construct $S$, let $\overline{H}: \R^{1+d} \to \R^{1+d}$ where 
	\begin{align*}
		\overline{H}(t,x) = (t,(1-t)x + t\sigma(x+a)) \hspace{2em} \text{for $(t,x) \in \R^{1+d}$}.
	\end{align*}
    For $r > 0$, set $\langle T,v,r+ \rangle := -\partial(T \mvert V_{r})$; notice that this agrees with the definition of $\langle T,v,r+ \rangle$ given in~\cite[4.2.1]{Federer69book} since $\partial T = 0$. In~\cite[4.2.1]{Federer69book}, it is shown that
	\begin{align}\label{STDeformSliceSpt}
		\supp \langle T,v,r+ \rangle \subset v^{-1}(\{r\}) \cap \supp T \hspace{1em} \text{for $r > 0$}
	\end{align}
	and
	\begin{align*}
		\langle T,v,r+ \rangle \in \Nrm_{k-1}(\R^d) \hspace{1em} \text{for $\mathscr{L}^{1}$-a.e.\ $r > 0$.}
	\end{align*}
	Fix an $r > 0$ with $T \mvert V_{r} \in \Nrm_{k}(\R^d)$ and define
	\begin{align*}
		S_{r} := \overline{H}_{*}([\![0,1]\!] \times (T \mvert V_{r})).
	\end{align*}
	Since $\sigma \circ \tau_{a}$ and $\overline{H}$ are Lipschitz on $V_{r}$ and $[0,1] \times V_{r}$ respectively, we can use Proposition~\ref{pp:AffHomVarEst} to deduce that $S_{r} \in \Nlip_{1+k}([0,1] \times \R^d)$. From~\eqref{DeformEst5} and~\eqref{DeformEst2} (which only holds $\mathscr{L}^{d}$-a.e.), we deduce via Lemma~\ref{lm:STDeformVarEsts} that, for $a,b \in \R$ with $a < b$, 
	\begin{align}\label{STDeformVarEst}
		\Var(S_{r};[a,b]) \leq 2d^{1/2}\cdot \vert b - a \vert \int_{\R^{d}} 1 + \frac{1}{v(x)^{k}} \dd \Vert T \Vert(x)
	\end{align}
	and that
	\begin{align}\label{STDeformLInftyEst}
		\Vert S_{r} \Vert_{\Lrm^{\infty}} \leq \int_{\R^{d}} 1 + \frac{1}{v(x)^{k}} \dd \Vert T \Vert(x).
	\end{align}
	Now suppose that $k \geq 1$. Notice that
	\begin{align}\label{STDeformBdryFormula}
		\partial S_{r} = \delta_{1} \times (\sigma \circ \tau_{a})_{*}(T \mvert V_{r}) - \delta_{0} \times (T \mvert V_{r}) + \overline{H}_{*}([\![0,1]\!] \times \langle T,v,r+ \rangle).
	\end{align}
	We can also obtain similar estimates with $\langle T,v,r+ \rangle$ in place of $T \mvert V_{r}$; by applying Lemma~\ref{lm:STDeformVarEsts} and using~\eqref{DeformEst2},~\eqref{DeformEst5} and~\eqref{STDeformSliceSpt}, we obtain
	\begin{align}\label{STDeformBdryVarEst}
		\Var(\overline{H}_{*}([\![0,1]\!] \times \langle T,v,r+ \rangle);[a,b]) \leq 2d^{1/2}\cdot\vert b - a \vert \cdot \bigg(1 + \frac{1}{r^{k-1}}\bigg) \Mbf(\langle T,v,r+\rangle)
	\end{align}
	for $a,b \in \mathbb{R}$ with $a < b$, and
	\begin{align}\label{STDeformBdryLInftyEst}
		\bigl\Vert \overline{H}_{*}([\![0,1]\!] \times \langle T,v,r+ \rangle) \bigr\Vert_{\Lrm^{\infty}} \leq \bigg(1 + \frac{1}{r^{k-1}}\bigg)\Mbf(\langle T,v,r+ \rangle).
	\end{align}
	It is shown in~\cite[4.2.2]{Federer69book} that
	\begin{align}\label{STSlicesToZero}
		\liminf_{r \downarrow 0} r^{1-k}\cdot\Mbf(\langle T,v,r+ \rangle) = 0.
	\end{align}
	Take a sequence $(r_{j})_j \subset \R$ with $r_{j} \to 0$ and $r_{j}^{1-k}\cdot\Mbf(\langle T,v,r_{j}+ \rangle) \to 0$. We can use~\eqref{STDeformVarEst},\eqref{STDeformBdryFormula}, and~\eqref{STDeformBdryVarEst} to deduce that $\{\Lip(S_{r_{j}})\}_j$ is uniformly bounded. This, together with the $\Lrm^{\infty}$-estimates above, allows us to apply Proposition~\ref{pp:SpaceTimeCpct} to find a (not relabelled) subsequence of $(S_{r_{j}})_j$ and $S \in \Nlip_{1+k}([0,1] \times \R^d)$ with $S_{r_{j}} \wkto S$ and
	\begin{align}\label{STDeformLimVarEsts}
		\Var(S) &\leq 2d^{1/2}\int 1 + \frac{1}{v(x)^{k}} \dd \Vert T \Vert(x), & \Vert S \Vert_{\Lrm^{\infty}} &\leq \int 1 + \frac{1}{v(x)^{k}} \dd \Vert T \Vert(x). 
	\end{align}
	The desired variation and $\Lrm^{\infty}$-estimates now follow from~\eqref{DeformEst3}. Notice that we can use Lemma~\ref{lm:Pythagoras} together with~\eqref{DeformEst1},~\eqref{STDeformBdryVarEst},~\eqref{STDeformBdryLInftyEst}, and~\eqref{STSlicesToZero} to deduce that 
	\begin{align*}
		\Mbf(\overline{H}_{*}([\![0,1]\!] \times \langle T,v,r_{j}+ \rangle)) \to 0.
	\end{align*}
	Since $(\sigma \circ \tau_{a})_{*}(T \mvert V_{r}) \to P$ and $T \mvert V_{r} \to T$ strongly, we deduce~\eqref{STDeformBdryS} from~\eqref{STDeformBdryFormula}. If $k = 0$, notice that the estimates~\eqref{STDeformVarEst} and~\eqref{STDeformLInftyEst} allow us to find a sequence $(r_{j})_j$ with $r_{j} \to 0$ and a Lipschitz trajectory $S$ with $S_{r_{j}} \wkto S$ as above. Furthermore, we have 
	\begin{align}\label{STZeroDimBdry}
		\partial S_{r} = \delta_{1} \times (\sigma \circ \tau_{a})_{*}(T \mvert V_{r}) - \delta_{0} \times (T \mvert V_{r}). 
	\end{align}
	Since $(\sigma \circ \tau_{a})_{*}(T \mvert V_{r}) \to P$ and $T \mvert V_{r} \to T$ strongly, we deduce~\eqref{STDeformBdryS} from~\eqref{STZeroDimBdry}.
        
    In any case, notice that, for $r > 0$,
	\begin{align*}
		\supp S_{r} \subset [0,1] \times H([0,1] \times (\supp T \cap V_{r}))
	\end{align*}
	where $H: \R^{1+d} \to \R^d$ is the affine homotopy from $\mathrm{id}_{\R^d}$ to $\sigma \circ \tau_{a}$. We can then work as in the proof of Proposition~\ref{pp:BoundarylessDeformation} to deduce that $\supp S \subset [0,1] \times (\supp T)_{2d}$.
	\end{proof}

	\section{The Lipschitz Deformation Distance}\label{s:def_distance}
	
	Let $K \subset \R^d$ be compact and take $T_{0},T_{1} \in \Zrm_{k}(K)$. In analogy with~\cite[Section 5]{Rindler23}, we wish to measure the distance between $T_{0}$ and $T_{1}$ by considering the Lipschitz trajectories between them. To this end, we set:
	\begin{align*}
		\Trajlip_{1+k,K}(T_{0},T_{1}) := \{S \in \Nlip_{1+k}([0,1] \times K): \partial S = \delta_{1} \times T_{1} - \delta_{0} \times T_{0}\}.
	\end{align*}
	and define the \textbf{(Lipschitz) deformation distance} between $T_{0}$ and $T_{1}$ as
	\begin{align*}
		\mathrm{d}_{K}(T_{0},T_{1}) := \inf\{\Var(S): S \in \Trajlip_{1+k,K}(T_{0},T_{1}) \}.
	\end{align*}
	Notice that $\mathrm{d}_{K}(T_{0},T_{1}) = \infty$ if $\Trajlip_{1+k,K}(T_{0},T_{1}) = \varnothing$.
	
	\subsection{Bi-Lipschitz Equivalence \texorpdfstring{of $\mathrm{d}_{K}$ and $\Fbf_{K}^{\circ}$}{}} Let $K \subset \R^d$ be compact, let $T_{0},T_{1} \in \Zrm_{k}(K)$, and suppose that $\Trajlip_{1+k,K}(T_{0},T_{1}) \neq \varnothing$. For $S \in \Trajlip_{1+k,K}(T_{0},T_{1})$, we have $\partial (\pbf_{*}S) = T_{1} - T_{0}$ and so $T_{1} - T_{0} \in \Brm_{k}(K)$. We can then use~\eqref{pMassByVar} to see that
	\begin{align*}
		\Fbf_{K}^{\circ}(T_{1}-T_{0}) \leq \Mbf(\pbf_{*}S) \leq \Var(S).
	\end{align*}
	Taking an infimum yields
	\begin{align}\label{FlatHomByDist}
		\Fbf_{K}^{\circ}(T_{1}-T_{0}) \leq \mathrm{d}_{K}(T_{0},T_{1}).
	\end{align}
    It follows that, if $\mathrm{d}_{K}(T_{0},T_{1}) = 0$, then $T_{0} = T_{1}$. One can use this observation, together with the variation estimates for concatenations and reversals obtained in Section \ref{s:dynamical}, to show that $\mathrm{d}_{K}$ is an ($\R \cup \{+\infty\}$)-valued metric on $\Zrm_{k}(K)$.

	Notice that \eqref{FlatHomByDist} still holds when $\Fbf_{K}^{\circ}(T_{1}-T_{0}) = \infty$, since then $T_{1} - T_{0} \not\in \Brm_{k}(K)$ and so $\Trajlip_{1+k,K}(T_{0},T_{1}) = \varnothing$. We shall now investigate the other inequality:
	
	\begin{lemma}\label{lm:BigDistByFlatHom} Let $A$, $K \subset \R^d$ be compact with $A \subset \interior K$. Take $T_{0},T_{1} \in \Zrm_{k}(A)$. Then
	\begin{align*}
		\mathrm{d}_{K}(T_{0},T_{1}) \leq \Fbf_{A}^{\circ}(T_{1}-T_{0}).
	\end{align*}
	\end{lemma}
	
	\begin{proof} It is not restrictive to assume that $\Fbf_{A}^{\circ}(T_{1}-T_{0}) < \infty$. Choose an intermediate compact set $B \subset \interior K$ with $A \subset \interior B$. Fix $\varepsilon > 0$ with $A_{2d\varepsilon} \subset B$. We apply the Dynamical Deformation Theorem~\ref{tm:STDeform} (at scale $\varepsilon$) to $T_{0}$, $T_{1} \in \Zrm_{k}(A)$ to respectively obtain $P_{0}$, $P_{1} \in \Prm_{k}(\R^d)$ with $\supp P_{0}$, $\supp P_{1} \subset B$ and $S_{0},S_{1} \in \Nlip_{1+k}([0,1] \times B)$ with
    \begin{equation*}
			\partial S_{0} = \delta_{1} \times P_{0} - \delta_{0} \times T_{0} \quad \text{and} \quad 
    \partial S_{1} = \delta_{1} \times P_{1} - \delta_{0} \times T_{1} .
	\end{equation*}
	Furthermore, for $i \in \{0,1\}$, we have
	\begin{align*}
		\Mbf(P_{i}) &\leq \gamma \Mbf(T_{i}), & \Var(S_{i}) &\leq \gamma\varepsilon\Mbf(T_{i}),
	\end{align*}
	where $\gamma > 0$ depends only on $k$ and $d$. Note that
	\begin{align*}
		\Fbf_{B}^{\circ}(P_{1} - P_{0}) & \leq \Fbf_{B}^{\circ}((P_{1} - P_{0})-(T_1-T_0)) + \Fbf_{B}^{\circ}(T_{1} - T_{0}) \\ & \leq \Mbf(\pbf_{*}S_{1}) + \Mbf(\pbf_{*}S_{0}) + \Fbf_{A}^{\circ}(T_{1} - T_{0}) < \infty.
	\end{align*}
	Take $X \in \Nrm_{k+1}(B)$ with $\partial X = P_{1} - P_{0}$ with $\Mbf(X) \leq \Fbf_{B}^{\circ}(P_{1}-P_{0}) + \varepsilon$. We then apply Proposition~\ref{pp:StrictPolyEqlBdry} to obtain $Y \in \Prm_{k+1}(\R^d)$ with $\supp Y \subset K$, $\partial Y = P_{1} - P_{0}$, and $\Mbf(Y) \leq \Mbf(X) + \varepsilon$. We now apply Lemma~\ref{lm:Stretch} to $-Y$ to obtain $\tilde{V} \in \Prm_{1+k}(\R^{1+d})$ with $\tilde{V} \in \Nlip_{1+k}([0,1] \times K)$ and
	\begin{align*}
		\partial \tilde{V} &= - \delta_{0} \times (P_{0} - P_{1}), & \Var(\tilde{V}) &= \Mbf(Y).
	\end{align*}
	We then set $V = \tilde{V} + [\![0,1]\!] \times P_{1}$, so that
	\begin{align}\label{DistByFlatEst1}
		\partial V &= \delta_{1} \times P_{1} - \delta_{0} \times P_{0}, & \Var(V) &= \Mbf(Y).
	\end{align}
	Define $S := S_{1}^{-1} \circ V \circ S_{0}$ and note that $S \in \Nlip_{1+k}([0,1] \times K)$. We can use~\eqref{DistByFlatEst1}, \eqref{pMassByVar}, and the estimates above to deduce that
	\begin{align*}
		\Var(S) &\leq 2(\Var(S_{1}) + \Var(S_{0})) + \Fbf_{A}^{\circ}(T_{1}-T_{0}) + 2\varepsilon.
	\end{align*}
	The estimates from the Dynamical Deformation Theorem then yield
	\begin{align*}
		\Var(S) &\leq \Fbf_{A}^{\circ}(T_{1}-T_{0}) + 2\varepsilon(\gamma(\Mbf(T_{1}) + \Mbf(T_{0})) + 1).
	\end{align*}
	The result follows since we can take $\varepsilon > 0$ to be arbitrarily small.
	\end{proof}
	
	In the case that $K \subset \R^d$ is additionally a Lipschitz neighbourhood retract, we obtain bi-Lipschitz equivalence:
	
	\begin{proposition}\label{pp:DistByFlatHom} Let $K \subset \R^d$ be a compact Lipschitz neighbourhood retract and let $T_{0},T_{1} \in \Zrm_{k}(K)$. Then there is $C > 0$ (depending on $k$,$d$, and $K$) such that
	\begin{align}\label{DistByFlatHom}
		\mathrm{d}_{K}(T_{0},T_{1}) \leq C\Fbf_{K}^{\circ}(T_{1}-T_{0}).
	\end{align}
	\end{proposition}
	
	\begin{proof} Since $K$ is a compact Lipschitz neighbourhood retract, there is a bounded neighbourhood $U$ of $K$ and a Lipschitz retraction $r: U \to K$. Let $A \subset U$ be compact with $K \subset \interior A$ and take $\varepsilon > 0$. Then we can use Lemma~\ref{lm:BigDistByFlatHom} to find $\tilde{S} \in \Trajlip_{1+k,A}(T_{0},T_{1})$ with
	\begin{align*}
		\Var(\tilde{S}) &\leq \Fbf_{K}^{\circ}(T_{1} - T_{0}) + \varepsilon.
	\end{align*}
	Let $\overline{r} := \mathrm{id}_{\R} \times r$ and set $S := \overline{r}_{*}\tilde{S}$. We can use Proposition~\ref{pp:SpaceRescale} to deduce that
	\begin{align*}
		\Var(S) &\leq \Lip(r)^{k+1}(\Fbf_{K}^{\circ}(T_{1}-T_{0})+\varepsilon).
	\end{align*}
	The result follows since $\varepsilon > 0$ was arbitrary.
	\end{proof}
	
	\subsection{The Equality Theorem} Let $K \subset \R^d$ be a compact Lipschitz neighbourhood retract and consider the collection $\mathrm{Ret}(K)$ of Lipschitz retractions $r: U \to K$, where $U \subset \R^d$ is a bounded neighbourhood of $K$. Define
	\begin{align*}
		\rho := \inf\{\Lip(r): r \in \mathrm{Ret}(K)\}.
	\end{align*}
	The proof of Proposition~\ref{pp:DistByFlatHom} demonstrates that, for $T_{0},T_{1} \in \Zrm_{k}(K)$, 
	\begin{align*}
		\mathrm{d}_{K}(T_{0},T_{1}) \leq \rho^{k+1}\Fbf_{K}^{\circ}(T_{1}-T_{0}).
	\end{align*}
	Note that $\rho = 1$ if $K$ is convex (see~\cite[4.1.16]{Federer69book}). However, in general, $\rho$ may be strictly larger than $1$, even if $K = \overline{\Omega}$ for some Lipschitz domain $\Omega \subset \R^{d}$. To see this, define $\Omega$ as
	\begin{align*}
		\Omega = \{(x_{1},x_{2}) \in B(0,1): x_{2} < \vert x_{1} \vert\}.
	\end{align*}
    Let $U \subset \R^{2}$ be any bounded neighbourhood of $\overline{\Omega}$. We claim that there is no Lipschitz retraction $r: U \to \overline{\Omega}$ with Lipschitz constant $L := \Lip(r) < \sqrt{2}$. To see this, assume, for a contradiction, that there is such a retraction $r$. Take $t \in (0,1)$ with $(0,t) \in U$. Notice that
	\begin{align*}
		\vert r(0,t) - (t,t) \vert &= \vert r(0,t) - r(t,t) \vert \leq Lt, \\
		\vert r(0,t) - (-t,t) \vert &= \vert r(0,t) - r(-t,t) \vert \leq Lt, 
	\end{align*}
    and so $r(0,t)\in B((t,t),Lt) \cap B((-t,t),Lt) \cap \overline{\Omega}$. However, one can check that $B((t,t),Lt) \cap B((-t,t),Lt) \cap \overline{\Omega} = \varnothing$ since $L<\sqrt{2}$ (see also Figure \ref{fig:no_retraction}), which is a contradiction.
	\begin{figure}[h]
		\begin{tikzpicture}
			\draw[->, thick] (-4,0)--(4,0) node[right]{$x_{1}$};
			\draw[->, thick] (0,-1)--(0,4) node[above]{$x_{2}$};
			
			\draw[thick] (0,0)--(4,4);	
			\draw[thick] (0,0)--(-4,4) node[above]{$x_{2} = \vert x_{1} \vert$};
			
			\fill[blue, opacity=0.1] (2,2) circle (2.2cm);
			\fill[blue, opacity=0.1] (-2,2) circle (2.2cm);
			
			\fill[blue] (0,2) circle (2pt); 
			\fill[blue] (2,2) circle (2pt);
			
			\draw[<->, blue, thick] (0.2,2)--(1.8,2) node[midway,above]{$t$};
			\draw[<->, blue, thick] (0.3,0.1)--(2,1.8) node[midway,sloped,below]{$\sqrt{2}t$};
			\draw[<->, blue, thick] (2,2.2)--(3.36,3.56) node[midway,sloped,above]{$Lt$};
		\end{tikzpicture}
		\caption{$B((t,t),Lt) \cap B((-t,t),Lt) \cap \overline{\Omega} = \varnothing$}\label{fig:no_retraction}
	\end{figure}
	
	Nevertheless, we can use good directions to show that the homogeneous flat norm and the Lipschitz deformation distance agree on $\Omega$:
	
	\begin{theorem}[Equality Theorem]\label{tm:Equality} Let $\Omega \subset \R^d$ and suppose that $\Omega$ is a $\Crm^{0}$ domain. Take $T_{0},T_{1} \in \Zrm_{k}(\overline{\Omega})$. Then 
		\begin{align*}
			\Fbf_{\overline{\Omega}}^{\circ}(T_{0} - T_{1}) = \mathrm{d}_{\overline{\Omega}}(T_{0},T_{1}).
		\end{align*}
	\end{theorem}
	
	\begin{proof} It suffices to show that $\mathrm{d}_{\overline{\Omega}}(T_{0},T_{1}) \leq \Fbf^{\circ}_{\overline{\Omega}}(T_{1} - T_{0})$, and it is not restrictive to assume that $\Fbf^{\circ}_{\overline{\Omega}}(T_{1} - T_{0}) < \infty$. Fix $\varepsilon > 0$, and let $f_{\Omega,\varepsilon}: \overline{\Omega} \to \Omega$ and $H_{\Omega,\varepsilon}: [0,1] \times \overline{\Omega} \to \overline{\Omega}$ be the Lipschitz maps constructed in Proposition~\ref{pp:ContractInto}. Choose an intermediate compact set $A \subset \Omega$ with $f_{\Omega,\varepsilon}(\overline{\Omega}) \subset \interior A$. For $i \in \{0,1\}$, set $\tilde{T}_{i} := (f_{\Omega,\varepsilon})_{*}T_{i}$. One observes that
	\begin{align*}
		\Fbf^\circ_{A}(\tilde{T}_{1}-\tilde{T}_{0}) \leq \Lip(f_{\Omega,\varepsilon})^{k+1}\Fbf^\circ_{\overline{\Omega}}(T_{1}-T_{0}) < \infty.
	\end{align*}
	Since $\tilde{T}_{0},\tilde{T}_{1} \in \Zrm_{k}(A)$, we have $\mathrm{d}_{\overline{\Omega}}(\tilde{T}_{0},\tilde{T}_{1}) \leq \Fbf_{A}^{\circ}(\tilde{T}_{1} - \tilde{T}_{0})$ from Lemma~\ref{lm:BigDistByFlatHom} and thus we can find $\tilde{S} \in \Trajlip_{1+k,\overline{\Omega}}(\tilde{T}_{0},\tilde{T}_{1})$ with 
	\begin{align*}
		\Var(\tilde{S}) \leq \Lip(f_{\Omega,\varepsilon})^{k+1}\Fbf^\circ_{\overline{\Omega}}(T_{1}-T_{0}) + \varepsilon.
	\end{align*}
	Now let $\overline{H_{\Omega,\varepsilon}}: [0,1] \times \overline{\Omega} \to [0,1] \times \overline{\Omega}$ where $\overline{H_{\Omega,\varepsilon}}(t,x) = (t,H_{\Omega,\varepsilon}(t,x))$ for $(t,x) \in [0,1] \times \overline{\Omega}$ and, for $i \in \{0,1\}$, define 
	\begin{align*}
		\tilde{S}_{i} := (\overline{H_{\Omega,\varepsilon}})_{*}([\![0,1]\!] \times T_{i}).
	\end{align*}
	Since $\Vert f_{\Omega,\varepsilon} - \mathrm{id}_{\overline{\Omega}} \Vert_{\infty} < \varepsilon$, we can use Proposition~\ref{pp:AffHomVarEst} to deduce that, for $i \in \{0,1\}$, $\tilde{S}_{i} \in \Trajlip_{1+k,\overline{\Omega}}(T_{i},\tilde{T}_{i})$ and
	\begin{align*}
		\Var(\tilde{S}_{i}) \leq \varepsilon(1+\Lip(f_{\Omega,\varepsilon})^{k})\Mbf(T_{i}).
	\end{align*}
	Set $S := \tilde{S}_{1}^{-1} \circ \tilde{S} \circ \tilde{S}_{0}$. Since $\Lip(f_{\Omega,\varepsilon}) \leq 1 + \varepsilon$, we obtain
	\begin{align*}
		\Var(S) \leq (1+\varepsilon)^{k+1}\Fbf^\circ_{\overline{\Omega}}(T_{1}-T_{0}) + \varepsilon(1 + (1+\varepsilon)^{k})(\Mbf(T_{0})+\Mbf(T_{1})).
	\end{align*}
	The result follows since $\varepsilon > 0$ was arbitrary.
	\end{proof}
	
	\subsection{The Space-Time Connectivity Theorem}\label{s:STConnect} Finally, we give a proof of the \emph{Space-Time Connectivity Theorem} discussed in the introduction, where a simplified version was presented as Theorem \ref{tm:AbbrvSTConnect}:
    
	\begin{theorem}[Space-Time Connectivity Theorem]\label{tm:Equivalence} Let $K \subset \R^d$ be a compact Lipschitz neighbourhood retract. Let $(T_j)_j \subset \Zrm_{k}(K)$ and take $T \in \Zrm_{k}(K)$. Suppose that there is $M > 0$ with $\Mbf(T_{j}) \leq M$ for $j \in \N$ and $\Mbf(T) \leq M$. Then, 
	\begin{align}\label{DistVsWkStar}
		\mathrm{d}_{K}(T_{j},T) \to 0 \hspace{2em} \text{if and only if} \hspace{2em} T_{j} \wkto T \hspace{0.2em} \text{and $T - T_{j} \in \Brm_{k}(K)$ eventually}.
	\end{align}
    Furthermore, in this case, for all $j \in \N$ along a (not relabelled) subsequence, there is $S_{j} \in \Nlip_{1+k}([0,1] \times K)$ such that
	\begin{align*}
		\partial S_{j} &= \delta_{1} \times T - \delta_{0} \times T_{j}
	\end{align*}
	and
	\begin{align*}
		\Var(S_{j}) &\to 0, & \limsup_{j \to \infty} \Vert S_{j} \Vert_{\Lrm^{\infty}} \leq C \cdot \limsup_{j \to \infty} \Mbf(T_{j})
	\end{align*}
	where $C > 0$ depends only on $k$, $d$, and $K$.
	\end{theorem}
	
	In what follows, we shall see that the first conclusion~\eqref{DistVsWkStar} of Theorem \ref{tm:Equivalence} is a straightforward consequence of several results that we have already obtained. However, the second conclusion requires a more refined argument. To this end, let $K \subset \R^{d}$ be a compact Lipschitz neighbourhood retract. For $T \in \Nrm_{k}(K)$, we define:
    \begin{align*}
        \STF{T}:=\inf\{\Var(S)+\Vert S \Vert_{\Lrm^{\infty}}: S \in \Trajlip_{1+k,K}(T,0)\}.
    \end{align*}
    Notice that $\STF{T} = \infty$ if $\Trajlip_{1+k,K}(T,0) = \varnothing$. It follows from Proposition \ref{pp:DistByFlatHom} that $\STF{T} < \infty$ if $T \in \Brm_{k}(K)$, and one can show that $\STF{\cdot}$ is a norm on $\Brm_{k}(K)$.
	
	\begin{lemma}\label{lm:EquivDeform} Let $K \subset \R^d$ be a compact Lipschitz neighbourhood retract and fix $M > 0$. For $s > 0$, there is a finite set $\Gamma_{s} \subset \Brm_{k}(K)$ with the following property: For $T \in \Brm_{k}(K)$ with $\Mbf(T) \leq M$, there is $R \in \Gamma_{s}$ and $S \in \Trajlip_{1+k,K}(T,R)$ with 
	\begin{align}\label{EquivDeformEsts}
		\Var(S) &\leq s, & \Vert S \Vert_{\Lrm^{\infty}} &\leq C \cdot \Mbf(T) + s,
	\end{align}
	where $C$ depends only on $k$, $d$, and $K$.
	\end{lemma}
	
	\begin{proof} Fix $s > 0$. Since $K \subset \R^{d}$ is a compact Lipschitz neighbourhood retract, there is a bounded neighbourhood $U \subset \R^{d}$ of $K$ and a Lipschitz retraction $r: U \to K$. Fix $\varepsilon > 0$ with $K_{2d\varepsilon} \subset U$ (to be chosen later), and define
    \begin{align*}
        \Gamma := \{r_{*}P: P \in \Srm_{k,\varepsilon}(\R^{d}), \hspace{0.2em} \supp P \subset U, \hspace{0.2em} \Mbf(P) \leq \gamma M\} \cap \Brm_{k}(K),
    \end{align*}
    where $\gamma > 0$ is the constant from the Dynamical Deformation Theorem~\ref{tm:STDeform}. Notice that $\Gamma$ is a bounded subset of a finite dimensional space, so there is a finite set $\Gamma_{s} \subset \Gamma$ such that, for $T \in \Gamma$, there is $R \in \Gamma_{s}$ with 
    \begin{align*}
        \STF{R-T} \leq \varepsilon.
    \end{align*}
    Take $T \in \Brm_{k}(K)$. We apply the Dynamical Deformation Theorem to $T$ (at scale $\varepsilon$) to obtain $P \in \Srm_{k,\varepsilon}(\R^{d})$ with $r_{*}P \in \Gamma$ and $\tilde{S}_{1} \in \Nlip_{1+k}([0,1] \times \R^{d})$ with 
    \begin{align*}
        \partial \tilde{S}_{1} &= \delta_{1} \times P - \delta_{0} \times T, & \supp \tilde{S}_{1} &\subset [0,1] \times U.
    \end{align*}
    Since $r_{*}P \in \Gamma$, there is $R \in \Gamma_{s}$ with $\STF{R-r_{*}P} < \varepsilon$. It follows that there is $\tilde{S}_{2} \in \Trajlip_{1+k,K}(R-r_{*}P,0)$ with 
    \begin{align}\label{EquivDeformEst1}
        \Var(\tilde{S}_{2}) &\leq 2\varepsilon, & \Vert \tilde{S}_{2} \Vert_{\Lrm^{\infty}} &\leq 2\varepsilon.
    \end{align}
    Let $\overline{r} := \mathrm{id}_{\R} \times r$. Define 
    \begin{align*}
        S_{1} &:= \overline{r}_{*}\tilde{S}_{1}, & S_{2} &:= \tilde{S}_{2} + [\![0,1]\!] \times r_{*}P.
    \end{align*}
    Set $S := S_{2}^{-1} \circ S_{1}$ and notice that $S \in \Trajlip_{1+k,K}(T,R)$. We can use \eqref{EquivDeformEst1}, Proposition \ref{pp:SpaceRescale} and the estimates from the Dynamical Deformation Theorem to deduce that
    \begin{align*}
        \Var(S) &=  \varepsilon(2+\gamma M\cdot\Lip(r)^{k+1}) & \Vert S \Vert_{\Lrm^{\infty}} &\leq \gamma\cdot \Lip(r)^{k}\cdot\Mbf(T)+2\varepsilon.
    \end{align*}
    The result follows with $C = \gamma\cdot\Lip(r)^{k}$ after choosing $\varepsilon > 0$ small enough that \eqref{EquivDeformEsts} holds.
	\end{proof}
	
	\begin{proof}[Proof of Theorem~\ref{tm:Equivalence}] Notice that $T_{j} \wkto T$ if and only if $\Fbf_{K}(T_{j} - T) \to 0$ (see, for example,~\cite[Theorem 7.1]{FedererFleming60}), and we have $\mathrm{d}_{K}(T_{j},T) \to 0$ if and only if $\Fbf_{K}^{\circ}(T_{j}-T) \to 0$ by~\eqref{FlatHomByDist} and~\eqref{DistByFlatHom}. We can therefore conclude~\eqref{DistVsWkStar} provided that $\Fbf_{K}^{\circ}(T_{j}-T) \to 0$ if and only if $\Fbf_{K}(T_{j}-T) \to 0$ and $T_{j}-T \in \Brm_{k}(K)$ eventually; this follows from the fact that $\Fbf_{K} \leq \Fbf_{K}^{\circ}$ (for one direction) and Corollary~\ref{cr:HomFlatVsFlat} (for the other direction).
		
	For the second conclusion, notice that it is not restrictive to assume that $(T-T_{j})_j \subset \Brm_{k}(K)$.Let $\Gamma_{1/2} \subset \Brm_{k}(K)$ be the set obtained in Lemma \ref{lm:EquivDeform} for $s=1/2$. Since $\Gamma_{1/2}$ is finite, there is a $R_{1} \in \Brm_{k}(K)$ (not depending on $j$) and a (not relabelled) subsequence of $(T_{j}-T)_{j}$ such that for every $j$ there is $V_{1,j} \in \Trajlip_{1+k,K}(T-T_{j},R_1)$ with
	\begin{align*}
		\Var(V_{1,j}) &\leq \frac12, & \Vert V_{1,j} \Vert_{\Lrm^{\infty}} &\leq C \cdot \Mbf(T-T_{j}) + \frac12,
	\end{align*}
	where $C > 0$ depends only on $k$, $d$, and $K$. We then successively apply the same argument with $s=2^{-n}$ for $n \in \N$ and select a diagonal subsequence to obtain a (not relabelled) subsequence of $(T_{j}-T)_{j}$ and a sequence $(R_{n})_{n} \subset \Brm_{k}(K)$ such that, for $n,j \in \N$ with $j \geq n$, there is $V_{n,j} \in \Trajlip_{1+k,K}(T-T_{j},R_{n})$ with
	\begin{align*}
		\Var(V_{n,j}) &\leq 2^{-n}, & \Vert V_{n,j} \Vert_{\Lrm^{\infty}} &\leq C \cdot \Mbf(T-T_{j}) + 2^{-n}.
	\end{align*}
    For $j \in \N$, define
	\begin{align*}
		\tilde{S}_{j} := [\![0,1]\!] \times T - V_{j,j+1}^{-1} \circ V_{j,j}.
	\end{align*}
	Notice that $\tilde{S}_{j} \in \Trajlip_{1+k,K}(T_{j},T_{j+1})$. We can estimate
	\begin{align*}
		\Var(\tilde{S}_{j}) &\leq 2^{-j+1}, & \Vert \tilde{S}_{j} \Vert_{\Lrm^{\infty}} &\leq (C+1)\Mbf(T) + C\cdot\max\{\Mbf(T_{j}),\Mbf(T_{j+1})\} + 2^{-j}.
	\end{align*}
	Fix $j \in \N$. For $m \in \N$, we define $S^{m}_{j} := \tilde{S}_{j+m-1} \circ \cdots \circ \tilde{S}_{j}$. Notice that, for $m \in \N$, we have
	\begin{align*}
		\Var(S^{m}_{j}) &= \sum_{i=0}^{m-1} \Var(\tilde{S}_{j+i}) \leq 2^{-j+1} \sum_{i=0}^{m-1}2^{-i} \leq 2^{-j+2}
	\end{align*}
    and
    \begin{align*}
        \Vert S^{m}_{j} \Vert_{\Lrm^{\infty}} \leq (C+1)\Mbf(T) + C \cdot \sup_{i \geq j} \Mbf(T_{i}) + 2^{-j}.
    \end{align*}
    Furthermore, we have 
	\begin{align*}
		\Var(\partial S^{m}_{j}) &\leq \Mbf(T_{j}) + \Mbf(T_{j+m}) \leq 2M, & \Vert \partial S^{m}_{j} \Vert_{\Lrm^{\infty}} &= 0.
	\end{align*}
	Notice that we can also apply Lemma~\ref{lm:Rescale} to ensure that, without loss of generality, $(\Lip(S^{m}_{j}))_{m}$ is uniformly bounded. We can therefore use Proposition~\ref{pp:SpaceTimeCpct} to pass to the limit $m \to \infty$; this yields 
	\begin{align*}
		S_{j} \in \Trajlip_{1+k,K}(T_{j},T)
	\end{align*}
	with 
	\begin{align*}
		\Var(S_{j}) &\leq 2^{-j+2}, & \Vert S_{j} \Vert_{\Lrm^{\infty}} \leq (C+1)\Mbf(T) + C \cdot \sup_{i \geq j} \Mbf(T_{i}) + 2^{-j}.
	\end{align*}
	Since $T_{j} \wkto T$, we have
	\begin{align}\label{EquivMassEst}
		\Vert S_{j} \Vert_{\Lrm^{\infty}} \leq (C+1)\liminf_{i \to \infty}\Mbf(T_{i}) + C \cdot \sup_{i \geq j} \Mbf(T_{i}) + 2^{-j}.
	\end{align}
	Taking the upper limit of~\eqref{EquivMassEst} and adjusting $C$ accordingly yields the result.
	\end{proof}

    \begin{remark} It is worth highlighting the statement of Theorem \ref{tm:Equality} in the case $k=0$, where it applies to finite signed measures (which are identified with normal $0$-currents under an assumption of finite mass). Under a uniform mass bound, the weak* convergence of measures is equivalent to the convergence in the $p$-Wasserstein metric for $p \in[1,\infty)$ (see, for example, \cite[Proposition 7.1.5]{AmbrosioGigliSavare05book}). In turn, when $p=1$, the Wasserstein metric is the same as the flat norm. Therefore, the Equality Theorem \ref{tm:Equality} in this setting yields a dynamical characterisation of the Wasserstein distance, proving that it is equivalent to the convergence in the (Lipschitz) deformation distance. By Theorem~\ref{tm:Equivalence} we may also require for the slices (that is, the intermediate measures between the two endpoints) to have uniformly bounded mass.
    \end{remark}

	\appendix
	
	\section{Proof of Proposition~\ref{pp:InwardPointingVecFld}}\label{ax:BZ}
	
	Although Proposition~\ref{pp:InwardPointingVecFld} is proven in~\cite{BallZarnescu17}, sections of the proof can be generalised (see Lemma~\ref{lm:ParamsAreCts}) and for the reader's convenience we give extra details in our notation. We start with two preliminary results, the first of which will allow us to work at $0 \in \R^d$ without loss of generality:
	
	\begin{lemma}\label{lm:Translate} Let $\Omega \subset \R^d$ be open and let $\delta > 0$. Let $x_{0} \in \R^d$ with $\partial\Omega \cap B(x_{0},\delta) \neq \varnothing$ and let $v \in \mathbb{S}^{d-1}$ be a good direction for $\Omega$ at $x_{0}$ at scale $\delta$. Fix $a \in \R^d$. Then $v$ is a good direction for $\Omega - a$ at $x_{0}-a$ at scale $\delta$.
	\end{lemma}
	
	\begin{proof} Take a continuous $g: \Pi(x_{0},v) \to \R$ which parameterises $\partial\Omega \cap B(x_{0},\delta)$ (see Definition~\ref{dfn:GoodDirections}). It follows from~\eqref{OrthProjDefns} that, for $x \in \R^d$,
	\begin{align}\label{OrthProjTranslate}
		\pi_{x_{0}-a,v}(x) = \pi_{x_{0},v}(x+a)-a \hspace{1em} \text{and} \hspace{1em} \pi^{\perp}_{x_{0}-a,v}(x) = \pi^{\perp}_{x_{0},v}(x+a)-a.
	\end{align}
	Define $\tilde{g}: \Pi(x_{0}-a,v) \to \R$ where $\tilde{g}(y)=g(y+a)$ for $y \in \Pi(x_{0}-a,v)$. It follows from~\eqref{OrthProjTranslate} that, for $x \in \R^d$,
	\begin{align}\label{Translate1}
		\tilde{g}(\pi_{x_{0}-a,v}(x)) = g(\pi_{x_{0},v}(x+a)).
	\end{align}
	On the other hand, for $x \in \R^{d}$, we have
    \begin{align*}
        \pi^{\perp}_{x_{0},v}(x+a) &= x_{0}+\phi^{-1}_{x_{0},v}(\pi^{\perp}_{x_{0},v}(x+a))v, \\
        \pi^{\perp}_{x_{0}-a,v}(x) &= x_{0}-a+\phi^{-1}_{x_{0}-a,v}(\pi^{\perp}_{x_{0}-a,v}(x))v,
    \end{align*}
    and so we can use~\eqref{OrthProjTranslate} to deduce that, for $x \in \R^{d}$,
	\begin{align}\label{Translate2}
		\phi^{-1}_{x_{0},v}(\pi^{\perp}_{x_{0},v}(x+a))=\phi^{-1}_{x_{0}-a,v}(\pi^{\perp}_{x_{0}-a,v}(x)).
	\end{align}
	Combining~\eqref{Translate1} and~\eqref{Translate2} and using the fact that $v$ is a good direction for $\Omega$ at $x_{0}$ at scale $\delta$ yields
	\begin{align*}
		x \in (\Omega - a) \cap B(x_{0}-a,\delta) \hspace{1em} &\iff \hspace{1em} g(\pi_{x_{0}-a,v}(x)) < \phi^{-1}_{x_{0}-a,v}(\pi^{\perp}_{x_{0}-a,v}(x)), \\
		x \in \partial(\Omega - a) \cap B(x_{0}-a,\delta) \hspace{1em} &\iff \hspace{1em} g(\pi_{x_{0}-a,v}(x)) = \phi^{-1}_{x_{0}-a,v}(\pi^{\perp}_{x_{0}-a,v}(x)).
	\end{align*}
	The result follows.
	\end{proof}
	
	The next result demonstrates that the continuity requirement for the ``parameterisation'' function of Definition~\ref{dfn:GoodDirections} is superfluous at points where the function is a parameterisation:
	
	\begin{lemma}\label{lm:ParamsAreCts} Let $\Omega \subset \R^d$ be open and let $\delta > 0$. Let $x_{0} \in \R^d$ with $\partial \Omega \cap B(x_{0},\delta) \neq \varnothing$. Take $v \in \mathbb{S}^{d-1}$ and let $g: \Pi(x_{0},v) \to \R$. Suppose that $g$ ``parameterises $\partial\Omega \cap B(x_{0},\delta)$'', that is,
	\begin{align*}
		\Omega \cap B(x_{0},\delta) &= \{x \in B(x_{0},\delta): g(\pi_{x_{0},v}(x)) < \phi^{-1}_{x_{0},v}(\pi^{\perp}_{x_{0},v}(x))\}, \\
		\partial \Omega \cap B(x_{0},\delta) &= \{x \in B(x_{0},\delta): g(\pi_{x_{0},v}(x)) = \phi^{-1}_{x_{0},v}(\pi^{\perp}_{x_{0},v}(x))\}.
	\end{align*}
	Furthermore, suppose that
	\begin{align}\label{ParamInBall}
		y+g(y)v \in \overline{B}(x_{0},\delta) \hspace{2em} \text{for $y \in \Pi(x_{0},v) \cap B(x_{0},\delta)$},
	\end{align}
	and suppose that $g(y) = 0$ for $y \in \Pi(x_{0},v)$ with $\vert y -x_0\vert \geq \delta$. Then $g$ is continuous.
	\end{lemma}
	
	\begin{proof} First note that it suffices to consider the case $x_{0} = 0$ by Lemma~\ref{lm:Translate}. Since $g$ is continuous on $\Pi(0,v) \cap \R^d\backslash B(0,\delta)$, it suffices to show that $g$ is continuous on $\Pi(0,v) \cap \overline{B}(0,\delta)$. Take $\{y_{j}\} \subset \Pi(0,v) \cap \overline{B}(0,\delta)$ with $y_{j} \to y$ for some $y \in \Pi(0,v) \cap \overline{B}(0,\delta)$.
		
	\textbf{Case 1.} Suppose that $\vert y \vert = \delta$. Note that $\vert y_{j} \vert^{2} + \vert g(y_{j}) \vert^{2} \leq \delta^2$ for $j \in \N$ by~\eqref{ParamInBall}. Since $\vert y_{j} \vert^{2} \to \delta^{2}$, we must have $g(y_{j}) \to 0 = g(y)$.
		
	\textbf{Case 2.} Suppose that $y + g(y)v \in B(0,\delta)$. Then there is $r > 0$ such that, for $\varepsilon \in (0,r)$,
	\begin{align*}
		y + (g(y) + \varepsilon)v &\in \Omega \cap B(0,\delta), \\
		y + (g(y) - \varepsilon)v &\in \R^d\backslash\overline{\Omega} \cap B(0,\delta).
	\end{align*}
	Fix $\varepsilon \in (0,r)$ and take $\eta > 0$ with 
	\begin{align*}
		B(y + (g(y) + \varepsilon)v,\eta) &\subset \Omega \cap B(0,\delta), \\
		B(y + (g(y) - \varepsilon)v,\eta) &\subset \R^d\backslash\overline{\Omega} \cap B(0,\delta).
	\end{align*}
	Since $y_{j} \to y$, there is $N \in \N$ with $\vert y_{j} - y \vert < \eta$ for $j > N$. It follows that, for $j > N$, 
	\begin{align*}
		y_{j} + (g(y)+\varepsilon)v &\in \Omega \cap B(0,\delta), \\
		y_{j} + (g(y)-\varepsilon)v &\in \R^d\backslash\overline{\Omega} \cap B(0,\delta).
	\end{align*}
	Since $g$ parameterises $\partial\Omega \cap B(0,\delta)$, we must have $g(y)-\varepsilon < g(y_{j}) < g(y)+\varepsilon$ for $j > N$.
		
	\textbf{Case 3.} Suppose that $y + g(y)v \in \partial B(0,\delta)$. If $g(y) = 0$, then $\vert y \vert = \delta$ and we can appeal to Case 1. We shall treat the case $g(y) > 0$; the case $g(y) < 0$ is analogous. In this case, we have $g(y) = (\delta^{2}-\vert y \vert^{2})^{1/2}$. Notice that there is $r > 0$ such that, for $\varepsilon \in (0,r)$,
	\begin{align*}
		y + (g(y) - \varepsilon)v \in \R^d\backslash\overline{\Omega} \cap B(0,\delta).
	\end{align*}
	Fix $\varepsilon \in (0,r)$ and take $\eta > 0$ with 
	\begin{align*}
		B(y + (g(y) - \varepsilon)v,\eta) &\subset \R^d\backslash\overline{\Omega} \cap B(0,\delta).
	\end{align*}
	Since $y_{j} \to y$, there is $N \in \N$ with $\vert y_{j} - y \vert < \eta$ for $j > N$. It follows that, for $j > N$, 
	\begin{align*}
		y_{j} + (g(y)-\varepsilon)v \in \R^d \backslash \overline{\Omega} \cap B(0,\delta).
	\end{align*}
	Since $g$ parameterises $\partial\Omega \cap B(0,\delta)$, we have $g(y_{j}) > g(y) - \varepsilon$. We also have $g(y_{j}) \leq (\delta^{2} - \vert y_{j} \vert^{2})^{1/2}$ since \eqref{ParamInBall} holds. Taking a lower and upper limit respectively yields
	\begin{align*}
		g(y) - \varepsilon \leq \liminf_{j \to \infty} g(y_{j}) \leq \limsup_{j \to \infty} g(y_{j}) \leq g(y).
	\end{align*}
	We deduce that $g(y_{j}) \to g(y)$ since $\varepsilon \in (0,r)$ was arbitrary.
	\end{proof}
	
	Let $\Omega \subset \R^d$ be open and let $\delta > 0$. Let $x_{0} \in \R^d$ with $\partial\Omega \cap B(x_{0},\delta) \neq \varnothing$ and suppose that $v \in \mathbb{S}^{d-1}$ is a good direction for $\Omega$ at $x_{0}$ at scale $\delta > 0$. A similar argument to that of Lemma \ref{lm:Translate} demonstrates that, for $x \in B(x_{0},\delta/2)$ with $\partial \Omega \cap B(x_{0},\delta/2) \neq \varnothing$, $v$ is a good direction for $\Omega$ at $x$ at scale $\delta/2$. It follows that we can construct fields of good directions locally. To prove Proposition~\ref{pp:InwardPointingVecFld}, we must piece these fields together, which we can do because of the following result:
	
	\begin{lemma} Let $\Omega \subset \R^d$ be open and let $\delta > 0$. Let $x_{0} \in \R^d$ with $\partial \Omega \cap B(x_{0},\delta) \neq \varnothing$ and let $v_{1}$, $v_{2} \in \mathbb{S}^{d-1}$ be good directions for $\Omega$ at $x_{0}$ at scale $\delta$. Then $v_{1} \neq -v_{2}$ and, for $\lambda \in (0,1)$, $w \in \R^d$, where
	\begin{align*}
		w = \frac{(1-\lambda)v_{1}+\lambda v_{2}}{\vert (1-\lambda)v_{1} + \lambda v_{2} \vert},
	\end{align*}
	is a good direction for $\Omega$ at $x_{0}$ at scale $\delta > 0$.
	\end{lemma}
	
	\begin{proof} First note that it suffices to consider the case $x_{0} = 0$ by Lemma~\ref{lm:Translate}. For the first statement, suppose that $v_{1} = -v_{2}$. Let $x \in \partial\Omega \cap B(0,\delta)$ and take $t > 0$ with $x + tv_{1} \in B(0,\delta)$. Since $v_{1}$ is a good direction for $\Omega$ at $x_{0}$ at scale $\delta$, we can use Lemma \ref{lm:GoodDirectionsAreInward} to deduce that $x + tv_{1} \in \Omega$. On the other hand, since $-v_{1}$ is a good direction for $\Omega$ at $x_{0}$ at scale $\delta > 0$, a similar argument shows that $x + tv_{1} \in \R^{d}\backslash\overline{\Omega}$ -- a contradiction. For the second statement, fix $\lambda \in (0,1)$ and set $\alpha = \vert (1-\lambda)v_{1}+\lambda v_{2} \vert$. 
		\begin{claim} Let $x \in \overline{\Omega} \cap B(0,\delta)$ and take $s \in (0,\alpha\cdot\dist(x,\partial B(0,\delta)))$. Then $x+sw \in \Omega \cap B(0,\delta)$.
		\end{claim}
		\begin{proof}[Proof of Claim] First notice that
			\begin{align*}
				x + s\cdot\frac{(1-\lambda)}{\alpha}v_{1} \in B(0,\delta).
			\end{align*}
			Since $(1-\lambda)(s/\alpha) > 0$, we can use Lemma~\ref{lm:GoodDirectionsAreInward} to deduce that
			\begin{align*}
				x + s\cdot\frac{(1-\lambda)}{\alpha}v_{1} \in \Omega.
			\end{align*}
			Now, since $\alpha \leq 1$, we have 
			\begin{align*}
				x + s\cdot\frac{(1-\lambda)}{\alpha}v_{1} + s\cdot\frac{\lambda}{\alpha}v_{2} = x + sw \in B(0,\delta).
			\end{align*}
			We can therefore use Lemma~\ref{lm:GoodDirectionsAreInward} once more to deduce the result.
		\end{proof}
		We now define
		\begin{align*}
			G := \{y \in \Pi(0,w) \cap B(0,\delta): \partial\Omega \cap \Pi^{\perp}(y,w) \cap B(0,\delta) \neq \varnothing\}.
		\end{align*}
		\begin{claim} Take $y \in G$. Then there is $t(y) \in \R$ with 
			\begin{align*}
				\partial\Omega \cap \Pi^{\perp}(y,w) \cap B(0,\delta) &= \{y+t(y)w\}, \\ 
				\Omega \cap \Pi^{\perp}(y,w) \cap B(0,\delta) &= \{y+tw: t > t(y)\} \cap B(0,\delta), \\ 
				\R^d\backslash\overline{\Omega} \cap \Pi^{\perp}(y,w) \cap B(0,\delta) &= \{y+tw: t < t(y)\} \cap B(0,\delta).
			\end{align*}
		\end{claim}
		\begin{proof} It suffices to show that 
			\begin{align*}
				\{y+t(y)w\} &\subset \partial\Omega \cap \Pi^{\perp}(y,w) \cap B(0,\delta), \\ 
				\{y+tw: t > t(y)\} \cap B(0,\delta) &\subset \Omega \cap \Pi^{\perp}(y,w) \cap B(0,\delta), \\ 
				\{y+tw: t < t(y)\} \cap B(0,\delta) &\subset \R^d\backslash\overline{\Omega} \cap \Pi^{\perp}(x,w) \cap B(0,\delta). 
			\end{align*}
			Since $y \in G$, there is $t(y) \in \R$ with $y + t(y)w \in \partial\Omega \cap \Pi^{\perp}(y,w) \cap B(0,\delta)$. For the second inclusion, define $I \subset \R$ where
			\begin{align*}
				I := \{t > t(y): y + tw \in B(0,\delta) \cap \R^d\backslash\Omega\}.
			\end{align*}
			Suppose, for a contradiction, that $I \neq \varnothing$. Set $t_{*} = \inf I$. Notice that $y+(t(y)+t)w \in \Omega$ for small $t > 0$ by the previous claim, so we have $t_{*} > t(y)$. Fix $\varepsilon > 0$ with $\varepsilon < t_{*}-t(y)$ and $\varepsilon < \alpha \cdot \dist(y+tw,\partial B(0,\delta))$ for $t \in (t(y),t_{*})$. There is $t \in I$ with $t < t_{*} + \varepsilon$ and so $t - \varepsilon \in (t(y),t_{*})$. Since $B(0,\delta)$ is convex, we must have $y + (t-\varepsilon)w \in B(0,\delta)$ and so, since $t-\varepsilon \not\in I$, we must have $y + (t-\varepsilon)w \in \Omega$. We can therefore appeal to the previous claim to yield $y + tw \in \Omega \cap B(0,\delta)$, which contradicts $t \in I$. We deduce that $I = \varnothing$. The third inclusion is similar.
		\end{proof}
		Let $g: \Pi(0,w) \to \R$ where
		\begin{align*}
			g(y) &= \left\{\begin{array}{cl}
				0 & \text{if $\vert y \vert \geq \delta$} \\
				t(y) & \text{if $y \in G$} \\
				-(\delta^{2}-\vert y \vert^{2})^{1/2} & \text{if $\Pi^{\perp}(y,w) \cap B(0,\delta) \subset \Omega$} \\
				(\delta^{2}-\vert y \vert^{2})^{1/2} & \text{if $\Pi^{\perp}(y,w) \cap B(0,\delta) \subset \R^d\backslash\overline{\Omega}$}.
			\end{array}\right.
		\end{align*}
		Take $y \in \Pi(0,w) \cap B(0,\delta)$. As the intersection of convex sets, $\Pi^{\perp}(y,w) \cap B(0,\delta)$ is connected and so, if $y \not\in G$, we must have either $\Pi^{\perp}(y,w) \cap B(0,\delta) \subset \Omega$ or $\Pi^{\perp}(y,w) \cap B(0,\delta) \subset \R^{d}\backslash\overline{\Omega}$. It follows that $g$ is well-defined. Also note that
		\begin{align}\label{FuncInBall}
			y+g(y)w \in \overline{B}(0,\delta) \hspace{2em} \text{for $y \in \Pi(0,w) \cap B(0,\delta)$}.
		\end{align}
		It therefore suffices (by Lemma~\ref{lm:ParamsAreCts}) to show that $g$ parameterises $\partial\Omega \cap B(0,\delta)$. To this end, take $x \in B(0,\delta)$. If $\pi_{w}(x) \in G$, then we can use the second claim above to deduce that
		\begin{align*}
			x \in \Omega \hspace{1em} &\iff \hspace{1em} g(\pi_{w}(x)) < \phi^{-1}_{0,w}(\pi^{\perp}_{w}(x)), \\
			x \in \partial\Omega \hspace{1em} &\iff \hspace{1em} g(\pi_{w}(x)) = \phi^{-1}_{0,w}(\pi^{\perp}_{w}(x)).
		\end{align*}
		If $\Pi^{\perp}(\pi_{w}(x),w) \cap B(0,\delta) \subset \Omega$, then we have $x \in \Omega$; since $x \in B(0,\delta)$ we have 
		\begin{align}\label{PointInBall}
			\vert \pi_{w}(x) \vert^{2} + \vert \phi^{-1}_{0,w}(\pi_{w}^{\perp}(x)) \vert^{2} \leq \delta^{2},
		\end{align}
		and so
		\begin{align*}
			g(\pi_{w}(x)) = - (\delta^{2} - \vert \pi_{w}(x) \vert^{2})^{1/2} < \phi^{-1}_{0,w}(\pi_{w}^{\perp}(x)).
		\end{align*}
		Similarly, if $\Pi^{\perp}(\pi_{w}(x),w) \cap B(0,\delta) \subset \R^d\backslash\overline{\Omega}$, then $x \in \R^d\backslash \overline{\Omega}$ and we can use~\eqref{PointInBall} to deduce that 
		\begin{align*}
			g(\pi_{w}(x)) = (\delta^{2} - \vert \pi_{w}(x) \vert^{2})^{1/2} > \phi^{-1}_{0,w}(\pi_{w}^{\perp}(x)).
		\end{align*}
		The result follows.
	\end{proof}
	
	One can then use induction to establish a similar result for any finite set of good directions:
	
	\begin{corollary}\label{cr:GDsAreGeoConvex} Let $\Omega \subset \R^d$ be open and let $\delta > 0$. Let $x_{0} \in \R^d$ with $\partial\Omega \cap B(x_{0},\delta) \neq \varnothing$. Let $n \in \N$ and take a set $\{v_{i}\}_{i=1}^{n} \subset \mathbb{S}^{d-1}$ of good directions for $\Omega$ at $x_{0}$ at scale $\delta > 0$. Let $\{\lambda_{i}\}_{i=1}^{n} \subset [0,1]$ and suppose that $\lambda_{1} + \cdots + \lambda_{n} = 1$. Then $\lambda_{1}v_{1} + \cdots + \lambda_{n}v_{n} \neq 0$ and
		\begin{align*}
			w := \frac{\lambda_{1}v_{1} + \cdots + \lambda_{n}v_{n}}{\vert \lambda_{1}v_{1} + \cdots + \lambda_{n}v_{n}\vert}
		\end{align*}
		is a good direction for $\Omega$ at $x_{0}$ at scale $\delta > 0$.
	\end{corollary}
	
	We are now in a position to prove Proposition~\ref{pp:InwardPointingVecFld}.
	
	\begin{proof}[Proof of Proposition~\ref{pp:InwardPointingVecFld}] Fix $x \in \partial\Omega$. Since $\Omega$ has a $\Crm^{0}$ boundary, there is a good direction $v(x) \in \mathbb{S}^{d-1}$ to $\Omega$ at $x$ at some scale $\delta(x) > 0$. Since $\Omega$ is bounded, we can find $n \in \N$, $\{x_{i}\}_{i=1}^{n} \subset \partial\Omega$, and a $\delta > 0$ with 
		\begin{align*}
			(\partial \Omega)_{\delta} \subset \bigcup_{i=1}^{n}B\bigg(x_{i},\frac{\delta(x_{i})}{2}\bigg).
		\end{align*}
		As we have observed above, for $i \in \N$ with $1 \leq i \leq n$ and $x \in B(x_{i},\delta(x_{i})/2)$, $v_{i}:=v(x_i)$ is a good direction for $\Omega$ at $x$ at scale $\delta(x_{i})/2$. Let $\{\varphi_{i}\}_{i=1}^{n}$ be a smooth partition of unity for $(\partial\Omega)_{\delta}$ subordinate to the cover $\{B(x_{i},\delta(x_{i}/2)\}_{i=1}^{n}$. Let $\tilde{n}: (\partial\Omega)_{\delta} \to \mathbb{S}^{d-1}$ where, for $x \in (\partial\Omega)_{\delta}$,
		\begin{align*}
			\tilde{n}(x) = \frac{\varphi_{1}(x)v_{1} + \cdots + \varphi_{n}(x)v_{n}}{\vert \varphi_{1}(x)v_{1} + \cdots + \varphi_{n}(x)v_{n}\vert}.
		\end{align*}
		Note that we ensure that $\tilde{n}$ is Lipschitz by possibly making $\delta$ smaller. The result follows from Corollary~\ref{cr:GDsAreGeoConvex}.
	\end{proof}
	
		\subsection*{Data Availability Statement} No data were used to support this research.

	\subsection*{Conflict of Interest and Funding Statement}
	
	This project has received funding from the UKRI Frontier Research Guarantee (ERC guarantee) grant EP/Z000297/1 (ERC CONCENTRATE). H.T.\ was supported by the Warwick Mathematics Institute Centre for Doctoral Training and the EPSRC grant EP/W524645/1. The authors have no relevant financial or non-financial interests to disclose.

	\bibliographystyle{plain}
	
\end{document}